\documentclass[11pt,a4paper,reqno]{amsart}

\usepackage{amsmath,amsfonts,amssymb,amsthm, dsfont, bm,color}

\usepackage[all]{xy}
\usepackage{tikz-cd}
\usepackage{enumerate}

\usepackage{geometry}
\usepackage{hyperref}
\hypersetup{colorlinks,linkcolor=[rgb]{0,0,1},citecolor=[rgb]{1,0,0}}

\renewcommand{\geq}{\geqslant}
\renewcommand{\leq}{\leqslant}

\newtheorem{theorem}{Theorem}[section]
\newtheorem{lemma}[theorem]{Lemma}
\newtheorem{proposition}[theorem]{Proposition}
\newtheorem{corollary}[theorem]{Corollary}
\newtheorem{conjecture}[theorem]{Conjecture}
\newtheorem{metaconjecture}[theorem]{Meta-conjecture}

\newtheorem*{conjecture*}{Conjecture}

\theoremstyle{definition}
\newtheorem{definition}[theorem]{Definition}
\newtheorem{example}[theorem]{Example}

\theoremstyle{remark}
\newtheorem{remark}[theorem]{Remark}

\newtheorem{question}[theorem]{Question}

\newcommand{\cM}{\mathcal{M}}

\newcommand{\NN}{\ensuremath{\mathbb{N}}} 
\newcommand{\PP}{\ensuremath{\mathbb{P}}}

\newcommand{\ZZ}{\ensuremath{\mathbb{Z}}}

\newcommand{\vv}{{\bm{\mathrm v}}}

\DeclareMathOperator{\sExt}{\mathcal{E}\!\mathit{xt}}
\DeclareMathOperator{\Gmot}{G_{mot}}
\DeclareMathOperator{\pHS}{HS^{pol}_{\mathbf{Q}}}
\DeclareMathOperator{\HS}{HS}

\DeclareMathOperator{\KS}{KS}
\DeclareMathOperator{\Ku}{Ku}
\DeclareMathOperator{\MT}{MT}

\newcommand{\C}{{\bf C}}
\newcommand{\R}{{\bf R}}
\newcommand{\Z}{{\bf Z}}
\newcommand{\N}{{\bf N}}
\newcommand{\Q}{{\bf Q}}
\newcommand{\ie}{\textit {i.e.}}
\newcommand{\cf}{\textit {cf.}~}
\newcommand{\vs}{\textit {vs.}~}

\newcommand{\resp}{\textit {resp.}~}

\newcommand{\alb}{\mathop{\rm alb}\nolimits} %albanese map
\newcommand{\Bl}{\mathop{\rm Bl}\nolimits} %Blow-up
\newcommand{\CH}{\mathop{\rm CH}\nolimits} % Chow groups
\newcommand{\ch}{\mathop{\rm ch}\nolimits} % Chern character
\newcommand{\End}{\mathop{\rm End}\nolimits} %endomorphism ring
\newcommand{\Ext}{\mathop{\rm Ext}\nolimits} % Ext functor
\newcommand{\Gr}{\mathop{\rm Gr}\nolimits} % Grassmannian or Grading
\newcommand{\Gal}{\mathop{\rm Gal}\nolimits} % Galois groups
\newcommand{\Hilb}{\mathop{\rm Hilb}\nolimits}
\newcommand{\Hom}{\mathop{\rm Hom}\nolimits}
\newcommand{\id}{\mathop{\rm id}\nolimits} %identity
\newcommand{\im}{\mathop{\rm Im}\nolimits} % image, imaginary part is \Im
\newcommand{\Ker}{\mathop{\rm Ker}\nolimits} % kernel
\newcommand{\NS}{\mathop{\rm NS} \nolimits} % morphisms
\newcommand{\pr}{\mathop{\rm pr}\nolimits} % projection
\newcommand{\Sing}{\mathop{\rm Sing}\nolimits} % singular locus
\newcommand{\Spec}{\mathop{\rm Spec}\nolimits}
\newcommand{\Stab }{\mathop{\rm Stab}\nolimits} % Stabilizer
\newcommand{\Sym}{\mathop{\rm Sym}\nolimits} % Symmetric products
\newcommand{\td}{\mathop{\rm td}\nolimits} % Todd class
\newcommand{\GL}{\mathop{\rm GL}\nolimits}

\newcommand{\SO}{\mathop{\rm SO}\nolimits}
\newcommand{\Spin}{\mathop{\rm Spin}\nolimits}
\newcommand{\CSpin}{\mathop{\rm CSpin}\nolimits}

\renewcommand{\bar}{\overline}
\renewcommand{\tilde}{\widetilde}
\renewcommand{\hat}{\widehat}

\newcommand{\cart}{\ar@{}[dr]|\square} % cartesian diagrams, write it after the left-up term to produce a square in the middle of the diagram
\renewcommand{\tilde}{\widetilde}

\newcommand{\AM }{\mathop{\rm {AM}}\nolimits}% Andre motives
\newcommand{\h}{\mathop{\mathfrak {h}}\nolimits} % motive functor
\newcommand{\CHM }{\mathop{\rm {CHM}}\nolimits}% category of Chow motives
\newcommand{\DM }{\mathop{\rm {DM}}\nolimits}% category of mixed motives
\newcommand{\GRM }{\mathop{\rm {GRM}}\nolimits}% category of Grothendieck motives
\renewcommand{\1}{\mathop{\mathds{1}}\nolimits} %trivial Chow motive
\newcommand{\SmProj }{\mathop{\rm {SmProj}}\nolimits}

\newcommand{\Vect }{\mathop{\rm {Vect}}\nolimits}
\DeclareMathOperator{\LLV}{LLV}

\begin{document}

\title{On the motive of O'Grady's ten-dimensional hyper-K\"ahler varieties}

\author{Salvatore Floccari}
\address{Radboud University\\
 IMAPP\\
 Heyendaalseweg 135\\
 6525 AJ, Nijmegen\\
 Netherlands}
 \email{s.floccari@math.ru.nl}

\author{Lie Fu}
\address{Universit\'e Claude Bernard Lyon 1\\
Institut Camille Jordan\\
43 Boulevard du 11 novembre 1918\\
69622 Cedex Villeurbanne\\
France}
\address{$\&$}
\address{Radboud University\\
IMAPP\\
Heyendaalseweg 135\\
6525 AJ, Nijmegen\\
Netherlands}
\email{fu@math.univ-lyon1.fr}

\author{Ziyu Zhang}
\address{ShanghaiTech University\\
Institute of Mathematical Sciences\\
393 Middle Huaxia Road\\
Shanghai 201210\\
P.R.China}
\email{zhangziyu@shanghaitech.edu.cn}



\begin{abstract}
We investigate how the motive of hyper-K\"ahler varieties is controlled by weight-2 (or surface-like) motives via tensor operations. In the first part, we study the Voevodsky motive of singular moduli spaces of semistable sheaves on K3 and abelian surfaces as well as the Chow motive of their crepant resolutions, when they exist. We show that these motives are in the tensor subcategory generated by the motive of the surface, provided that a crepant resolution exists. This extends a recent result of B\"ulles to the O'Grady-10 situation. In the non-commutative setting, similar results are proved for the Chow motive of moduli spaces of (semi-)stable objects of the K3 category of a cubic fourfold. As a consequence, we provide abundant examples of hyper-K\"ahler varieties of O'Grady-10 deformation type satisfying the standard conjectures. In the second part, we study the Andr\'{e} motive of projective hyper-K\"{a}hler varieties. We attach to any such variety its defect group, an algebraic group which acts on the cohomology and measures the difference between the full motive and its weight-2 part. When the second Betti number is not 3, we show that the defect group is a natural complement of the Mumford--Tate group inside the motivic Galois group, and that it is deformation invariant. We prove the triviality of this group for all known examples of projective hyper-K\"{a}hler varieties, so that in each case the full motive is controlled by its weight-2 part. As applications, we show that for any variety motivated by a product of known hyper-K\"ahler varieties, 
all Hodge and Tate classes are motivated, the motivated Mumford--Tate conjecture \ref{conj:mtcMot} holds, and the Andr\'e motive is abelian. This last point completes a recent work of Soldatenkov and provides a different proof for some of his results. 
\end{abstract}

\maketitle
\setcounter{tocdepth}{1}
\tableofcontents

\section{Introduction}
An important source of constructions of higher-dimensional algebraic varieties is given by taking moduli spaces of (complexes of) coherent sheaves, subject to various stability conditions, on some lower-dimensional algebraic varieties. The topological, geometric, algebraic and arithmetic properties of the variety are  certainly expected to be reflected in and sometimes even control the corresponding properties of the moduli space. Such relations can be made precise in terms of cohomology groups (enriched with Hodge structures and Galois actions for instance) or more fundamentally, at the level of motives\footnote{We work with rational coefficients for cohomology groups and motives. All varieties are defined over the field of complex numbers if not otherwise specified.}. The prototype of such interplay we have in mind is del Ba\~no's result \cite{Ban01}, which says that the Chow motive of the moduli space $\cM_{r, d}(C)$ of stable vector bundles of coprime rank and degree on a smooth projective curve $C$ is a direct summand of the Chow motive of some power of the curve; in other words, the Chow motive of $\cM_{r, d}(C)$ is in the pseudo-abelian tensor subcategory generated by the Chow motive of $C$. In \cite{Ban01}, a precise formula for the virtual motive of $\cM_{r, d}(C)$ in terms of the virtual motive of $C$ was obtained, a result which has been recently lifted to the level of motives in a greater generality by Hoskins and Pepin-Lehalleur \cite{HPL18}.

In the realm of compact hyper-K\"ahler varieties \cite{beauville1983varietes} \cite{Huy99}, this philosophy plays an even more important role: it turns out that taking the moduli spaces of (complexes of) sheaves on Calabi--Yau  surfaces or their non-commutative analogues provides the most general and almost exhaustive way for constructing examples, see \cite{Muk84} \cite{O'G99} \cite{O'G03} \cite{PR13} \cite{Yos01} \cite{BM14a} \cite{BM14b} \cite{BLMS} and \cite{BLMNPS} etc. As the first important relationship between the K3 or abelian surface $S$ and a moduli space $\cM:=\cM_S(\vv)$ of stable (complexes of) sheaves on $S$ with (non-isotropic) Mukai vector $\vv$, the second cohomology of $\cM$ is identified, as a Hodge lattice, with the orthogonal complement of $\vv$ in $\tilde{H}(S, \Z)$, the Mukai lattice of $S$ \cite{O'G97} \cite{PR13}. Regarding the aforementioned result of del Ba\~no in the curve case, a relation between the motive of $S$ and the motive of $\cM$ is desired. 
The first breakthrough in this direction is the following result of B\"ulles \cite{Bue18} based on the work of Markman \cite{Mar02}.
\begin{theorem}[B\"ulles]\label{thm:Buelles}
Let $S$ be a projective K3 or abelian surface together with a Brauer class $\alpha$. Let $\cM$ be a smooth and projective moduli space of stable objects in $D^b(S, \alpha)$ with respect to some Bridgeland stability condition. Then the Chow motive of $\cM$ is contained in the pseudo-abelian tensor subcategory generated by the Chow motive of $S$.
\end{theorem}

The analogous result on the level of Grothendieck motives or Andr\'e motives was obtained before by Arapura \cite{MR2271985}. It is also worth pointing out that B\"ulles' method gives a short and new proof of del Ba\~no's result using the classical analogue of Markman's result in the curve case proved by Beauville \cite{Bea95}. 

\subsection{Singular or open moduli spaces and resolutions}
\label{subsec:Intro1}
The first objective of the paper is to investigate the situations beyond Theorem \ref{thm:Buelles}.

More precisely, let us fix a projective K3 or abelian surface $S$ together with a Brauer class $\alpha$, a not necessarily primitive Mukai vector $\vv$ and a not necessarily generic stability condition $\sigma$ on $D^b(S, \alpha)$. We want to understand, in terms of the motive of $S$, the (mixed) motives \cite{Voe00} of the following varieties (or algebraic spaces\footnote{See the recent work \cite{AHLH} for the existence of good moduli spaces.}). 
\begin{itemize}
    \item  The (smooth but in general non-proper) moduli space of $\sigma$-stable objects $$\cM^\mathrm{st}:=\cM^\mathrm{st}_{S, \sigma}(\vv, \alpha).$$
    \item The (proper but in general singular) moduli space of $\sigma$-semistable objects $$\cM:=\cM_{S, \sigma}(\vv, \alpha).$$
    \item A crepant resolution $\tilde\cM$ of $\cM$, if exists.
\end{itemize}
Here is our expectation for their motives:
\begin{conjecture}\label{conj:SingularModuli}
Notation is as above. 
\begin{enumerate}[$(i)$]
    \item The motives and the motives with compact support (in the sense of Voevodsky) of $\cM^\mathrm{st}$ and $\cM$ are in the triangulated tensor subcategory generated by the motive of $S$ within the category of Voevodsky's geometric motives.
    \item  The Chow motive of $\tilde{\cM}$ (if it exists) is in the pseudo-abelian tensor subcategory generated by the motive of $S$ within the category of Chow motives.
\end{enumerate}
\end{conjecture}

Our first main result below confirms Conjecture \ref{conj:SingularModuli} in the presence of a crepant resolution. Recall that by \cite{KLS06}, this happens only in the case of O'Grady's ten-dimensional example \cite{O'G99} (extended by \cite{PR13}).
\begin{theorem}[=Corollaries \ref{cor:motive_tildeM} and \ref{cor:motive_Ms}]\label{thm:main1}
    Let $S$ be a projective K3 or abelian surface, let $\alpha$ be a Brauer class, let $\vv_0\in \tilde{H}(S)$ be a primitive Mukai vector with $\vv_0^2=2$, and let $\sigma$ be a $\vv_0$-generic stability condition on $D^b(S, \alpha)$. Denote by $\cM^\mathrm{st}$ (\resp $\cM$) the 10-dimensional moduli space of $\sigma$-stable (\resp semistable) objects in $D^b(S,\alpha)$ with Mukai vector $\vv=2\vv_0$. Let $\tilde{\cM}$ be any crepant resolution of $\cM$. Then
    the conclusions of Conjecture \ref{conj:SingularModuli} hold.
\end{theorem}

Note that by the result of Rie{\ss}   \cite{Rie14}, birational hyper-K\"ahler varieties have isomorphic Chow motives, hence we only need to treat one preferred crepant resolution, namely the one constructed by O'Grady \cite{O'G99}.

\begin{remark}[Hodge numbers of OG10]
The Hodge numbers of hyper-K\"ahler varieties of OG10-type are recently computed by de Cataldo--Rapagnetta--Sacc\`a in \cite{dCRS} via the decomposition theorem and a refinement of Ng\^o's support theorem. A representation theoretic approach was discovered shortly after by Green--Kim--Laza--Robles \cite[Theorem 3.26]{GKLR}, where the vanishing of the odd cohomology is required to conclude. Note that Theorem \ref{thm:main1} implies in particular the triviality of the odd cohomology of hyper-K\"ahler varieties of OG10-type and hence allows \cite{GKLR} to obtain an independent proof of  \cite[Theorem A]{dCRS}; see \cite[Remark 3.30]{GKLR}.
\end{remark}

\subsection{Non-commutative Calabi--Yau ``surfaces''}
We see in the above setting that the Calabi--Yau surface plays its role almost entirely through its derived category and the second goal of the paper is to extend Theorem \ref{thm:Buelles} and the results of \S\ref{subsec:Intro1}  to the non-commutative setting.

Indeed, it has been realized since \cite{BLMS} that one can develop an equally satisfactory theory of moduli spaces starting with a 2-Calabi--Yau category $\mathcal{A}$, \ie~an $\Ext$-finite saturated triangulated category in which the double shift $[2]$ is a Serre functor, equipped with Bridgeland stability conditions. Such a category often comes as an admissible subcategory of the derived category of a Fano variety, as the ``key" component (the so-called \textit{Kuznetsov component}) in some semi-orthogonal decomposition. We expect the similar relations as in \S\ref{subsec:Intro1}
between the motive of the moduli space of stable objects in this category $\mathcal{A}$ and the (non-commutative) motive of $\mathcal{A}$, hence also the motive of the Fano variety. 

To be more precise, let us leave the general technical results to \S\ref{sec:NCK3} and stick in the introduction to the most studied example of such 2-Calabi--Yau categories, namely the Kuznetsov component of  the derived category of a cubic fourfold. Let $Y$ be a smooth cubic fourfold and let $\Ku(Y)\colon\!\!\!\!=\langle \mathcal{O}_Y,\mathcal{O}_Y(1),\mathcal{O}_Y(2)\rangle^{\perp}=\{E\in D^b(Y)~\mid~ \Ext^*(\mathcal{O}_Y(i), E)=0 \text{ for } i=0, 1, 2\}$ be its Kuznetsov component, which is a K3 category.\footnote{A K3 category is a 2-Calabi--Yau category whose Hochschild homology coincides with that of a K3 surface.} One can associate with it a natural Hodge lattice $\tilde{H}(\Ku(Y))$ using topological K-theory \cite{AT14}. In \cite{BLMS}, a natural stability condition on $\Ku(Y)$ is constructed and by the general theory of Bridgeland \cite{Bri07}, we have  at our disposal a connected component of the manifold of stability conditions, denoted by $\Stab^\dagger(\Ku(Y))$.

Our second main result generalizes B\"ulles' Theorem \ref{thm:Buelles} to this non-commutative setting:
\begin{theorem}[Special case of Theorem \ref{thm:NCK3_smooth}]\label{thm:main2}
Let $Y$ be a smooth cubic fourfold, let $\Ku(Y)$ be its Kuznetsov component, let $\vv\in \tilde{H}(\Ku(Y))$ be a primitive Mukai vector, and let $\sigma\in \Stab^\dagger(\Ku(Y))$ be a $\vv$-generic stability condition. Then the Chow motive of the projective hyper-K\"ahler manifold $\cM:=\cM_{\Ku(Y), \sigma}(\vv)$ is in the pseudo-abelian tensor subcategory generated by the Chow motive of~$Y$.
\end{theorem}

\begin{remark}
Note that by the recent work of Li--Pertusi--Zhao \cite{LPZ18}, the moduli spaces considered in Theorem \ref{thm:main2} already include the hyper-K\"ahler fourfold $F(Y)$ constructed as Fano variety of lines in $Y$ \cite{BD85} and the hyper-K\"ahler eightfold $Z(Y)$ constructed from twisted cubics in $Y$ (when $Y$ does not contain a plane) \cite{LLSvS}. In the first case, the conclusion of Theorem \ref{thm:main2} can be deduced from the earlier work of Laterveer \cite{Lat17}; in the second case, our approach was speculated in \cite[Remark 2.7]{CCL18}. Nevertheless, Theorem \ref{thm:main2} applies to the infinitely many complete families of projective hyper-K\"ahler varieties recently constructed by Bayer \textit{et al.} \cite{BLMNPS}.
\end{remark}

Just as in \S\ref{subsec:Intro1}, for non-primitive Mukai vectors or non-generic stability conditions, the moduli space of stable (\resp semistable) objects $\cM^{\mathrm{st}}:=\cM^{\mathrm{st}}_{\Ku(Y), \sigma}(\vv)$ (\resp $\cM:=\cM_{\Ku(Y), \sigma}(\vv)$) is in general not proper (\resp smooth). We expect the following analogy of Conjecture \ref{conj:SingularModuli} in this non-commutative setting.
\begin{conjecture}[Special case of Conjecture \ref{conj:SingularModuli_NC}]
\label{conj:cubic_singular}
Notation is as above. 
\begin{enumerate}[$(i)$]
    \item The motives and the motives with compact support (in the sense of Voevodsky) of $\cM^\mathrm{st}$ and $\cM$ are in the triangulated tensor subcategory generated by the motive of $Y$ within the category of Voevodsky's geometric motives.
    \item  If there exists a crepant resolution $\tilde{\cM}\to \cM$, then the Chow motive of $\tilde{\cM}$ is in the pseudo-abelian tensor subcategory generated by the motive of $Y$ within the category of Chow motives.
\end{enumerate}
\end{conjecture}

Analogously to Theorem \ref{thm:main1}, our third result establishes Conjecture \ref{conj:cubic_singular} in the ten-dimensional situation studied in \cite{LPZ19}, where a crepant resolution of $\cM$ exists (it is again of O'Grady-10 deformation type). Recall that $\tilde{H}(\Ku(Y))$ contains (and it is equal to, if $Y$ is very general) a canonical $A_2$-lattice generated by $\lambda_1$ and $\lambda_2$ (see \cite{BLMS} for the notation).

\begin{theorem}
    \label{thm:main3}
    Notation is as above. Assume that $Y$ is very general.  Let the Mukai vector $\vv=2\vv_0$ with $\vv_0=\lambda_1+\lambda_2$ and let $\sigma$ be $\vv_0$-generic. Then  
   the conclusions of Conjecture \ref{conj:cubic_singular} hold true for $\cM^{\mathrm{st}}$, $\cM$ and any crepant resolution $\tilde{\cM}$ of $\cM$.
\end{theorem}

As a by-product, we deduce Grothendieck's standard conjectures \cite{MR0268189} \cite{MR1265519} for many hyper-K\"ahler varieties of O'Grady-10 deformation type, \textit{cf.~\cite{CM13}}.

\begin{corollary}
    \label{cor:StandConj}
    The standard conjectures hold for all the crepant resolutions $\tilde{\cM}$ appeared in Theorem \ref{thm:main1} and Theorem \ref{thm:main3}.
\end{corollary}

Theorem \ref{thm:main3} and Corollary \ref{cor:StandConj} are proved in the end of \S \ref{sec:NCK3}.

\subsection{Defect groups of hyper-K\"ahler varieties}
It is easy to see that a general projective deformation of (a crepant resolution of) a moduli space of semistable sheaves (or objects) on a Calabi--Yau surface is no longer of this form (even by deforming the surface).
If we still want to understand the motive of such a hyper-K\"ahler variety $X$ in terms of some tensor constructions of a ``surface-like'' (or rather ``weight-2'') motive, the right substitution of the surface motive would be the degree-2 motive of $X$ itself. We are therefore interested in the following meta-conjecture.
\begin{metaconjecture}\label{conj:meta}
Let $X$ be a projective hyper-K\"ahler variety and fix some rigid tensor category of motives. If the odd Betti numbers of $X$ vanish, then its motive is in the tensor subcategory generated by its degree-2 motive. In general, the motive of $X$ lies in the tensor subcategory generated by the Kuga--Satake construction of its degree-2 motive. In any case, the motive of $X$ is abelian. 
%If $X$ has some non-zero odd Betti number, then its motive lies in the tensor subcategory generated by the Kuga--Satake construction of its degree-2 motive. In any case, the motive of $X$ is abelian.
\end{metaconjecture}

We will see in Proposition \ref{cor:MT} that the analogous statement holds at the level of Hodge structures. This is essentially a consequence of Verbitsky's results \cite{verbitsky1996cohomology}, related works are \cite{KSV2017} and \cite{solda19}. Unfortunately, staying within the category of Chow motives (or Voevodsky motives), we are confronted with several essential difficulties:
\begin{itemize}
    \item As an immediate obstruction, to speak of the degree $2$ motive, we have to admit the algebraicity of the K\"unneth projector, which is part of the standard conjectures.
    \item Even in the case where the standard conjectures are known (for example \cite{CM13}), the construction of the degree $2$ part of the Chow motive $\h(X)$, denoted by $\h^2(X)$, is still conjectural in general: assuming the cohomological K\"{u}nneth projector is algebraic, there is no canonical way to lift it to an algebraic cycle which is a projector modulo rational equivalence (see Murre \cite{Mur93}); even when such a candidate construction is available (see for example \cite{SV16}, \cite{VialCrelle} \cite{FTV19} in some special cases), it seems too difficult to relate $\h(X)$ and $\h^2(X)$ for a general $X$ in the moduli space of hyper-K\"ahler varieties. Nevertheless, let us point out that B\"ulles' Theorem \ref{thm:Buelles} and our extensions Theorems \ref{thm:main1}, \ref{thm:main2} and \ref{thm:main3} indeed give some evidence in this direction (see also Corollary~\ref{cor:InfChowAb}).
    \item The algebraicity of the Kuga--Satake construction is wide open.
\end{itemize}

The third purpose of the paper is to make precise sense of the meta-conjecture \ref{conj:meta}. To circumvent the aforementioned difficulties we leave the category of Chow motives and work within the category of \textit{Andr\'e motives} \cite{andre1996Motives}. Essentially, this amounts to replacing rational equivalence by homological equivalence and formally adding the cycles predicted by the standard conjectures; the result is a semisimple abelian $\Q$-linear tannakian category, see \S \ref{subsec:AM} for a quick introduction. Through the tannakian formalism, most properties of an Andr\'{e} motive $M$ are encoded in its motivic Galois group $\Gmot(M)$. Note that since the Hodge theoretic version of meta-conjecture \ref{conj:meta} holds, its validity at the level of Andr\'{e} motives is implied by Conjecture \ref{mhc} which says that all Hodge classes are motivated.

%The category of motives modulo homological equivalence is the so-called Grothendieck motives (\cf \cite{andre} or \S\ref{subsec:AM}). However, to get around the technical difficulty imposed by the standard conjecture, we work with the neutral tannakian semisimple abelian category of Andr\'e motives \cite{andre1996Motives}, which agrees with Grothendieck motives assuming the standard conjecture\,

Our main contribution in this direction is about a $\Q$-algebraic group, which we call the \textit{defect group}, associated with a projective hyper-K\"ahler variety.
Let $X$ be a projective hyper-K\"ahler variety and let $\mathcal{H}(X)$ be its Andr\'e motive. We have the K\"unneth decomposition $\mathcal{H}(X)=\bigoplus_{i} \mathcal{H}^i(X)$. The \textit{even motive} of $X$ is by definition $\mathcal{H}^+(X)=\bigoplus_{i} \mathcal{H}^{2i}(X)$. 
The \textit{even defect group} of $X$, denoted by $P^+(X)$, is defined as the kernel of the surjective morphism of motivic Galois groups induced by the natural inclusion $\mathcal{H}^2(X)\subset \mathcal{H}^+(X)$, namely,
$$P^+(X):=\Ker\left(\Gmot(\mathcal{H}^+(X))\twoheadrightarrow \Gmot(\mathcal{H}^2(X))\right).$$
By definition, $P^+(X)$ is trivial if and only if $\mathcal{H}^+(X)$ belongs to the tannakian subcategory of Andr\'{e} motives generated by $\mathcal{H}^2(X)$.

If all the odd Betti numbers of $X$ vanish, then by convention the \textit{defect group} of $X$, denoted by $P(X)$, is simply $P^+(X)$. Otherwise, the role of $\mathcal{H}^2(X)$ is naturally taken by a Kuga--Satake abelian variety $A$ attached to this weight-2 motive, see Definition \ref{def:Kuga-Satake}; the reader may safely take for $A$ the abelian variety given by the classical Kuga--Satake construction \cite{deligne1971conjecture}. The \emph{Kuga--Satake category} $\mathsf{KS}(X):= \langle \mathcal{H}^1(A) \rangle $ is independent of the choice of $A$, see Theorem \ref{thm:kuga-satake}; furthermore, provided that $b_2(X) \neq 3$, we prove in Lemma \ref{lem:KSvsodd} that the motive $\mathcal{H}^1(A)$ belongs to the tannakian subcategory of Andr\'{e} motives generated by~$\mathcal{H}(X)$. We define the \textit{defect group} of $X$ as the kernel of the corresponding surjective morphism
$$P(X):=\Ker\left(\Gmot(\mathcal{H}(X))\twoheadrightarrow \Gmot(\mathcal{H}^1(A))\right)$$
 of motivic Galois groups.
The uniqueness of the Kuga--Satake category ensures that $P(X)$ does not depend on the choice of $A$; by definition, the defect group $P(X)$ is trivial if and only if $\mathcal{H}(X)$ belongs to the tannakian category $\mathsf{KS}(X)$.

%If all the odd Betti numbers of $X$ vanish, then by convention the \textit{defect group} of $X$, denoted by $P(X)$, is simply $P^+(X)$. In the presence of non-vanishing odd Betti numbers, we can show that $\mathcal{H}^1(\KS(X))$ is in the tannakian subcategory of Andr\'e motives generated by $\mathcal{H}(X)$, where $\KS(X)$ is the Kuga--Satake abelian variety associated to the Hodge structure $H^2(X, \ZZ)$. Then one defines the \textit{defect group} of $X$ as the following kernel
%$$P(X):=\Ker\left(\Gmot(\mathcal{H}(X))\twoheadrightarrow \Gmot(\mathcal{H}^1(\KS(X)))\right).$$

%Alternatively, we can show (see the proof of Theorem \ref{thm:product}$(ii)$) that the kernel of the surjective morphism on motivic Galois groups associated to the inclusion $\mathcal{H}^2(X)\subset \mathcal{H}(X)$ naturally splits as the direct product of the order-2 subgroup $\langle\iota\rangle$ and $P(X)$:
%$$\Ker\left(\Gmot(\mathcal{H}(X))\twoheadrightarrow \Gmot(\mathcal{H}^2(X))\right)=P(X)\times \langle\iota\rangle,$$
%where $\iota$ is the automorphism of $\mathcal{H}(X)$ which acts on $\mathcal{H}^i(X)$ as the multiplication by $(-1)^i$.

Recall that the motivic Galois group of $\mathcal{H}(X)$ contains naturally the \textit{Mumford--Tate group} $\MT(H^*(X))$. We show that the defect group is a canonical complement. 
%The reason behind the assumption that $b_2\neq 3$ is Andr\'{e}'s Theorem \ref{abelianityH2} which says that in this case the motive $\mathcal{H}^2(X)$ is abelian. This condition is verified in all known examples.

\begin{theorem}[=Theorem \ref{thm:product2}, Splitting]\label{thm:product}
Notation is as before. Assume that $b_2(X)\neq 3$.
Then, inside $\Gmot(\mathcal{H}(X))$, the subgroups $P(X)$ and $\MT(H^*(X))$ commute, intersect trivially with each other and generate the whole group. In short, we have an equality:
    $$\Gmot(\mathcal{H}(X))=\MT(H^*(X))\times P(X).$$
 Similarly, the even defect group is a direct complement of the even Mumford--Tate group in the motivic Galois group of the even Andr\'e motive of $X$,
 $$\Gmot(\mathcal{H}^+(X))=\MT(H^+(X))\times P^+(X).$$
\end{theorem}
It follows that $\MT(H^+(X))$
is canonically isomorphic to $\Gmot(\mathcal{H}^2(X))$, and hence to $\MT(H^2(X))$ by Andr\'e's results \cite{Andre1996} \cite{andre1996Motives}. But this is the first step towards the proof of Theorem \ref{thm:product} (see Proposition \ref{cor:MT}). 
Note that the natural morphism $\Gmot(\mathcal{H}(X))\twoheadrightarrow \Gmot(\mathcal{H}^+(X))$ preserves the direct product decomposition given in the theorem, so that $P^+(X)$ is a quotient of $P(X)$.

Theorem \ref{thm:product} can be seen as a structure result for the motivic Galois group. The proof is given in \S\ref{subsec:Splitting}. It admits the following consequence (proved in \S\ref{subsec:Justification}), which justifies the name of the group $P(X)$.
\begin{corollary}[=Corollary \ref{cor:Pdefect2}]
\label{cor:Pdefect}
    For any projective hyper-K\"ahler variety $X$ with $b_2(X)\neq 3$, the following conditions are equivalent:
    \begin{enumerate}[$(i^+)$]
        \item The even defect group $P^+(X)$ is trivial.
        \item The even Andr\'e motive $\mathcal{H}^+(X)$ is in the tannakian subcategory generated by $\mathcal{H}^2(X)$.
        \item $\mathcal{H}^+(X)$ is abelian.
        \item Conjecture \ref{mhc} holds for $\mathcal{H}^+(X)$: $\MT(H^+(X))=\Gmot(\mathcal{H}^+(X)).$
    \end{enumerate}
    Similarly, if some odd Betti number of $X$ is not zero, we have the following equivalent conditions:
    \begin{enumerate}[$(i)$]
        \item The defect group $P(X)$ is trivial.
        \item The Andr\'e motive $\mathcal{H}(X)$ is in the tannakian subcategory generated by $\mathcal{H}^1(\mathrm{KS}(X))$, where $\mathrm{KS}(X)$ is any Kuga--Satake abelian variety associated to $H^2(X)$.
        \item $\mathcal{H}(X)$ is abelian.
        \item Conjecture \ref{mhc} holds for $\mathcal{H}(X)$: $\MT(H^*(X))=\Gmot(\mathcal{H}(X))$.
    \end{enumerate}
\end{corollary}

Thanks to Corollary \ref{cor:Pdefect}, Conjecture \ref{mhc} for hyper-K\"ahler varieties and the meta-conjecture \ref{conj:meta} for their Andr\'e motives are all equivalent to the following conjecture.
\begin{conjecture}\label{conj:Ptrivial}
The defect group of any projective hyper-K\"ahler variety is trivial. 
\end{conjecture}

\begin{remark}[Potential approaches]\label{rm:Potential}
We are not able to prove Conjecture \ref{conj:Ptrivial} in general so far, but only for all the known examples of hyper-K\"ahler varieties (Corollary \ref{cor:KnownHKs} below). However,  
\begin{enumerate}[$(i)$]
    \item we will show in Corollary \ref{cor:application} that Conjecture \ref{conj:Ptrivial} is implied by the following conjecture: an Andr\'e motive is of Tate type if and only if its Hodge realization is of Tate type;
    \item The defect group satisfies many constraints. For example, its action on the rational cohomology ring is compatible with the ring structure as well as the Looijenga--Lunts--Verbitsky Lie algebra action \cite{looijenga1997lie} \cite{verbitsky1996cohomology}, and most importantly, it is a deformation invariant.
\end{enumerate}
\end{remark}
 
\begin{theorem}[=Theorem \ref{thm:DeformationInv2}, Deformation invariance of defect groups]\label{thm:DeformationInv}
    Let $S$ be a smooth quasi-projective variety and $\mathcal{X}\to S$ be a smooth proper morphism with fibers being projective hyper-K\"ahler manifolds with $b_2\neq 3$. Then for any $s, s'\in S$, the defect groups $P(X_s)$ and $P(X_{s'})$ are canonically isomorphic, and similarly for the even defect groups.
\end{theorem}

\subsection{Applications to ``known" hyper-K\"ahler varieties}
In the sequel, a hyper-K\"ahler variety is called \textit{known}, if it is deformation equivalent to Hilbert schemes of K3 surfaces ($\mathrm{K3}^{[n]}$-type) \cite{beauville1983varietes}, generalized Kummer varieties associated to abelian surfaces ($\operatorname{Kum}^n$-type) \cite{beauville1983varietes}, O'Grady's 6-dimensional examples (OG6-type) \cite{O'G03}, or O'Grady's 10-dimensional examples (OG10-type)~\cite{O'G99}.

First, we can prove Conjecture \ref{conj:Ptrivial} for all known hyper-K\"ahler varieties.
\begin{corollary}\label{cor:KnownHKs}
    The defect group is trivial for all known hyper-K\"ahler varieties.
\end{corollary}

Combining this with Corollary \ref{cor:Pdefect}, we have the following consequences, providing evidences to the meta-conjecture \ref{conj:meta} in the world of Andr\'e motives. 
\begin{corollary}\label{cor:KnownHKs2}
Let $X$ be a \textit{known} projective hyper-K\"ahler variety. Then 
\begin{enumerate}[$(i)$]
    \item its Andr\'e motive is abelian;
    \item for any $m\in \N$, all Hodge classes of $H^*(X^{m}, \Q)$ are motivated (hence absolutely Hodge);
    \item if $X$ is of $\mathrm{K3}^{[n]}$, OG6, or OG10-type, then $\mathcal{H}(X)\in \langle\mathcal{H}^2(X)\rangle$;
    \item if $X$ is of $\operatorname{Kum}^n$-type, then $\mathcal{H}(X)\in \langle\mathcal{H}^1(\KS(X))\rangle$ and $\mathcal{H}^+(X)\in \langle\mathcal{H}^2(X)\rangle$.
\end{enumerate}
\end{corollary}
The item $(i)$ on the abelianity of Andr\'e motive is proved for $\mathrm{K3}^{[n]}$-type by Schlickewei \cite{Sch12}, for $\operatorname{Kum}^n$-type and OG6-type in the recent work of Soldatenkov \cite{soldatenkov19}.

Second, we can prove the Mumford--Tate conjecture for all known hyper-K\"ahler varieties defined over a finitely generated field extension of $\Q$; see \S\ref{subsec:MTC} for the precise statement of the conjecture. For varieties of $\mathrm{K3}^{[n]}$-type, it has been proven in \cite{floccari2019}.
In fact, what we obtain in the following Theorem \ref{thm:MMTKnownHK} is a stronger result in two aspects:
%is even better in two aspects: 
\begin{itemize}
    \item we identify the Mumford--Tate group and the Zariski closure of the image of the Galois representation \textit{via} a third group, namely the motivic Galois group. This is the so-called \emph{motivated Mumford--Tate conjecture} \ref{conj:mtcMot};
    \item we can treat products. In general, it is far from obvious to deduce the Mumford--Tate conjecture for a product of varieties from the conjecture for the factors. Thanks to the work of Commelin \cite{commelin2019} this can be done when the Andr\'e motives of the varieties involved are abelian.
\end{itemize}

\begin{theorem}[Special case of Theorem \ref{thm:abelianHKmotives1}]\label{thm:MMTKnownHK}
Let $k$ be a finitely generated subfield of $\C$. For any smooth projective $k$-variety that is motivated~\footnote{A smooth projective variety $X$ is said to be \textit{motivated} by another smooth projective variety $Y$ if its Andr\'e motive $\mathcal{H}(X)$ belongs to $\langle \mathcal{H}(Y)\rangle$, the tannakian subcategory of Andr\'e motives generated by $\mathcal{H}(Y)$; or equivalently, $\mathcal{H}(X)$ is a direct summand of the Andr\'e motive of a power of $Y$. Note that any non-zero divisor of $Y$ gives rise to a splitting injection $\mathbb{Q}(-1)\to \mathcal{H}(Y)$, hence $\langle \mathcal{H}(Y)\rangle$ is automatically stable by Tate twists.}
by a product of \textit{known} hyper-K\"ahler varieties, the motivated Mumford--Tate conjecture \ref{conj:mtcMot} holds.
In particular, the Tate conjecture and the Hodge conjecture are equivalent for such varieties.
\end{theorem}

\begin{remark}[Relation to \cite{soldatenkov19}]
The combination of Corollary \ref{cor:Pdefect} and Theorem \ref{thm:DeformationInv} (plus the fact that two deformation equivalent hyper-K\"ahler varieties can be connected by algebraic families) implies that the abelianity of the Andr\'e motive of hyper-K\"ahler varieties is a deformation invariant property (Corollary \ref{cor:application}$(i)$). When finalizing the paper, we discovered the recent update of Soldatenkov's preprint \cite{soldatenkov19}, where he also obtained this result, as well as Corollary \ref{cor:KnownHKs2}$(i)$, except for the O'Grady-10 case. We attribute the overlap to him. The proofs and points of view are somewhat different: \cite{soldatenkov19} makes a detailed study of the Kuga--Satake construction in families, while our argument does not involve the Kuga--Satake construction when the odd cohomology is trivial, but relies on Andr\'e's theorem \cite{Andre1996} on the abelianity of $\mathcal{H}^2$. As a bonus of  emphasizing the usage of defect groups in our study, on the one hand, there seems to be some promising approaches mentioned in Remark \ref{rm:Potential} to show the abelianity of the Andr\'e motive of hyper-K\"ahler varieties in general; on the other hand, even if Conjecture \ref{conj:Ptrivial} (hence the abelianity) turned out to fail for some deformation family of hyper-K\"ahler varieties, our method can still control their Andr\'e motives by its degree-2 part together with information on the Andr\'e motive of one given member in that family, see Corollary \ref{cor:application} $(ii), (iii)$.
\end{remark}

\noindent{\textbf{Convention:}} From \S\ref{sec:StableModuli} to \S\ref{subsec:KnownHK}, all varieties are defined over the field of complex numbers $\C$ if not otherwise stated. 
%From \S\ref{sec:defect}, we consider only \textit{irreducible} hyper-K\"{a}hler varieties, and drop the adjective irreducible. 
We work exclusively with rational coefficients for cohomology groups and Chow groups, as well as for the corresponding categories of motives. For simplicity, the notation $\CHM$ (\resp $\DM$, $\AM$) stands for the category of rational Chow motives (\resp rational geometric motives in the sense of Voevodsky, rational Andr\'e motives) over a base field $k$, which are usually denoted by $\CHM(k)_{\Q}$ (\resp $\DM_{gm}(k)_{\Q}$, $\AM(k)$) in the literature.

\noindent{\textbf{Acknowledgement:}} We want to thank Chunyi Li for helpful discussions and Ben Moonen for his careful reading of the draft. We also thank the referee for all the helpful comments for improvement.

\section{Generalities on motives}
In this section, we recall various categories of motives that we will be using, gather some of their basic properties, and explain some relations between them. Most of the content is standard and well-documented, except Proposition \ref{prop:GeneratedCats} and results of \S \ref{subsubsec:familiesAM} in the non-projective setting.

\subsection{Chow motives}\label{subsec:CHM}
Let $\mathrm{SmProj}_{k}$ be the category of smooth projective varieties over an arbitrary base field $k$.
Let $\CHM$ be the category of Chow motives with rational coefficients, equipped with the functor $$\h: \mathrm{SmProj}^{op}_{k}\to \CHM.$$ 
We follow the notation and conventions of \cite{andre}. $\CHM$ is a pseudo-abelian rigid symmetric tensor category, whose objects consist of triples
$(X,p,n)$, where $X$ is a smooth projective variety of dimension $d_X$ over the base field $k$, $p\in	\CH^{d_X}(X\times_{k} X)$ with $p\circ p = p$, and $n\in \Z$. Morphisms $f: M=(X,p,n) \to N=(Y,q,m)$ are elements $\gamma \in \CH^{d_X+m-n}(X\times_{k} Y)$ such that $\gamma \circ p = q\circ \gamma = \gamma$. The tensor product of two motives is defined in the obvious way by the fiber product over the base field, while the dual of $M=(X, p, n)$ is $M^\vee = (X,{}^tp,-n+d_X)$, where ${}^tp$ denotes the transpose of $p$. The Chow motive of a smooth projective variety $X$ is defined as $\h(X) := (X,\Delta_X,0)$,	where $\Delta_X$ denotes the class of the diagonal inside $X\times_{k} X$, and the
\emph{unit motive} is denoted by $\mathds{1} := \h(\Spec (k))$. In
particular, we have $\CH^l(X) = \Hom(\mathds{1}(-l),\h(X))$. The
\emph{Tate motive} of weight $-2i$ is the motive $\mathds{1}(i) := (\Spec (k), \Delta_{\Spec(k)},i)$. A motive is said to be of \emph{Tate type} if it is isomorphic to a direct sum of Tate motives (of various weights).

Given a Chow motive $M\in \CHM$, the pseudo-abelian tensor subcategory of $\CHM$ generated by $M$ is by definition the smallest full subcategory of $\CHM$ containing $M$ that is stable under isomorphisms, direct sums, direct summands, tensor products and duality. We denote this subcategory by $\langle M \rangle_{\CHM}$; it is again a pseudo-abelian rigid tensor category.
Note that if $M=\h(X)$ is the motive of a smooth projective variety $X$, then any divisor on $X$ gives rise to a splitting injection $\1(-1)\to \h(X)$; therefore when $X$ has a non-zero divisor, $\langle \h(X)\rangle_{\CHM}$ contains the Tate motives and hence it is also stable under Tate twists.

\subsection{Mixed motives}\label{subsec:DM}
Let $\mathrm{Sch}_{k}$ be the category of separated schemes of finite type over a perfect base field $k$.
Let $\DM$ be Voevodsky's triangulated category of geometric motives over $k$ with rational coefficients \cite{Voe00}.
There are two canonical functors 
$$M: \operatorname{Sch}_{k}\to\DM \text{ and } M_{c}: (\operatorname{Sch}_{k}, \text{proper morphisms})\to\DM.$$
For any $X\in \operatorname{Sch}_k$, $M(X)$ is called its \textit{(mixed) motive}  and $M_c(X)$ is called its \textit{motive with compact support} (or rather its \emph{Borel--Moore motive}). There is a canonical comparison morphism $M(X)\to M_c(X)$, which is an isomorphism if $X$ is proper over $k$. The category 
$\DM$ is a rigid tensor triangulated category, where the duality functor is determined by the so-called \emph{motivic Poincar\'e duality}, which says that for any connected smooth $k$-variety $X$ of dimension $d$, 
\begin{equation}\label{eqn:MotPD}
M(X)^\vee\simeq M_c(X)(-d)[-2d].
\end{equation}
The Chow groups are interpreted as the corresponding Borel--Moore theory. More precisely, if $X$ is an equi-dimensional quasi-projective $k$-variety, then for any $i\in \N$,
\begin{equation}\label{eqn:ChowBM}
\CH_{i}(X)=\Hom(\1(i)[2i], M_{c}(X)).
\end{equation}

An important property we will use is the localization distinguished triangle \cite{Voe00}: let $Z$ be a closed subscheme of $X\in \operatorname{Sch}_k$, then there is a distinguished triangle in $\DM$:
\begin{equation}\label{eqn:localization}
M_{c}(Z)\to M_{c}(X)\to M_{c}(X\backslash Z)\to M_c(Z)[1].
\end{equation}

Given a mixed motive $M\in \DM$, the tensor triangulated subcategory of $\DM$ generated by $M$, denoted by $\langle M \rangle_{\DM}$ is the smallest full subcategory of $\DM$ containing $M$ that is stable under isomorphisms, direct sums,  tensor products, duality and cones (hence also shifts and direct summands). By definition, $\langle M \rangle_{\DM}$ is a pseudo-abelian rigid tensor triangulated category. Again for a smooth projective variety $X$ admitting a non-zero effective divisor, $\langle M(X) \rangle_{\DM}$ contains all Tate motives, and hence it is also stable by Tate twists.

By \cite{Voe00}, there is a fully faithful tensor functor
$$\CHM^{op} \longrightarrow \DM,$$
which sends the Chow motive $\h(X)$ of a smooth projective variety $X$ to its mixed motive $M(X)\simeq M_c(X)$; for any $i\in \mathbb{Z}$, the Tate object $\1(-i)$ in $\CHM$ is sent to $\1(i)[2i]$. In this paper we identify $\CHM^{op}$ with its essential image in $\DM$. 
\begin{question}[Elimination of cones]
If a Chow motive can be obtained from another Chow motive by performing tensor operations and cones in $\DM$, can it be obtained already within $\CHM$ by performing tensor operations therein? 
\end{question}
The following observation gives a positive answer to this question. 
 \begin{proposition}\label{prop:GeneratedCats}
Notation is as before. Let $M$ be a Chow motive. Then we have an equality of subcategories of $\DM$:
$$\langle M\rangle_{\CHM}=\langle M\rangle_{\DM}\cap \CHM.$$
 \end{proposition}

 \begin{proof}
 The argument is due to Wildeshaus \cite[Proposition 1.2]{Wil15}, which we reproduce here for the convenience of the readers. This statement is also independently discovered by Hoskins--Pepin-Lehalleur recently in \cite{HPL19}. 
In \cite{Bon10} Bondarko introduced the notion of weight structures on triangulated categories and constructs a bounded non-degenerate weight structure $w$ on $\DM$ whose heart $\DM^{w=0}$ consists of the Chow motives $\CHM^{op}$. As $\langle M\rangle_{\DM}$ is generated by the subcategory $\langle M \rangle_{\CHM}$ and the latter, being a subcategory of $\CHM$, is \emph{negative} in the sense of \cite[Definition 4.3.1]{Bon10}, we can apply \cite[Theorem 4.3.2 II]{Bon10} to conclude that there exists a unique bounded weight structure $v$ on $\langle M\rangle_{\DM}$ whose heart $\langle M\rangle_{\DM}^{v=0}$ is $\langle M\rangle_{\CHM}$. 
By shifting, we see that for any $n\in \Z$, 
\begin{equation}\label{eq:sameweight}
\langle M\rangle_{\DM}^{v=n}\subset \DM^{w=n}.
\end{equation}
 We claim that for any $n$, we have 
 \begin{equation}\label{eq:vinw}
 \langle M\rangle_{\DM}^{v\geq n}\subset \DM^{w\geq n}\cap \langle M\rangle_{\DM} \text{  and   }  \langle M\rangle_{\DM}^{v\leq n}\subset \DM^{w\leq n} \cap \langle M\rangle_{\DM}.
 \end{equation}
 Indeed, it suffices to show the first inclusion in the case $n=0$. Given any object $N$ of $\langle M\rangle_{\DM}^{v\geq 0}$, by the boundedness of $v$, $N$ can be obtained by a finite sequence of successive extensions of objects of $\langle M\rangle_{\DM}$ with non-negative $v$-weight, which have non-negative $w$-weight by \eqref{eq:sameweight}. Therefore $N\in \DM^{w\geq 0}$, and the claim is proved.\\
 Now we show that the inclusions in \eqref{eq:vinw} are actually equalities. Given $N\in \DM^{w\geq 0}\cap \langle M\rangle_{\DM}$, we have $\Hom(N, N')=0$ for all $N'\in \langle M\rangle_{\DM}^{v\leq -1}$ since, by the second inclusion in \eqref{eq:vinw}, $N'\in \DM^{w\leq -1}$. Therefore $N\in \langle M\rangle_{\DM}^{v\geq 0}$. The first equality is proved; the argument for the second one is similar. \\
 As a consequence, $\langle M\rangle_{\CHM}=\langle M\rangle_{\DM}^{v=0}=\DM^{w=0}\cap \langle M\rangle_{\DM}=\CHM\cap  \langle M\rangle_{\DM}$.
 \end{proof}
 
 Obviously, the proof shows that the same result holds if we start with a subcategory of $\CHM$ instead of just an object. The case of the subcategory of abelian motives is exactly \cite[Proposition~1.2]{Wil15}, from which we borrowed the argument above.
 
 Note that due to the abstract machinery of weight structures, the above proof does not give a constructive way to eliminate the usage of cones if a Chow motive is explicitly expressed in terms of a second one by tensor operations and cones.

\subsection{Andr\'e motives}\label{subsec:AM}
Let the base field $k$ be a subfield of the field of complex numbers $\C$.
Replacing the Chow group by the $\Q$-vector space of algebraic cycles modulo homological equivalence (here we use the rational singular homology group of the associated complex analytic space) in the construction of Chow motives (\S\ref{subsec:CHM}) one obtains the category of \emph{Grothendieck motives}, denoted $\GRM$, which comes with a canonical full functor $\CHM\to \GRM$. The category of Grothendieck motives is conjectured to be semisimple and abelian; Jannsen \cite{Jan92} showed that it is the case if and only if numerical equivalence agrees with homological equivalence, which is one of Grothendieck's standard conjectures.

The standard conjectures being difficult, in \cite{andre1996Motives} an unconditional theory was proposed by Andr\'e, refining Deligne's category of absolute Hodge motives \cite{deligne1982hodge}. He replaced in the construction of Grothendieck motives the group of algebraic cycles up to homological equivalence by the group of \emph{motivated cycles}, which are roughly speaking cohomology classes that can be obtained by using algebraic cycles and the Hodge $*$-operator. The resulting category of \emph{Andr\'e motives} is denoted by $\AM$, and it is a semisimple abelian category. The canonical faithful functor $\GRM\to \AM$ is an isomorphism if the standard conjectures hold true for all smooth projective varieties.

The virtue of $\AM$ is that it works well with the tannakian formalism. There are natural functors:
$$\SmProj^{op}\xrightarrow{\mathcal{H}}\AM\xrightarrow{r}\pHS\xrightarrow{F} \Vect_{\Q},$$ where $\mathcal{H}$ is the functor that associates to a variety its Andr\'e motive, $\pHS$ is the category of polarizable rational Hodge structures, $r$ is the Hodge realization functor, and $F$ is the forgetful functor. The composition of $r\circ \mathcal{H}$ is equal to the functor $H$ attaching to a smooth projective variety its rational cohomology group.
It is easy to see that the functors $r$ and $F$ are conservative.

\subsubsection{Mumford--Tate group and motivic Galois group}

It is well-known that $\pHS$ is a neutral tannakian semisimple abelian category with fiber functor $F$. Given a polarizable rational Hodge structure $V$, let $\langle V\rangle_{\HS}$ be the full tannakian subcategory of $\pHS$ generated by $V$. The restriction of $F$ to this subcategory is again a fiber functor. The \emph{Mumford--Tate group} of $V$, denoted by $\MT(V)$, is by definition the automorphism group of the tensor functor $F|_{\langle V\rangle_{\HS}}$ and $\langle V\rangle_{\HS}$ is equivalent to the category of representations of $\MT(V)$. Note that as $V$ is assumed to be polarizable, $\MT(V)$ is reductive.  Mumford--Tate groups are known to be always connected. The $\MT(V)$-invariants in a tensor construction on $V$ are precisely the Hodge classes of type (0,0). 
%It is sometimes convenient to use the following \textit{special Mumford--Tate group} (also called the \textit{Hodge group}) of $V$:
%$$\SMT(V):=\Ker\left(\MT(V)\to \GL(\Q(n))\right),$$
%where $n$ is the greatest common divisor of all integers $l$ such that the Tate Hodge structure $\Q(l)$ is in $\langle V\rangle_{\HS}$. Note that such $l\neq 0$ always exists when $V$ is not of pure weight 0.
%By convention, $\SMT(V)=\MT(V)$ for any weight-0 Hodge structure $V$.

In a similar fashion, $\AM$ is also neutral tannakian with fiber functor $F\circ r$. Given an Andr\'e motive $M\in \AM$, the tannakian subcategory $\langle M\rangle_{\AM}$ is again neutral tannakian, with fiber functor $F\circ r|_{\langle M\rangle_{\AM}}$; the tensor automorphism group of this functor is denoted by $\Gmot(M)$ and called the \emph{motivic Galois group} of $M$. The tannakian category $\langle M\rangle_{\AM}$ is then equivalent to the category of representations of this reductive group, and the $\Gmot(M)$-invariants in any tensor construction of $r(M)$ are precisely the motivated classes. 
%The \textit{special motivic Galois group} of $M$ is defined as follows:
%$$\SGmot(M):=\Ker\left(\Gmot(M)\to \GL(\Q(n))\right),$$ where %$n:=\gcd\{l\in \Z~\mid~ \Q(l)\in\langle M\rangle_{\AM} \}$. In this paper, the Andr\'e motives we consider always contain $\mathcal{H}^2$ of some positive dimensional smooth projective variety, hence $n=1$, due to the existence of divisors.

%\begin{remark}[$\Gmot$ \vs $\SGmot$]
%For any Andr\'e motive $M$, the special motivic Galois group $\SGmot(M)$ is characterized as the reductive subgroup of $\GL(r(M))$ whose invariants on any tensor construction of $r(M)$ are precisely the motivated classes\footnote{By definition, a motivated class is in particular of type $(0,0)$. Here by a \textit{motivated class of type} $(p,p)$, we mean a class which becomes motivated after a $p$-th Tate twist.} of type $(p,p)$ for some $p\in \Z$. Alternatively, $\SGmot(M)$ can also be defined as the tannakian group of the tensor category $\langle M\rangle_{\AM/\Q(1)}$ with the natural fiber functor, where $\AM/\Q(1)$ is the orbit category (\cf \cite{Tab15}) of $\AM$ by formally imposing that $\Q(1)\simeq \Q$.
%We have the following relation between the motivic Galois group and its special version:
%\[
%\Gmot(M) = w(\mathbb{G}_{m,\Q})\cdot \SGmot(M), 
%\]
%where $w\colon\mathbb{G}_{m,\Q}\to \Gmot(\mathcal{M})$ is the weight cocharacter.
%The same remark also holds for Mumford--Tate groups.
%\end{remark}

\subsubsection{Motivated \vs Hodge}

Let $k\subset \C$ be in addition algebraically closed. For any $M\in \AM$, as all motivated cycles are Hodge classes, the tensor invariants of the motivic Galois group are all tensor invariants of the Mumford--Tate group. Both groups being reductive, we have a canonical inclusion $\MT(r(M))\subset \Gmot(M)$, by \cite[Proposition~3.1]{deligne1982hodge}. The Hodge conjecture implies that the
 converse should hold as well. 	
	\begin{conjecture}[Hodge classes are motivated]\label{mhc}
		Let $k$ be an algebraically closed subfield of $\C$. For any $M\in\AM$, we have an equality of subgroups of $\GL\bigl(r(M)\bigr)$: $$\MT\bigl(r(M)\bigr)= \Gmot(M).$$
	\end{conjecture}
Since Mumford--Tate groups are connected, Conjecture \ref{mhc} predicts in particular that $\Gmot(M)$ should also be connected; already this statement is a difficult open problem.

The most significant evidence to this conjecture is Andr\'e's result in \cite{andre1996Motives} saying that on abelian varieties, all Hodge classes are motivated, strengthening the previous result of Deligne \cite{deligne1982hodge} on absolute Hodge classes. Let us state the result in the following form:
\begin{theorem}[{\cite[Theorem~0.6.2]{andre1996Motives}}] \label{mhcabelian}
		Conjecture~\ref{mhc} holds for any abelian Andr\'e motives. More precisely, over an algebraically closed field $k\subset \C$, for any $M\in \AM^{ab}$, the rigid tensor subcategory of $\AM$ generated by the motives of abelian varieties,  we have $\MT\bigl(r(M)\bigr)= \Gmot(M)$.		
\end{theorem}

%\begin{remark}[Mumford--Tate conjecture]
%Fixing a prime number $\ell$, the Tate conjecture implies that any Tate class in the $\ell$-adic cohomology is a  motivated cycle (with $\Q_{\ell}$-coefficient). As an analogue of Conjecture \ref{mhc}, we can formulate the \emph{motivated Tate conjecture} which says that the neutral component of the Zariski closure of the image of the Galois representation given by the $\ell$-adic cohomology, which is \emph{a priori} a subgroup of the $\ell$-adic motivic Galois group $\Gmot_{, \ell}$, is actually equal to it. The validity of both motivated Hodge and Tate conjectures implies the famous Mumford--Tate conjecture. 
%\end{remark}

\subsection{Relative Andr\'e motives and monodromy: proper setting}\label{subsubsec:familiesAM}

Another remarkable aspect of Andr\'{e} motives is their behaviour under deformations. The results presented below are essentially due to Andr\'e (based on Deligne \cite{deligne1971theorie}) in the projective setting and formalized by Moonen \cite[\S4]{moonen17}. We generalize these results to the proper setting. Let $k$ be an uncountable and algebraically closed subfield of $\C$. The starting point is the following observation.

\begin{lemma}\label{rmk:observation}
The contra-variant functor $\mathcal{H}\colon \SmProj_k\to \AM$ extends naturally to the category $\operatorname{SmProp}_k$ of smooth proper varieties. 
\end{lemma}
\begin{proof}
Let $X$ be a smooth and proper (non-necessarily projective) algebraic variety defined over $k$. 
%Then $X$ has a well-defined Andr\'{e} motive $\mathcal{H}(X)=\bigoplus_i \mathcal{H}^{i}(X)$, with the expected properties. 
Consider its Nori motive $\mathcal{H}_{\mathrm{Nori}}(X)=\bigoplus_i\mathcal{H}_{\mathrm{Nori}}^{i}(X)$. For each $i\in \NN$, $\mathcal{H}_{\mathrm{Nori}}^{i}(X)$ carries a weight filtration $W_\bullet$, inducing the weight filtration on its Hodge realization \cite[Theorem 10.2.5]{HKMS}. In particular,
$$r\left(\Gr^W_l\mathcal{H}^i_{\operatorname{Nori}}(X)\right)=\Gr^W_lH^i(X).$$ However, the Hodge structure $H^{i}(X)$ is pure\footnote{This can be easily seen in the following way: by Chow's lemma, one can find a  blow-up $\tilde X \to X$ with $\tilde X$ smooth and projective. Then by the projection formula, $H^i(X)$ is a direct summand of the pure Hodge structure $H^i(X)$, hence is also pure.}, $\Gr^W_lH^i(X)$ is zero for all $l\neq 0$. By the conservativity of $r$, the Nori motive $\Gr^W_l\mathcal{H}^i_{\operatorname{Nori}}(X)$ is also trivial for $l\neq 0$. In other words, $\mathcal{H}^i_{\operatorname{Nori}}(X)$ is pure. We conclude by invoking Arapura's theorem \cite{Ara13} which says that the category of pure Nori motives is equivalent to the category of Andr\'{e} motives.
\end{proof}

The following result generalizes Andr\'{e}'s \emph{deformation principle} for motivated cycles \cite[Th\'{e}or\`{e}me 0.5]{andre1996Motives} to the proper setting (but always with projective fibers). It has been obtained recently by Soldatenkov \cite[Proposition 5.1]{soldatenkov19}. We include here an alternative proof with the point being that Andr\'{e}'s original proof actually works, when combined with Lemma \ref{rmk:observation}.

\begin{theorem}[Andr\'e--Soldatenkov]\label{thm:DefPrinciple}
Let $S$ be a connected and reduced variety and let $f\colon \mathcal{X} \to S$ be a proper smooth morphism with projective fibers. Let $\xi \in H^{0}(S, R^{2i}f_* \Q (i))$, and assume that there exists $s_0\in S$ such that the restriction $\xi_{s_0}\in H^{2i}(X_{s_0},\Q(i))$ of $\xi$ to the fibre over $s_0$ is motivated. Then, for all $s\in S$, the class $\xi_{s}\in H^{2i}(X_s, \Q(i))$ is motivated.
	\end{theorem} 
\begin{proof}
    As in \cite{andre1996Motives}, we can assume that $S$ is a smooth affine curve. Choose a smooth compactification $\bar{\mathcal{X}}$ of the total space $\mathcal{X}$ and let $j_s\colon X_s\to \bar{\mathcal{X}}$ be the inclusion morphism for all $s\in S$. The theorem of the fixed part \cite[4.1.1]{deligne1971theorie} ensures that the image of the morphism of Hodge structures $j_s^*\colon H^{2i}(\bar{\mathcal{X}}, \Q(i))\to H^{2i}(X_s, \Q(i))$ coincides with the subspace of monodromy invariants. Andr\'{e}'s proof uses the morphism $j^*_s$ induced on Andr\'{e} motives, and conclude that the subspace of monodromy invariants at $s\in S$ is a submotive which does not depend on the chosen point. Now, in our case $\bar{\mathcal{X}}$ is not necessarily projective, but still has a well-defined Andr\'{e} motive $\mathcal{H}(\bar{\mathcal{X}})=\bigoplus_i \mathcal{H}^{i}(\bar{\mathcal{X}})$ by Lemma \ref{rmk:observation}, and $j_s^*$ is a morphism of Andr\'{e} motives. Then we can conclude via the same argument as in Andr\'{e} \cite{andre1996Motives}.
\end{proof}

The following definition extends slightly the usual notion of families of Andr\'e motives.
\begin{definition}[\cf {\cite[Definition 4.3.3]{moonen17}}]
\label{def:RelativeAM}
Let $S$ be a smooth connected quasi-projective variety. An \textit{Andr\'e motive} (\resp \textit{generalized Andr\'e motive}) over $S$ is a triple $(\mathcal{X}/S, e, n)$ with
\begin{itemize}
    \item $f\colon\mathcal{X}\to S$ a smooth projective (\resp proper) morphism with connected projective fibers,
    \item $e$ a global section of $R^{2d}(f\times f)_*\Q_{\mathcal{X}\times_S\mathcal{X}}(d)$, where $d$ is the relative dimension of $f$,
    \item $n$ an integer,
\end{itemize}
such that for some $s\in S$ (or equivalently by Theorem \ref{thm:DefPrinciple}, for any $s\in S$), the value $e(s)\in H^{2d}(X_s\times X_s, \Q(d))$ is a motivated projector.
\end{definition}

These objects, with morphisms defined in the usual way, form a tannakian semisimple abelian category denoted by $\AM(S)$ (\resp $\tilde\AM(S)$). Obviously, a generalized Andr\'e motive over a point is nothing else but an Andr\'e motive introduced before. There is a natural realization functor from the category of generalized Andr\'e motives over $S$ to the tannakian category of \textit{algebraic} variations\footnote{A variation of $\Q$-Hodge structures over $S$ is called \textit{algebraic} if the restriction to some non-empty Zariski open subset $U$ of $S$ is a direct summand of a variation of the form $R^if_*\Q(j)$ for some smooth projective morphism $f\colon \mathcal{X}\to U$ and some integer $j$.} of $\Q$-Hodge structures in the sense of Deligne \cite[Definition 4.2.4]{deligne1971theorie}: $$
\AM(S)\subset \tilde\AM(S)\xrightarrow{r}\operatorname{VHS}_{\Q}^{\operatorname{a}}(S)\subset \operatorname{VHS}_{\Q}^{\operatorname{pol}}(S).$$

By construction, for any smooth proper morphism $f\colon \mathcal{X}\to S$ with projective fibers and any integer $i$, we have a generalized Andr\'{e} motive $\mathcal{H}^i(\mathcal{X}/S)$ whose realization is $R^{i}f_*\Q\in \operatorname{VHS}_{\Q}^{\operatorname{a}}(S)$.

%More generally, by a family of motives $M/S$ over $S$ we will mean a sum of tensor products, duals and Tate twists of families arising as above from a smooth proper morphism with projective fibers. The Hodge realization $r(M)/S$ is an algebraic variation of $\Q$-Hodge structures in the sense of \cite[Definition 4.2.4]{deligne1971theorie}.
Given a (generalized) Andr\'e motive $M/S\in \tilde\AM(S)$, we aim to study the variation of motivic Galois groups $\Gmot(M_s)$ and Mumford--Tate groups $\MT(r(M)_s)$ when $s$ varies in $S$.
Consider the monodromy representation $\pi_1(S,s)\to \GL(r(M)_s)$ associated to the local system underlying the realization of $M/S$. The \textit{algebraic monodromy group} at a point $s\in S$, denoted by $\mathrm{G}_{\mathrm{mono}}(M/S)_s$, is defined as the Zariski closure in $\GL(r(M)_s)$ of the image of the monodromy representation. It is not necessarily connected, but it becomes so after some finite \'{e}tale cover of $S$; Deligne \cite[Theorem 4.2.6]{deligne1971theorie} proved that $\mathrm{G}_{\mathrm{mono}}(M/S)_s^0 $ is a semisimple $\Q$-algebraic group. The variation of these groups with $s$ determines a local system of algebraic groups $\mathrm{G}_{\mathrm{mono}}(M/S)$.

\begin{theorem}[\cf {\cite[\S 4.3]{moonen17}} ]\label{thm:familiesMotives}
Let $S$ be as above and let $M/S$ be a generalized Andr\'{e} motive over $S$. There exists two local systems of reductive algebraic groups $\MT(r(M)/S)$ and $\Gmot(M/S)$ over~$S$ with the following properties:
\begin{enumerate}[(i)]
  	\item we have inclusions of local systems of algebraic groups: $$\mathrm{G}_{\mathrm{mono}}(M/S)^0\subset \MT(r(M)/S)\subset \Gmot(M/S)\subset \GL(r(M)/S);$$
  	\item for a very general (\ie, outside of a countable union of closed subvarieties of $S$) point $s\in S$, we have $\MT(r(M)_s)=\MT(r(M) /S)_s$ and $\Gmot(M_s)=\Gmot(M/S)_s$;
  	\item for all $s\in S$, we have $\MT(r(M)_s) \subset \MT(r(M)/S)_s$ and  $ \Gmot(M_s)\subset \Gmot(M/S)_s$;
  	%and equality holds if and only if $\mathrm{G}_{\mathrm{mono}}(M/S)_s^0\subset \MT(r(M)_s)$ (\resp $\mathrm{G}_{\mathrm{mono}}(M/S)_s^0\subset \Gmot(M_s)$).
  	\item for all $s\in S$, we have
  	\[
	\mathrm{G}_{\mathrm{mono}}(M/S)^0_s \cdot \MT(r(M)_s) = \MT(r(M)/S)_s  \text{ and }
	\mathrm{G}_{\mathrm{mono}}(M/S)^0_s \cdot \Gmot(M_s) = \Gmot(M/S)_s.
	\]	
	In particular, each of the inclusion in $(iii)$ is an equality if and only if $\mathrm{G}_{\mathrm{mono}}(M/S)_s^0$ is contained respectively in $\MT(r(M)_s)$  and  $\Gmot(M_s)$.
  \end{enumerate}
  The local system $\MT(r(M)/S)$ is called the generic Mumford--Tate group of $r(M)/S$, and $\Gmot(M/S)$ is called the generic motivic Galois group of $M/S$.
\end{theorem}
\begin{proof}
There exists a non-empty Zariski open subset $U\subset S$ such that the restriction of $M/S$ to $U$ is an Andr\'e motive over~$U$. The desired conclusions hold for the restricted family over~$U$ by Theorems 4.1.2, 4.1.3, 4.3.6, and 4.3.9 in Moonen's survey \cite{moonen17}; hence, we get two local systems of algebraic groups over $U$ with the properties above.
The fundamental group of $S$ is a quotient of that of $U$. Since $(i)$ holds over $U$, we can extend the generic Mumford--Tate and motivic Galois groups which we have over $U$ to local systems $\MT(r(M)/S)$ and $\Gmot(M/S)$ over $S$. We prove that these local systems satisfy the desired properties. Note that $(i)$ and $(ii)$ are immediate since both conditions can be checked over $U$, where we already know they hold.\\
%can be checked generically therefore follow from the fact that they holds over $U$.\\
$(iii)$. We only give the proof for the generic motivic Galois group; the argument for the generic Mumford--Tate group is similar.
Up to a base change of the family $M/S$ by a finite \'{e}tale cover of~$S$, we may assume that the algebraic monodromy group is connected. Let $s_0\in S$ be any point such that $\mathrm{G}_{\mathrm{mono}}(M/S)_{s_0}$ is contained in $\Gmot(M_{s_0})$; this is the case for a very general point, by $(i)$ and $(ii)$. The monodromy group acts on $\Gmot(M_{s_0})$ by conjugation, and this defines a local system of algebraic groups $\Gmot(M_{s_{0}}/S)$ with fiber isomorphic to the motivic Galois group at the point $s_0$.
Consider any tensor construction $T/S=(M/S)^{\otimes m}\otimes (M/S)^{\vee, \otimes n}$, and let $\xi_{s_0}$ be the cohomology class of a motivated cycle in $r(T)_{s_0}$. The class $\xi_{s_0}$ is monodromy invariant, and therefore it extends to a global section $\xi$ of the local system underlying $r(T)/S$. By Theorem~\ref{thm:DefPrinciple} the restriction $\xi_{s}$ is motivated for any $s\in S$. By the reductivity of the groups involved, we deduce that for any $s\in S$ we have $\Gmot(M_s)\subset \Gmot(M_{s_0}/S)_{s}$, and we conclude by $(ii)$ that the latter must be equal to $\Gmot(M/S)_s$. This proves $(iii)$ and that if $\mathrm{G}_{\mathrm{mono}}(M/S)_{s}\subset\Gmot(M_{s})$ then $\Gmot(M_{s})=\Gmot(M/S)_s$. \\
%Thanks to $(i)$ and $(ii)$, we conclude that we must have an equality of local systems of groups $\Gmot(M_{s_0}/S)=\Gmot(M/S)$. It follows that for any $s\in S$ such that $\mathrm{G}_{\mathrm{mono}}(M/S)_s\subset \Gmot(M_s)$ we have $\Gmot(M_s)=\Gmot(M/S)_s$.
%On the other hand, if this last equality holds, part $(i)$ implies that the algebraic monodromy group is contained in $\Gmot(M_s)$.\\
$(iv)$. By $(i)$ and $(iii)$, we clearly have $\mathrm{G}_{\mathrm{mono}}(M/S)^0_s \cdot \Gmot(M_s) \subset \Gmot(M/S)_s$. Since both sides are reductive, we only need to compare their invariants on the tensor constructions $T/S$ on $M/S$ as above. If $\xi_s\in r(T)_s$ is invariant for the action of $\mathrm{G}_{\mathrm{mono}}({M}/S)^0_s \cdot \Gmot({M}_s)$, then it is the class of a motivated cycle which is monodromy invariant. By Theorem \ref{thm:DefPrinciple}, it extends to a global section $\xi$ of $r(T)/S$ such that $\xi_{s'}$ is motivated at any $s'\in S$. It follows that $\xi_s$ is invariant for $\Gmot({M}/S)_s$. The proof of the assertion regarding the Mumford--Tate group is similar.
\end{proof}

\subsection{Relations}
We summarize in the diagram below the natural functors relating the various categories of motives we discussed above. For the sake of completeness, we inserted in the diagram also Nori's category of mixed motives $\operatorname{MM_{Nori}}$, whose pure part is the abelian category of Andr\'e motives by Arapura's result in \cite{Ara13}, see also \cite{HKMS} for a recent account.

\begin{equation}
\xymatrix{
\SmProj^{op}_{k}\ar@{^{(}->}[r] \ar@{^{(}->}[dddd]\ar[dr]_{\h}&\operatorname{SmProp}_k^{op}\ar[dr]_{\mathcal{H}}\ar[drr] \ar[drrr]_-{H^{*}}&&&\\
&\CHM\ar[r]\ar@{^{(}->}[dd]&\AM\ar[r]_{r}\ar@{^{(}->}[d]&\pHS\ar[r]_{F}\ar@{^{(}->}[d]&\Vect_{\Q}\ar@{=}[d]\\
&&\operatorname{MM_{Nori}}\ar[r]^r&\operatorname{MHS_{\mathbf{Q}}}\ar[r]^{F}&\Vect_{\Q}\\
&\DM^{op}\ar[r]^-C &D^{b}(\operatorname{MM_{Nori}})\ar[u]_{\oplus H^i}\ar[r]^r&D^{b}(\operatorname{MHS_{\mathbf{Q}}})\ar[r]^{F}\ar[u]_{\oplus H^i}& D^{b}(\Vect_{\Q})\ar[u]_{\oplus H^i}\\
\operatorname{Sch}^{op}_{k}\ar[ur]^{M}\ar[urr]\ar[urrr]\ar[urrrr]&&&&
}
\end{equation}
Here the comparison functor $C$ is due to Harrer \cite[Theorem 7.4.17]{Har16}.

\section{Motives of the stable loci of moduli spaces}\label{sec:StableModuli}
In this section, we generalize an argument of B\"ulles \cite{Bue18}  to give a relationship between the motive of the (in general quasi-projective) moduli space of stable sheaves on a K3 or abelian surface and the motive of the surface. 

Let $S$ be a projective K3 surface or abelian surface. Denote by $\widetilde{\operatorname{NS}}(S)=H^0(S, \mathbb{Z})\oplus \operatorname{NS}(S)\oplus H^4(S, \mathbb{Z})$ the algebraic Mukai lattice, equipped with the following Mukai pairing: for any $\vv=(r, l, s)$ and $\vv'=(r, l, s') $ in $\widetilde{\NS}(S)$, $$\langle \vv, \vv'\rangle:= (l, l')-rs'-r's\in \Z.$$ Given a Brauer class $\alpha$, a Mukai vector  $\vv\in \widetilde{\operatorname{NS}}(S)$ with $\vv^2 \geq 0$ and a Bridgeland stability condition $\sigma$ of the $\alpha$-twisted derived category $D^{b}(S, \alpha)$, let  $\cM^{\mathrm{st}}$ be the moduli space of $\sigma$-stable objects in $D^{b}(S, \alpha)$ with Mukai vector $\vv$.  By \cite{Muk84}, $\cM^{\mathrm{st}}$ is a smooth quasi-projective holomorphic symplectic variety of dimension $2m:=\vv^2+2$. To understand the (mixed) motive of $\cM^\mathrm{st}$, let us first recall the following result of Markman, extended by Marian--Zhao.

\begin{theorem}[\cite{Mar02} \cite{Mar07} \cite{MZ17}]
    \label{thm:diag-chern-k3}
    Let $\mathcal{E}$ and $\mathcal{F}$ be two (twisted) universal families over $\cM^\mathrm{st}\times S$. Then
    $$ \Delta_{\cM^\mathrm{st}} = c_{2m}(- \sExt^!_{\pi_{13}}(\pi_{12}^\ast(\mathcal{E}), \pi_{23}^\ast(\mathcal{F}))) \in \CH^{2m}(\cM^\mathrm{st} \times \cM^\mathrm{st}), $$
    where $2m$ is the dimension of $\cM^\mathrm{st}$ and $\sExt^!_{\pi_{13}}(\pi_{12}^\ast(\mathcal{E}), \pi_{23}^\ast(\mathcal{F}))$ denotes the class of the complex $R\pi_{13,*}(\pi_{12}^{*}(\mathcal{E})^{\vee}\otimes^{\mathbb{L}}\pi_{23}^{*}(\mathcal{F}))$ in the Grothendieck group of $\cM^\mathrm{st} \times \cM^\mathrm{st}$, where $\pi_{ij}$'s are the natural projections from $\cM^{\mathrm{st}}\times S\times \cM^{\mathrm{st}}$.
\end{theorem}

\begin{proof}[Pointer to references]
    For the case of Gieseker-stable sheaves,  \cite[Theorem 1]{Mar02} states the result for the cohomology class, but the proof gives the equality in Chow groups. Indeed, in \cite[Theorem 8]{Mar07}, the statement is for Chow groups. Moreover, the assumption on the existence of a universal family can be dropped (\cite[Proposition 24]{Mar07}): it suffices to replace in the formula the sheaves $\mathcal{E}$ and $\mathcal{F}$ by certain universal classes in the Grothendieck group $K_0(S\times \cM^\mathrm{st})$ constructed  in \cite[Definition 26]{Mar07}. More recently, it is shown in \cite{MZ17} that the technique of Markman can be adapted to obtain the result in the full generality as stated.
\end{proof}

As a consequence, we can obtain the following analogue of \cite[(3), p.6]{Bue18}

\begin{proposition}[Decomposition of the diagonal]
    \label{prop:diag-prod-k3}
   There exist finitely many integers $k_i$ and cycles $\gamma_i \in \CH^{e_i}(\cM^\mathrm{st} \times S^{k_i})$, $\delta_i \in \CH^{d_i}(S^{k_i} \times \cM^\mathrm{st})$, such that
    $$ \Delta_{\cM^{\mathrm{st}}} = \sum \delta_i \circ \gamma_i \in \CH^{2m}(\cM^\mathrm{st}\times \cM^\mathrm{st}),$$ 
    here $\dim \cM^{\mathrm{st}}=2m=e_{i}+d_{i}-2k_{i}$ for all $i$.
\end{proposition}

\begin{proof}
    We follow the proof of \cite[Theorem 1]{Bue18}. First of all, we observe that by Lieberman's formula (see
	\cite[\S 3.1.4]{andre} and \cite[Lemma~3.3]{Via17} for a proof), the following two-sided ideal of $\CH^*(\cM^\mathrm{st} \times \cM^\mathrm{st})$ (with respect to the ring structure given by the composition of correspondences)
    $$ I = \langle \beta \circ \alpha \mid \alpha \in \CH^\ast(\cM^\mathrm{st} \times S^k), \beta \in \CH^\ast(S^k \times \cM^\mathrm{st}), k \in \NN \rangle \subseteq \CH^\ast(\cM^\mathrm{st} \times \cM^\mathrm{st}) $$
    is closed under the intersection product, hence is a $\Q$-subalgebra of $\CH^\ast(\cM^\mathrm{st} \times \cM^\mathrm{st})$. A computation similar to \cite[(2), p.6]{Bue18} using the Grothendieck--Riemann--Roch theorem shows that
    $$ \ch(-[\sExt_{\pi_{13}}^!(\pi_{12}^\ast(\mathcal{E}), \pi_{23}^\ast(\mathcal{F}))]) = - (\pi_{13})_\ast(\pi_{12}^\ast\alpha \cdot \pi_{23}^\ast\beta) $$
    where
    $$ \alpha = \ch(\mathcal{E}^\vee) \cdot \pi_2^\ast\sqrt{\td(S)} \quad\text{and}\quad \beta = \ch(\mathcal{F}) \cdot \pi_2^\ast\sqrt{\td(S)}. $$
    It follows that $\ch_n(-[\sExt_{\pi_{13}}^!(\pi_{12}^\ast(\mathcal{E}), \pi_{23}^\ast(\mathcal{F}))]) \in I$ for any $n \in \NN$. An induction argument then shows that $c_n(-[\sExt_{\pi_{13}}^!(\pi_{12}^\ast(\mathcal{E}), \pi_{23}^\ast(\mathcal{F}))]) \in I$ for each $n \in \NN$. 
    In particular, combined with Theorem \ref{thm:diag-chern-k3},  $\Delta_{\cM^{\mathrm{st}}}$ is in $I$, which is equivalent to the conclusion.
\end{proof}

In terms of mixed motives, one can reformulate Proposition \ref{prop:diag-prod-k3} as follows.

\begin{corollary}[Factorization of the comparison map]
In the category $\DM$ of mixed motives, the canonical comparison morphism $M(\cM^\mathrm{st})\to M_c(\cM^\mathrm{st})$ can be factorized as the following composition:
$$M(\cM^\mathrm{st})\to \bigoplus_i M(S^{k_i})(e_i-2k_i)[2e_i-4k_i]\to M_c(\cM^\mathrm{st}),$$
for finitely many integers $k_{i}$'s and $e_{i}$'s.
\end{corollary}
\begin{proof}
It is enough to remark that by \eqref{eqn:MotPD} and \eqref{eqn:ChowBM}, for any $j\in \Z$, the space $\CH^{j}(\cM^{\mathrm{st}}\times S^{k_{i}})$ is equal to the space $$\Hom_{\DM}(M(\cM^{\mathrm{st}}), M(S^{k_{i}})(j-2k_{i})[2j-4k_{i}])$$ as well as to the space $$\Hom_{\DM}(M(S^{k_{i}})(2m-j)[4m-2j], M_{c}(\cM^{\mathrm{st}})).$$
\end{proof}

\begin{remark}[Hodge realization]
In Proposition \ref{prop:diag-prod-k3}, if one denotes $\gamma = \oplus \gamma_i$ and $\delta = \oplus \delta_i$, then we get the following morphisms of mixed Hodge structures. 
$$ H_c^\ast(\cM^\mathrm{st}) \stackrel{\gamma}{\longrightarrow} \bigoplus_{i} H^\ast(S^{k_i})(2k_i-e_i) \stackrel{\delta}{\longrightarrow} H^\ast(\cM^\mathrm{st}), $$
where the composition is precisely the comparison morphism from the compact support cohomology to the usual cohomology.
\end{remark}

%\begin{remark}[Better bound on the powers]
%As is explained in \cite{Bue18} and can also be viewed in the proof of Proposition \ref{prop:diag-prod-k3}, the integers $k_{i}$ can be bounded by $2m$, the dimension of of $\cM^{\mathrm{st}}$. The authors are recently informed by Robert Laterveer that when $S$ is a K3 surface and $\cM^{\mathrm{st}}$ is projective, the integers $k_{i}$'s can in fact be bounded by $m$.
%\end{remark}

\begin{remark}[Challenge for Kummer moduli spaces]\label{rmk:ChallengeKummerst}
In the case that $S$ is an abelian surface, the moduli space $\cM^\mathrm{st}$ is isotrivially fibered over $S\times  \hat{S}$ (which is the Albanese fibration when $\cM^\mathrm{st}$ is projective). We usually denote by $\mathcal{K}^{\mathrm{st}}:=\mathcal{K}^{\mathrm{st}}_{S, H}(\vv)$ its fiber. The analogue of Theorem \ref{thm:diag-chern-k3} seems to be unknown for $\mathcal{K}^{\mathrm{st}}$.
\end{remark}

\section{Motive of O'Grady's moduli spaces and their resolutions}\label{sec:OG10}

In this section, we study the motive of O'Grady's  $10$-dimensional hyper-K\"ahler varieties \cite{O'G99}. Those are symplectic resolutions of certain singular moduli spaces of sheaves on K3 or abelian surfaces. We first recall the construction.

\subsection{Symplectic resolution of the singular moduli space}\label{subsec:Resolution}

Let $S$ be a projective K3 surface or abelian surface, let $\alpha$ be a Brauer class, and let $\vv = 2 \vv_0$ be a Mukai vector, such that $\vv_0\in\widetilde{\NS}(S)$ is primitive with $\vv_0^2=2$. Let $\sigma$ be a $\vv_{0}$-generic stability condition on the $\alpha$-twisted derived category $D^{b}(S, \alpha)$ (for example, a $\vv_{0}$-generic polarization). We write
$$ \cM^\mathrm{st} = \cM_{S, \sigma}(\vv, \alpha)^\mathrm{st} $$
for the smooth and quasi-projective moduli space of $\sigma$-stable objects in $D^{b}(S, \alpha)$ with Mukai vector $\vv$, and
$$ \cM = \cM_{S, \sigma}(\vv, \alpha)^\mathrm{ss} $$
for the (singular) moduli space of $\sigma$-semistable objects with the same Mukai vector. In \cite{O'G99}, O'Grady constructed a symplectic resolution 
$\tilde{\cM}$ of $\cM$ (see also \cite{KLS06}), which is a projective (irreducible if $S$ is a K3 surface) holomorphic symplectic manifold  of dimension 10, not deformation equivalent to the fifth Hilbert schemes of the surface $S$. We know that these hyper-K\"ahler varieties are all deformation equivalent \cite{PR13}.

Let us briefly recall the geometry of $\cM$. We follow the notations in \cite{O'G99}, see also \cite{LS06} and \cite[\S 2]{MRS18}. The moduli space $\cM$ admits a filtration
$$ \cM \supset \Sigma \supset \Omega $$
where 
$$\Sigma = \Sing(\cM) = \cM \setminus \cM^{\mathrm{st}} \cong \Sym^2(\cM_{S,\sigma}(\vv_0, \alpha))$$
is the singular locus of $\cM$, which consists of strictly $\sigma$-semistable objects; and 
$$\Omega = \Sing(\Sigma) \cong \cM_{S,\sigma}(\vv_0, \alpha)$$ 
is the singular locus of $\Sigma$, hence the diagonal in $\Sym^2(\cM_{S,\sigma}(\vv_0, \alpha))$. Notice that $\cM_{S,\sigma}(\vv_0, \alpha)$ is a smooth projective holomorphic symplectic fourfold deformation equivalent to the Hilbert squares of $S$.

In \cite{O'G99}, O'Grady produced a symplectic resolution $\tilde{\cM}$ of $\cM$ in three steps. As the explicit geometry is used in the proof of our main result, we briefly recall his construction.

\textsc{Step 1.} We blow up $\cM$ along $\Omega$, resulting a space $\bar{\cM}$ with an exceptional divisor $\bar{\Omega}$. The only singularity of $\bar{\cM}$ is an $A_1$-singularity along the strict transform $\bar{\Sigma}$ of $\Sigma$. In fact, $\bar{\Sigma}$ is smooth, satisfying
$$\bar{\Sigma} \cong \Hilb^2(\cM_{S,\sigma}(\vv_0, \alpha)),$$ 
with the morphism $\bar{\Sigma} \to \Sigma$ being the corresponding Hilbert-Chow morphism, whose exceptional divisor is precisely the intersection of $\bar{\Omega}$ and $\bar{\Sigma}$ in $\bar{\cM}$.

\textsc{Step 2.} We blow up $\bar{\cM}$ along $\bar{\Sigma}$ to obtain a (non-crepant) resolution  $\hat{\cM}$ of $\cM$. The exceptional divisor $\hat{\Sigma}$ is thus a $\PP^1$-bundle over $\bar{\Sigma}$. We denote by $\hat{\Omega}$ the strict transform of $\bar{\Omega}$. Then $\hat{\cM}$ is a smooth projective compactification of $\cM^\mathrm{st}$, with boundary 
$$ \partial\hat{\cM} = \hat{\cM} \setminus \cM^\mathrm{st} = \hat{\Omega} \cup \hat{\Sigma} $$ 
being the union of two smooth hypersurfaces which intersect transversally.

\textsc{Step 3.} Lastly, an extremal contraction of $\hat{\cM}$ contracts $\hat{\Omega}$ as a $\PP^2$-bundle to $\tilde{\Omega}$, which is a 3-dimensional quadric bundle (more precisely, the relative Lagrangian Grassmannian fibration associated to the tangent bundle) over $\Omega$. The space obtained is denoted by $\tilde{\cM}$, which is shown to be a symplectic resolution of $\cM$.

\begin{remark}
By the main result of Lehn--Sorger \cite{LS06}, O'Grady's symplectic resolution can also be obtained by a single blow-up of $\cM$ along its (reduced) singular locus $\Sigma$. The exceptional divisor $\tilde\Sigma$ is nothing else but the image of $\hat{\Sigma}$ under the contraction in the third step described above, which is singular along $\tilde{\Omega}$, the preimage of $\Omega$. If we blow up $\tilde{\cM}$ along $\tilde{\Omega}$, we will obtain again  $\hat{\cM}$, with the exceptional divisor being $\hat{\Omega}$ and the strict transform of $\tilde\Sigma$ being $\hat\Sigma$. In short, the order of blow-ups can be ``reversed"; see the following commutative diagram from \cite[\S 2]{MRS18}:

\begin{equation*}
    \xymatrix{
    &\Bl_{\tilde \Omega}\tilde{\cM}=\hat\cM=\Bl_{\bar\Sigma}\bar{\cM}\ar[dl]\ar[dr]&\\
    \tilde{\cM}=\Bl_\Sigma \cM\ar[dr]&&\bar{\cM}=\Bl_\Omega \cM\ar[dl]\\
    &\cM&.
    }
\end{equation*}
\end{remark}

% \subsection{Formulas for computing motives}

% We now collect some useful formulas for computing motives, that will be used later.

% \begin{lemma}
%     \label{lem:quadric-bundle}
%     Let $X$ be a smooth projective variety, and $f: Y \to X$ be a smooth quadric bundle of relative dimension $2m-1$ for some positive integer $m$. Then
%     $$ \h(Y) = \oplus_{i=0}^{2m-1} \h(X)(-i). $$
% \end{lemma}

% \begin{proof}
%     Todo... Use \cite[Theorem 4.2]{Via13}, or simply \cite[Remark 4.6]{Via13}.
% \end{proof}

\subsection{The motive of O'Grady's resolution}

We will compute the Chow motives of the boundary components of $\hat{\cM}$, then describe the Chow motives of the resolutions $\hat{\cM}$ and $\tilde{\cM}$. We start with the following observation.

\begin{lemma}
    \label{lem:Hilb2}
    Let $X$ be a smooth projective variety. The Chow motive $\h(\Hilb^2(X))$ belongs to $\langle \h(X)\rangle_{\CHM}$, the pseudo-abelian tensor subcategory of $\CHM$ generated by $\h(X)$.
\end{lemma}

\begin{proof}
    We assume $\dim X = n$. Let $\Delta_X \subseteq X \times X$ be the diagonal, then by \cite[\S 9]{Man68}, we have
    $$ \h \left( \Bl_{\Delta_X}(X \times X) \right) = \h(X^2) \oplus \left( \oplus_{i=1}^{n-1} \h(X)(-i) \right). $$
    Since $\Hilb^2(X) = \Bl_{\Delta_X}(X \times X) / \ZZ_2$, its motive is the $\ZZ_2$-invariant part
    $$ \h(\Hilb^2(X)) = \h \left( \Bl_{\Delta_X}(X \times X) \right)^{\ZZ_2} $$
    which is a direct summand  of $\h \left( \Bl_{\Delta_X}(X \times X) \right)$, hence is contained in the desired subcategory.
\end{proof}

\begin{lemma}
    \label{lem:motive_hathat}
    The Chow motives $\h(\hat{\Sigma})$,  $\h(\hat{\Omega})$ and $\h(\hat\Sigma\cap \hat\Omega)$ are all contained in the subcategory $\langle\h(S)\rangle_{\CHM}$.
\end{lemma}

\begin{proof}
    By O'Grady's construction, $\hat{\Sigma}$ is a $\PP^1$-bundle over $\bar{\Sigma} \cong \Hilb^2(\cM_{S,\sigma}(\vv_0, \alpha))$. It follows from \cite[\S 7]{Man68} that
    $$ \h(\hat{\Sigma}) = \h(\bar{\Sigma}) \oplus \h(\bar{\Sigma})(-1).  $$
    By \cite[Theorem 0.1]{Bue18}, $\h(\cM_{S, \sigma}(\vv_0, \alpha))$ is in the tensor subcategory of Chow motives generated by $\h(S)$. It follows by Lemma \ref{lem:Hilb2} that $\h(\bar{\Sigma})$ is also in this subcategory, therefore so is $\h(\hat{\Sigma})$.
    
    Again by O'Grady's construction, $\hat{\Omega}$ is a $\PP^2$-bundle over $\tilde{\Omega}$. It follows that
    $$ \h(\hat{\Omega}) = \h(\tilde{\Omega}) \oplus \h(\tilde{\Omega})(-1) \oplus \h(\tilde{\Omega})(-2). $$
    Moreover, since $\tilde{\Omega}$ is a $3$-dimensional quadric bundle over $\Omega$, by \cite[Remark 4.6]{Via13},
    %\footnote{As we cannot find a reference with proof of this simple result, let us sketch the argument. Thanks to the localization sequence and five lemma, by stratifying the base scheme, we are reduced to the case where the fibration is a trivial product with a smooth quadric. Now it follows from the fact that quadrics are homogeneous and in particular cellular.} 
    we have that
    $$ \h(\tilde{\Omega}) = \h(\Omega) \oplus \h(\Omega)(-1) \oplus \h(\Omega)(-2) \oplus \h(\Omega)(-3). $$
    Since $\Omega \cong \cM_{S,\sigma}(\vv_0, \alpha)$, it follows by \cite[Theorem 0.1]{Bue18} that $\h(\Omega)$ is in the thick tensor subcategory of Chow motives generated by $\h(S)$, hence the same is true for $\h(\tilde{\Omega})$ and $\h(\hat{\Omega})$.
    
    Similarly, the intersection $\hat\Sigma\cap \hat\Omega$ is a smooth conic bundle over $\tilde\Omega$, again by \cite[Remark 4.6]{Via13}, its motive is in the tensor subcategory generated by that of $\tilde\Omega$. One concludes as for $\hat\Omega$.
\end{proof}
Here comes the key step of the proof.
\begin{proposition}
    \label{prop:motive_hatM}
    The Chow motive $\h(\hat{\cM})$ belongs to $\langle \h(S)\rangle_{\CHM}$.
\end{proposition}
We give two proofs with the same starting point, namely Proposition \ref{prop:diag-prod-k3}. The difference is that the first one is elementary by staying in the category of Chow motives and is geometric so that in principle it gives rise to an explicit expression of the Chow motive  $\h(\hat{\cM})$ in terms of $\h(S)$; the second one is quicker by using mixed motives and Proposition \ref{prop:GeneratedCats}, but it is hopeless to deduce any concrete relation between these two motives via this approach.  
\begin{proof}[{First proof of Proposition \ref{prop:motive_hatM}}]
    By Proposition \ref{prop:diag-prod-k3}, we have
    $$ [\Delta_{\cM^\mathrm{st}}] = \sum \delta_i \circ \gamma_i \in \CH^{10}(\cM^\mathrm{st} \times \cM^\mathrm{st}),$$
    where $\gamma_i \in \CH^{e_i}(\cM^\mathrm{st} \times S^{k_i})$ and $\delta_i \in \CH^{d_i}(S^{k_i} \times \cM^\mathrm{st})$. Let $\hat{\gamma}_i \in \CH^{e_i}(\hat{\cM} \times S^{k_i})$ and $\hat{\delta}_i \in \CH^{d_i}(S^{k_i} \times \hat{\cM})$ be any closure of cycles representing $\gamma_i$ and $\delta_i$ respectively. Then the support of the class
    $$ \Delta_{\hat{\cM}} - \sum \hat{\delta}_i \circ \hat{\gamma}_i \in \CH^{10}(\hat{\cM} \times \hat{\cM})$$
    lies in the boundary $(\hat{\cM} \times \partial\hat{\cM}) \cup (\partial\hat{\cM} \times \hat{\cM})$, hence we can write in $\CH^{10}(\hat{\cM}\times \hat{\cM})$
    \begin{equation}
        \label{eqn:hatM-diag}
        \Delta_{\hat{\cM}}= \sum \hat{\delta}_i \circ \hat{\gamma}_i + Y_{\hat{\Sigma}} + Y_{\hat{\Omega}} + Z_{\hat{\Sigma}} + Z_{\hat{\Omega}}
    \end{equation}
    for some algebraic cycles $Y_{\hat{\Sigma}} \in \CH^9(\hat{\cM} \times \hat{\Sigma})$, $Y_{\hat{\Omega}} \in \CH^9(\hat{\cM} \times \hat{\Omega})$, $Z_{\hat{\Sigma}} \in \CH^9(\hat{\Sigma} \times \hat{\cM})$ and $Z_{\hat{\Omega}} \in \CH^9(\hat{\Omega} \times \hat{\cM})$.
    
    For each $i$, the cycles $\hat{\gamma}_i$ and $\hat{\delta}_i$ can be viewed as morphisms of motives
    $$ \h(\hat{\cM}) \stackrel{\hat{\gamma}_i}{\longrightarrow} \h(S^{k_i})(n_i) \stackrel{\hat{\delta}_i}{\longrightarrow} \h(\hat{\cM}) $$
    for $n_i = e_i-2m = 2k_i - d_i$. On the other hand, we denote by $j_{\hat{\Sigma}}$ and $j_{\hat{\Omega}}$ the closed embedding of $\hat{\Sigma}$ and $\hat{\Omega}$ in $\hat{\cM}$ respectively. Then we have morphisms of motives
    \begin{align*}
        &\h(\hat{\cM}) \stackrel{Y_{\hat{\Sigma}}}{\longrightarrow} \h(\hat{\Sigma}) \stackrel{(j_{\hat{\Sigma}})_\ast}{\longrightarrow} \h(\hat{\cM}), \\
        &\h(\hat{\cM}) \stackrel{Y_{\hat{\Omega}}}{\longrightarrow} \h(\hat{\Omega}) \stackrel{(j_{\hat{\Omega}})_\ast}{\longrightarrow} \h(\hat{\cM}), \\
        &\h(\hat{\cM}) \stackrel{j_{\hat{\Sigma}}^\ast}{\longrightarrow} \h(\hat{\Sigma})(-1) \stackrel{Z_{\hat{\Sigma}}}{\longrightarrow} \h(\hat{\cM}), \\
        &\h(\hat{\cM}) \stackrel{j_{\hat{\Omega}}^\ast}{\longrightarrow} \h(\hat{\Omega})(-1) \stackrel{Z_{\hat{\Omega}}}{\longrightarrow} \h(\hat{\cM}). \\
    \end{align*}
    It follows by \eqref{eqn:hatM-diag} that the sum of all the above compositions add up to the identity on $\h(\hat{\cM})$. Hence $\h(\hat{\cM})$ is a direct summand of 
    $$ \left( \oplus_i \h(S^{k_i})(n_i) \right) \oplus \h(\hat{\Sigma}) \oplus \h(\hat{\Omega}) \oplus \h(\hat{\Sigma})(-1) \oplus \h(\hat{\Omega})(-1). $$
    Combining this with Lemma \ref{lem:motive_hathat}, we finish the proof.
\end{proof}

\begin{proof}[{Second proof of Proposition \ref{prop:motive_hatM}}]
By a repeated use of the localization distinguished triangle \eqref{eqn:localization}, we see that for a variety together with a locally closed stratification, if the motive with compact support of each stratum is in some triangulated tensor subcategory of $\DM$, then so is the motive with compact support of the ambiant space; conversely, if the motive with compact support of the ambiant scheme as well as those of all but one strata are in some triangulated tensor subcategory of $\DM$, then so is the motive with compact support of the remaining stratum.

Now from the geometry recalled in \S\ref{subsec:Resolution}, we see that $\hat{\cM}$ has a stratification with four strata $\cM^{\mathrm{st}}$, $\hat{\Omega}\backslash \hat{\Sigma}$, $\hat{\Sigma}\backslash \hat{\Omega}$, $\hat{\Omega}\cap \hat{\Sigma}$. By Lemma \ref{lem:motive_hathat}, Proposition \ref{prop:diag-prod-k3} and the previous paragraph, the motives with compact support of $\hat{\cM}$ as well as those of its strata and their closures are in $\langle M(S)\rangle_{\DM}$. Since $\hat{\cM}$ is smooth and projective, its motive lies in the subcategory of Chow motives, hence in $\langle \h(S)\rangle_{\CHM}$ by Proposition \ref{prop:GeneratedCats}.
\end{proof}

\begin{corollary}
    \label{cor:motive_tildeM}
    The Chow motive $\h(\tilde{\cM})$ is contained in $\langle\h(S)\rangle_{\CHM}$.
\end{corollary}

\begin{proof}
    Since $\hat{\cM}$ is a blow-up of $\tilde{\cM}$ along a smooth center, it follows by \cite[\S 9]{Man68} that $\h(\tilde{\cM})$ is a direct summand of $\h(\hat{\cM})$. Then the conclusion follows from Proposition \ref{prop:motive_hatM} together with the fact that $\langle\h(S)\rangle_{\CHM}$ is closed under direct summands.
\end{proof}

%The graph $\Gamma$ of the contraction morphism $\hat{\cM} \longrightarrow \tilde{\cM}$, considered as cycle classes in $\CH^{10}(\tilde{\cM} \times \hat{\cM})$ and $\CH^{10}(\hat{\cM} \times \tilde{\cM})$, produce morphisms of motives
%$$ \h(\tilde{\cM}) \stackrel{\Gamma}{\longrightarrow} \h(\hat{\cM}) \stackrel{\Gamma}{\longrightarrow} \h(\tilde{\cM}). $$
%It is easy to check that the composition of the two correspondences is the identity morphism. Hence $\h(\tilde{\cM})$ is a direct summand of $\h(\hat{\cM})$.

\begin{corollary}\label{cor:motive_Ms}
    The mixed motives $M(\cM^{\mathrm{st}})$, $M_{c}(\cM^{\mathrm{st}})$ and $M(\cM)\simeq M_{c}(\cM)$ all belong to $\langle M(S)\rangle_{\DM}$, the triangulated tensor subcategory of $\DM$ generated by $M(S)$.
\end{corollary}
\begin{proof}
    Recall first that $\hat{\cM}$ has a stratification with strata being $\cM^{\mathrm{st}}$, $\hat\Omega\cap \hat\Sigma$, $\hat\Sigma\backslash\hat\Omega$ and $\hat\Omega\backslash\hat\Sigma$. By Lemma \ref{lem:motive_hathat} and Proposition \ref{prop:motive_hatM}, together with a repeated use of the distinguished triangle for motives with compact support (\cite[P.195]{Voe00}) yields that the motives with compact support of all strata as well as their closures are in $\langle M(S)\rangle_{\DM}$. This proves the claim for $M_c(\cM^\mathrm{st})$ and $M_c(\cM)$. The remaining claim for $M(\cM^\mathrm{st})$ follows from the motivic Poincar\'e duality \eqref{eqn:MotPD}.
\end{proof}

\begin{corollary}\label{cor:InfChowAb}
There are infinitely many projective hyper-K\"ahler varieties of O'Grady-10 deformation type whose Chow motive is abelian.
\end{corollary}
\begin{proof}
By Corollary \ref{cor:motive_tildeM}, it suffices to see that there are infinitely many projective K3 surfaces with abelian Chow motives. To this end, we can take for example the  Kummer K3 surfaces or K3 surfaces with Picard number at least 19 by \cite{Ped12}.
\end{proof}

\begin{remark}[Motives of Kummer moduli spaces]
When $S$ is an abelian surface, the previously considered moduli spaces $\cM^{\mathrm{st}}$, $\cM$, $\widehat{\cM}$ and $\tilde{\cM}$ are all isotrivally fibered over $S\times \hat{S}$, via the map $E\mapsto (c_{1}(E), \alb(c_{2}(E)))$. Let us denote the corresponding fibers by $\mathcal{K}^{\mathrm{st}}$, $\mathcal{K}$, $\widehat{\mathcal{K}}$, $\tilde{\mathcal{K}}$, called Kummer moduli spaces of sheaves. Except for some special cases like generalized Kummer varieties (see \cite{FTV19}), the analogy of Proposition \ref{prop:motive_hatM}, Corollary \ref{cor:motive_tildeM} and Corollary \ref{cor:motive_Ms} are unknown for those fibers in general. The missing ingredient is the analogy of Markman's Theorem \ref{thm:diag-chern-k3}, see Remark \ref{rmk:ChallengeKummerst}.
\end{remark}

\section{Moduli spaces of objects in 2-Calabi--Yau categories}\label{sec:NCK3}
As is alluded to in the introduction, many projective hyper-K\"ahler varieties are constructed as moduli spaces of objects in some 2-Calabi--Yau categories, and it is natural to wonder how the motive of the moduli space is related to the ``motive'' of this category, whatever it means\footnote{The motive of a differential graded category can certainly be made precise: it is the theory of non-commutative motives, see Tabuada \cite{Tab15} for a recent account. However, we will take a more naive approach here, which gives more precise information by keeping the Tate twists.}. 

The most prominent case of 2-Calabi--Yau category is the K3 category constructed as the Kuznetsov component of the derived category of a smooth cubic fourfold. However, given the rapid development of the study of stability conditions for many other 2-Calabi--Yau categories, we decided to treat them here in a broader generality. The prudent reader can stick to the cubic fourfold case without missing the point.

Let $Y$ be a smooth projective variety and $\mathcal{A}$ an admissible triangulated subcategory of $D^b(Y)$, the bounded derived category of coherent sheaves on $Y$. Assume that $\mathcal{A}$ is \textit{2-Calabi--Yau}, that is, the Serre functor of $\mathcal{A}$ is the double shift $[2]$.

\begin{example}\label{ex:StabCond}
Here are some interesting examples we have in mind:
\begin{enumerate}[$(i)$]
    \item $Y$ is a K3 surface or abelian surface, and $\mathcal{A}=D^b(Y)$.
    \item $Y$ is a smooth cubic fourfold and $\mathcal{A}$ is the Kuznetsov component defined as the semi-orthogonal complement of the exceptional collection $\langle\mathcal{O}_Y, \mathcal{O}_Y(1), \mathcal{O}_Y(2)\rangle$, see \cite{Kuz10}.
    \item $Y$ is a Gushel--Mukai variety \cite{Gus83} \cite{Muk89} \cite{DK19} of even dimension $n=4$ or 6, and $\mathcal{A}$ is the Kuznetsov component defined as the semi-orthogonal complement of the exceptional collection $$\langle\mathcal{O}_Y, U^\vee, \mathcal{O}_Y(1), U^\vee(1),\cdots, \mathcal{O}_Y(n-3), U^\vee(n-3)\rangle,$$ where $U$ is the rank-2 vector bundle associated to the Gushel map $Y\to \Gr(2, 5)$
and $\mathcal{O}_Y(1)$ is the pull-back of the Pl\"ucker polarization, see \cite{KP16}.
\item $Y$ is a smooth hyperplane section of $\Gr(3, 10)$, called the Debarre--Voisin (Fano) variety \cite{DV10}, and $\mathcal{A}$ is the semi-orthogonal complement of the exceptional collection $$\langle \mathcal{B}_Y, \mathcal{B}_Y(1),\cdots, \mathcal{B}_Y(8)\rangle,$$
where $\mathcal{B}_Y$ is the restriction of the exceptional collection $\mathcal{B}$ of length 12 in the rectangular Lefschetz decomposition of $\Gr(3,10)$ constructed by Fonarev \cite{Fon15}, see \cite[\S3.3]{MS18}.
\end{enumerate}
\end{example}

Assume that the manifold of stability conditions on $\mathcal{A}$ is non-empty, which is expected for all the cases in Example \ref{ex:StabCond} and is established and studied for K3 and abelian surfaces in \cite{Bri08} (see also \cite{YY14}), for the Kuznetsov component of cubic fourfolds by \cite{BLMS}, and for the Kuznetsov component of Gushel--Mukai fourfolds by \cite{PPZ19}. We denote the distinguished connected component of the stability manifold by $\Stab^\dagger(\mathcal{A})$.

As in \cite{AT14}, we can define a lattice structure on the topological K-theory of $\mathcal{A}$, denoted by $\tilde{H}(\mathcal{A})$, see \cite[\S3.4]{MS18}. Now for any $\vv\in \tilde{H}(\mathcal{A})$, and $\sigma\in \Stab^\dagger(\mathcal{A})$, one can form $\cM^{\mathrm{st}}:=\cM_{\mathcal{A},\sigma}(\vv)^{\mathrm{st}}$ (\resp $\cM:=\cM_{\mathcal{A},\sigma}(\vv)$) the moduli space of $\sigma$-stable (\resp $\sigma$-semistable) objects in $\mathcal{A}$ with Mukai vector $\vv$, which is a smooth quasi-projective holomorphic symplectic variety (\resp proper and possibly singular symplectic variety).

One can now extend Theorem \ref{thm:diag-chern-k3} and Proposition  \ref{prop:diag-prod-k3} to the non-commutative setting as follows.

\begin{proposition}
    \label{prop:diag-chern-cubic4}
    The notation and assumption are as above. 
    \begin{enumerate}[$(i)$]
        \item Let $\mathcal{E}$ and  $\mathcal{F}\in D^b(\cM^{\mathrm{st}} \times Y)$ be two universal families. Then
    $$ \Delta_{\cM^{\mathrm{st}}} = c_{2m}(- \sExt^!_{\pi_{13}}(\pi_{12}^\ast(\mathcal{E}), \pi_{23}^\ast(\mathcal{F}))) \in \CH^{2m}(\cM^{\mathrm{st}} \times \cM^{\mathrm{st}}), $$
    where $2m$ is the dimension of $\cM^{\mathrm{st}}$, $\sExt^!_{\pi_{13}}(\pi_{12}^\ast(\mathcal{E}), \pi_{23}^\ast(\mathcal{F}))$ denotes the class of the complex $R\pi_{13,*}(\pi_{12}^{*}(\mathcal{E})^{\vee}\otimes^{\mathbb{L}}\pi_{23}^{*}(\mathcal{F}))$ in the Grothendieck group of $\cM^\mathrm{st} \times \cM^\mathrm{st}$, where $\pi_{ij}$'s are the natural projections from $\cM^{\mathrm{st}}\times Y\times \cM^{\mathrm{st}}$.
    \item There exist finitely many integers $k_i$ and cycles $\gamma_i \in \CH(\cM^\mathrm{st} \times Y^{k_i})$, $\delta_i \in \CH(Y^{k_i} \times \cM^{\mathrm{st}})$, such that $$ \Delta_{\cM^{\mathrm{st}}} = \sum_i \delta_i \circ \gamma_i \in \CH^{2m}(\cM^{\mathrm{st}} \times \cM^{\mathrm{st}}).$$
    \end{enumerate}
\end{proposition}

\begin{proof}
    The proof of $(i)$ is similar to the proof of Markman's theorem \cite{Mar02} or rather its extension in \cite{MZ17}. Their proof only uses standard properties for stable objects and the Serre duality for K3 surfaces, which both hold for $\mathcal{A}$.\\
    The proof of $(ii)$ is exactly the same as in Proposition \ref{prop:diag-prod-k3} (B\"ulles' argument), by replacing $S$ by $Y$ everywhere.
\end{proof}

We first consider the situation where the stability agrees with semi-stability. Then $\vv$ must be primitive and $\sigma$ is $\vv$-generic. In this case, $\cM$ is a smooth and projective hyper-K\"ahler variety, if it is not empty. Once we have the decomposition of the diagonal in Proposition \ref{prop:diag-chern-cubic4} $(ii)$, the same proof as in \cite{Bue18} yields the following generalization of Theorem \ref{thm:Buelles}.

\begin{theorem}\label{thm:NCK3_smooth}
Let $Y$ be a smooth projective variety and let $\mathcal{A}$ be an admissible triangulated subcategory of $D^b(Y)$ such that $\mathcal{A}$ is 2-Calabi--Yau. Let $\vv$ be a primitive element in the topological K-theory of $\mathcal{A}$ and let $\sigma\in \Stab^\dagger(\mathcal{A})$ be a $\vv$-generic stability condition. If $\cM:=\cM_{\mathcal{A}, \sigma}(\vv)$ is non-empty, then its Chow motive is in the pseudo-abelian tensor subcategory generated by the Chow motive of $Y$.
\end{theorem}

As a non-commutative analogue of Conjecture \ref{conj:SingularModuli}, we formulate the following conjecture. 

\begin{conjecture}\label{conj:SingularModuli_NC}
In the same situation as in Theorem \ref{thm:NCK3_smooth}, except that $\vv$ is not necessarily primitive and $\sigma$ is not necessarily generic. If $\cM^{\mathrm{st}}:=\cM_{\mathcal{A}, \sigma}(\vv)^{\mathrm{st}}$ and  $\cM:=\cM_{\mathcal{A}, \sigma}(\vv)$ are non-empty, then their motives and motives with compact support are in the tensor triangulated subcategory generated by the motive of $Y$. If moreover $\cM$ admits a crepant resolution $\tilde{\cM}$, then the Chow motive of $\tilde{\cM}$ is in the pseudo-abelian tensor subcategory generated by the Chow motive of $Y$.
\end{conjecture}

For evidence for Conjecture \ref{conj:SingularModuli_NC}, we restrict to the case where $Y$ is a very general cubic fourfold and $\mathcal{A}$ is its Kuznetsov component. Let $\lambda_1$ and $\lambda_2$ be the cohomological Mukai vectors of the projections into $\mathcal{A}$ of $\mathcal{O}_Y(1)$ and $\mathcal{O}_Y(2)$ respectively. Then the topological K-theory of $\mathcal{A}$ is an $A_2$-lattice with basis $\{\lambda_1, \lambda_2\}$, equipped with a K3-type Hodge structure \cite{AT14}. Then for a generic stability condition $\sigma$ (see \cite{BLMS}), there is an O'Grady-type crepant resolution of the singular moduli space $\cM_{\mathcal{A}, \sigma}(2\lambda_1+2\lambda_2)$, which is of O'Grady-10 deformation type, see \cite{LPZ19}. Our result is that Conjecture \ref{conj:SingularModuli_NC} holds true in this case. See Theorem \ref{thm:main3} in the introduction for the precise statement.

\begin{proof}[Proof of Theorem \ref{thm:main3}]
The argument is more or less the same as in \S\ref{sec:OG10}: the singular locus of the moduli space of semistable objects $\cM(2\vv_0)$ is  $\Sym^2(\cM(\vv_0))$, whose singular locus is the diagonal $\cM(\vv_0)$. By the same procedure of blow-ups as in \S\ref{subsec:Resolution}, we get a smooth projective variety $\hat{\cM}$ together with a stratification such that the motive with compact support of all strata belong to the tensor triangulated subcategory generated by the motive of $\cM(\vv_0)$, hence also to the subcategory generated by $\h(Y)$, by Theorem \ref{thm:NCK3_smooth}. The rest of the proof is the same as \S\ref{sec:OG10}.
\end{proof}

\begin{proof}[Proof of Corollary \ref{cor:StandConj}]
In the situation of Theorem \ref{thm:main1} (\textit{resp.}~Theorem \ref{thm:main3}), $\tilde{\cM}$ is \textit{motivated} by the surface $S$ (\textit{resp.}~the cubic fourfold $Y$) in the sense of Arapura \cite[Lemma 1.1]{MR2271985}: indeed, by applying the full functor $\CHM\to \GRM$ from the category of Chow motives to that of Grothendieck motives, our main result implies that the Grothendieck motive of $\tilde{\cM}$ is in the pseudo-abelian tensor subcategory generated by the Grothendieck motive of $S$ (\textit{resp.}~$Y$). Since the Lefschetz standard conjecture is known for $S$ and $Y$, we can invoke Arapura's result \cite[Lemma 4.2]{MR2271985} to obtain the standard conjectures for $\tilde{\cM}$.
\end{proof}

\section{Defect groups of hyper-K\"ahler varieties}\label{sec:defect}

In this section we study the Andr\'{e} motives of projective hyper-K\"{a}hler varieties with $b_2\neq 3$. For any such $X$, we construct the defect group $P(X)$, and prove Theorem \ref{thm:product2} (=Theorem \ref{thm:product}) and Corollary \ref{cor:Pdefect2} (=Corollary \ref{cor:Pdefect}). In the next section we will apply these results to the known examples of hyper-K\"{a}hler varieties.

%This section is of general nature. For any projective hyper-K\"ahler variety $X$ with $b_2\neq 3$ we construct a conjecturally trivial algebraic group $P(X)$, called the \emph{defect group}, which measures the failure of Conjecture \ref{mhc} for $X$ as well as the non-abelianity of the Andr\'e motive $\mathcal{H}(X)$ and satisfies plenty of constraints. We will see in the next section how this defect group helps us to solve those conjectures for all known examples of hyper-K\"ahler varieties.

The starting point and a main tool of our study is the following general theorem due to Andr\'{e}. 
%(the statement about the Mumford--Tate conjecture has been explicitely addressed and generalized by Moonen in \cite{moonen2017}).
\begin{theorem}[\cite{Andre1996}]\label{abelianityH2}
Let $X$ be a projective hyper-K\"{a}hler variety such that $b_2(X)\neq 3$. Then the Andr\'{e} motive $\mathcal{H}^2(X)$ is abelian. In particular, Conjecture \ref{mhc} holds for $\mathcal{H}^2(X)$.
%Moreover, if $X$ is defined over a finitely generated field $k\subset\C$, the motivated Mumford--Tate conjecture holds for
%$ \mathcal{H}^2(X)\in AM_{\Q }(k)$.
\end{theorem}

We review the Lie algebra action constructed by Looijenga--Lunts \cite{looijenga1997lie} and Verbitsky \cite{verbitsky1996cohomology} on cohomology groups of varieties, as well as its remarkable properties when applied to compact hyper-K\"{a}hler manifolds. This action is crucial for the proof of Theorem \ref{thm:product2}. To ease the notation, the coefficient field $\Q$ in all cohomology groups is suppressed.

\subsection{The Looijenga--Lunts--Verbitsky (LLV) Lie algebra}\label{subsec:LLV}

Let $X$ be a $2m$-dimensional compact hyper-K\"ahler variety. A cohomology class $x\in H^2(X)$ is said to satisfy the \textit{Lefschetz property} if the maps given by cup-product $L_x^{j}\colon \, H^{2m-j}(X)\to H^{2m+j}(X)$ sending $\alpha$ to $x^{j}\cup \alpha$, are isomorphisms for all $j\geq 0$. The Lefschetz property for a class $x$ in $H^2 (X)$ is equivalent to the existence of a $\mathfrak{sl}_2$-triple $(L_x,\theta, \Lambda_x)$, where $\theta\in \End\bigl(H^*(X)\bigr)$ is the degree-$0$ endomorphism which acts as multiplication by $k-2m$ on $H^k(X)$ for all $k\in \NN$. Moreover, in this case $\Lambda_x$ is uniquely determined by $L_x$ and $\theta$. Note that the first Chern class of an ample divisor on~$X$ has the Lefschetz property by the hard Lefschetz theorem.

The $\LLV$-Lie algebra of $X$, denoted by $\mathfrak{g}_{\LLV}(X)$, is defined as the Lie subalgebra of $\mathfrak{gl}(H^*(X))$ generated by the $\mathfrak{sl}_2$-triples $(L_x, \theta, \Lambda_x)$ as above for all cohomology classes $x\in H^{2}(X)$ satisfying the Lefschetz property. It is shown in \cite[\S(1.9)]{looijenga1997lie} that~$\mathfrak{g}_{\LLV}(X)$ is a semisimple~$\Q$-Lie algebra, evenly graded by the adjoint action of $\theta$. The construction does not depend on the complex structure; therefore,~$\mathfrak{g}_{\LLV}(X)$ is deformation invariant. 

Let us denote by $H$ the space~$H^2(X)$ equipped with the Beauville--Bogomolov quadratic form \cite{beauville1983varietes}. Let $\tilde{H}$ denote the orthogonal direct sum of $H$ and a hyperbolic plane $U=\langle v,w\rangle$ equipped with the form $v^2=w^2=0$ and $vw=-1$. We summarize the main properties of the Lie algebra~$\mathfrak{g}_{\mathrm{LLV}}(X)$.

\begin{theorem}[Looijenga--Lunts--Verbitsky]\label{gLLV}
	\begin{enumerate}[$(i)$]
		\item \label{struct}
		There is an isomorphism of $\Q$-Lie algebras
		$$\mathfrak{g}_{\LLV}(X)\cong\mathfrak{so}(\tilde{H}),$$
		which maps $\theta\in\mathfrak{g}_{\mathrm{LLV}}(X)$ to the element of $\mathfrak{so}(\tilde{H})$ which acts as multiplication by $-2$ on~$v$, by $2$ on $w$, and by $0$ on $H$. Hence, we have
		\[
		\mathfrak{g}_{\LLV}(X)=\mathfrak{g}_{-2}(X)\oplus\mathfrak{g}_{0}(X)\oplus\mathfrak{g}_{2}(X).
		\]
		Moreover, $\mathfrak{g}_0(X) \cong \mathfrak{so}(H ) \oplus \Q \cdot\theta$, is the centralizer of $\theta$ in $\mathfrak{g}_{\LLV}(X)$. The abelian subalgebra $\mathfrak{g}_{2}(X)$ is the linear span of the endomorphisms $L_x$, for $x\in H^{2}(X)$, and $\mathfrak{g}_{-2}(X)$ is the span of the $\Lambda_x$, for all $x\in H^2(X)$ with the Lefschetz property. 
		\item The Lie subalgebra $\mathfrak{so}(H) \subset \mathfrak{g}_0(X)$ acts by derivations on the graded algebra $H^*(X)$. The induced action of $\mathfrak{so}(H)$ on $H^2(X)$ is the standard representation.
	    \item \label{weil} 
	    Let $\rho\colon \mathfrak{so}(H)\to \mathfrak{gl}(H^*(X))$ be the induced representation of $\mathfrak{so}(H)\subset \mathfrak{g}_0(X)$. Then the Weil operator\footnote{Here the Weil operator refers to the derivation of the usual Weil operator, which acts on $H^{p,q}(X)$ as multiplication by $i^{p-q}$. Hence, $W$ acts on each $H^{p,q}(X)$ as multiplication by~$i(p-q)$.} $W$ is an element of $\rho\bigl(\mathfrak{so}(H)\bigr)\otimes\R$. 
	    %Here, the Weil operator is the real endomorphism $W$ of $H^*(X,\R)$ whose $\C$-linear extension acts on on each $H^{p,q}(X)$ as multiplication by~$i(p-q)$.
	   \end{enumerate}
\end{theorem}

The above theorem is proved in \cite{verbitsky1996cohomology}, and in~\cite[Proposition~4.5]{looijenga1997lie}, see also the appendix of \cite{KSV2017}. These proofs are carried out with real coefficients, but immediately imply the result with rational coefficients: since $\mathfrak{g}_{\LLV}(X)$ is defined over~$\Q$, the equality  
$ \mathfrak{g}_{\LLV}(X)\otimes \R = \mathfrak{so}(\tilde{H})\otimes \R$
of Lie subalgebras of $\mathfrak{gl}(\tilde{H})\otimes \R$ implies that the same equality already holds with rational coefficients.

\begin{remark}[Integration]\label{remark:SOrepresentation}
Let $\rho^+\colon\mathfrak{so}(H) \to \mathfrak{gl}(H^+(X))$ be the induced representation on the even cohomology. It follows from \cite[Corollary 8.2]{verbitsky1995mirror} that $\rho^+$ integrates to a faithful representation 
\[
\tilde{\rho}^+\colon  \SO(H)\to \prod_i\GL\bigl(H^{2i}(X)\bigr), 
\]
such that the induced representation on $H^2(X)$ is the standard representation. 
If the odd cohomology of $X$ is non-trivial, $\rho$ integrates to a faithful representation 
\[
\tilde{\rho}\colon \Spin(H) \to \prod_i\GL\bigl(H^{i}(X)\bigr),
\]
and the kernel of the action of $\Spin(H)$ on the even cohomology is an order-2 subgroup $\langle \iota \rangle$, where $\iota$ is the non-trivial element in the kernel of the double cover $\Spin(H)\to \SO(H)$ and $\tilde{\rho}(\iota)$ acts on $H^i(X)$ via multiplication by $(-1)^{i}$. Note also that the action induced by $\tilde{\rho}$ and $\tilde{\rho}^+$ is via algebra automorphisms, thanks to Theorem \ref{gLLV}$(ii)$.
\end{remark}

\subsection{Splitting of the motivic Galois group}\label{subsec:Splitting}
 Let $H(X)$ (\resp $H^+(X)$) be the full (\resp  even) rational cohomology group of $X$ equipped with Hodge structure. The natural inclusions of $H^2(X)$ into $ H^+(X)$  and $H^*(X)$ induce surjective morphisms of Mumford--Tate groups $$\pi^+_2\colon \MT(H^+(X))\to\MT(H^2(X));$$ $$\pi_2\colon \MT(H^*(X))\to \MT (H^2(X)).$$ Let $\iota\in\GL(H^*(X))$ act on each $H^{i}(X)$ via the multiplication by $(-1)^{i}$ for all $i$.

\begin{proposition}
\label{cor:MT}
The notation is as above.
\begin{enumerate}[$(i)$]
    \item The morphism $\pi^+_2$ is an isomorphism.
	In particular, the Hodge structure $H^+(X)$ belongs to the tensor subcategory of $\mathrm{HS}_{\Q}^{\mathrm{pol}}$ generated by $H^2(X)$.
	\item If $X$ has non-vanishing odd cohomology, the morphism $\pi_2$
	is an isogeny with kernel $\langle \iota \rangle\simeq \Z/2\Z$. Moreover, if $A$ is any Kuga--Satake variety for $H^2(X)$ in the sense of Definition \ref{def:Kuga-Satake}, we have $\langle H^*(X) \rangle_{\HS}=\langle H^1(A) \rangle_{\HS}$.
\end{enumerate}
\end{proposition}

The natural choice for $A$ is the abelian variety obtained through the Kuga--Satake construction on $H^2(X)$ equipped with the Beauville-Bogomolov form, see \S\ref{subsec:KSconstruction}; let us remark that also the construction of \cite{KSV2017} yields a Kuga--Satake variety for $H^2(X)$ in our sense.

The proof of the proposition will be given after some preliminary results. Recall (\cite{deligne1971conjecture}) that the algebraic group $\CSpin(H)$ is the quotient of $\mathbb{G}_{m} \times \Spin(H)$ in which we identify the element $-1\in \mathbb{G}_m$ with the non-trivial central element $\iota$ of $\Spin(H)$. We introduce a representation 
\[
\sigma \colon \CSpin(H) \to \prod_i \GL(H^{i}(X))
\]
by $\sigma= w\cdot \tilde{\rho}$, where $\tilde{\rho}\colon \Spin(H)\to \prod_i\GL(H^{i}(X))$ is the representation from Remark \ref{remark:SOrepresentation} and $w\colon \mathbb{G}_{m}\to \GL(H^{i}(X))$ is the weight cocharacter, \ie~$w(\lambda)$ acts on $H^{i}(X)$ as multiplication by $\lambda^{i}$, for all $i$ and all $\lambda$. This is a priori a representation of $\mathbb{G}_{m}\times \Spin(H)$, but it indeed factors through $\CSpin(H)$ since by Remark \ref{remark:SOrepresentation} we have $w(-1)=\tilde{\rho}(\iota)$. We also set $$\sigma^+:\CSpin(H)\to \prod_i \GL(H^{2i}(X)) \text{ and }$$ $$\sigma^2: \CSpin(H)\to \GL(H^2(X))$$ to be the induced representations on the even cohomology and on $H^2(X)$ respectively.

\begin{lemma}\label{lem:imageSigma}
\begin{enumerate}[(i)]
\item The homomorphism $\sigma^+\colon \CSpin(H)\to \prod_i\GL(H^{2i}(X))$ is an isogeny of degree 2 onto its image. The natural projection $\pr_2^+\colon \prod_i \GL(H^{2i}(X))\to \GL(H^2(X))$ maps the image of $\sigma^+$ isomorphically onto the image of $\sigma^2$. 
\item If $X$ has non-vanishing odd cohomology, the representation $\sigma\colon \CSpin(H)\to \prod_i \GL(H^{i}(X))$ is faithful, and the projection $\pr_2\colon \prod_i \GL(H^i(X)) \to \GL(H^2(X))$ induces a degree 2 isogeny between the image of $\sigma$ and the image of $\sigma^2$.
\end{enumerate}
\end{lemma}
\begin{proof}
By Remark \ref{remark:SOrepresentation} and the explicit description of $w$, the kernels of $\sigma^+$ and $\sigma^2$ both coincide with the central subgroup of order 2 generated by $(-1,1)= (1, \iota)$. This proves part $(i)$. If $X$ has non-vanishing odd cohomology, $\sigma$ is faithful by Remark \ref{remark:SOrepresentation}, and the second assertion follows. 
\end{proof}

\begin{remark}
Note that the twisted representation $\sigma'=w'\cdot \tilde{\rho}$ where $w'(\lambda)$ acts on $H^{i}(X)$ via multiplication by $\lambda^{i-2m}$ is the representation obtained via integration of $\mathfrak{g}_0\to \prod_i \mathfrak{gl}(H^{i}(X))$.
\end{remark}

The point of introducing the above representation is that it controls the Mumford--Tate group.
\begin{lemma}\label{lem:contained}
The Mumford--Tate group $\MT(H^*(X))$ is contained in the image of $\sigma$.
\end{lemma}
\begin{proof}
Let $G=\im(\sigma)$. Since both $\MT(H^*(X))$ and $G$ are reductive, by \cite[Proposition 3.1]{deligne1982hodge} it suffices to check that for any tensor construction 
	\[
	T=\bigoplus_i\ H^*(X)^{\otimes m_i}\otimes\bigl(H^*(X)^\vee\bigr)^{\otimes n_i},
	\]
any element~$\alpha$ of~$T$ that is invariant for $G$ is also fixed by $\MT(H^*(X))$.
Let $\alpha\in T$ be such an invariant for $G$. Then the image of $\alpha$ in $T\otimes\C$ is annihilated by all elements of $\rho(\mathfrak{so}(H)) \otimes \C$. By Theorem~\ref{gLLV}$(iii)$, $\alpha$ is annihilated by the Weil operator $W$. Therefore $\alpha$ is of type $(p,p)$ for some integer $p$. However, since $w(\mathbb{G}_{m})$ also acts trivially on $\alpha$, we must have $p=0$; hence $\alpha$ is a Hodge class of type (0, 0) and is thus fixed by the Mumford--Tate group.
\end{proof}

\begin{proof}[Proof of Proposition \ref{cor:MT}]
$(i)$ Lemma \ref{lem:contained} implies that $\MT(H^+(X))\subset \im (\sigma^+)$. The morphism~$\pi^+_2$ is the restriction of the natural projection $\pr^+_2\colon \prod_i \GL(H^{2i}(X)) \to \GL(H^2(X))$. Lemma \ref{lem:imageSigma} implies in particular that the restriction of $\pr_2^+$ to $\im(\sigma^+)$ is injective; hence its restriction to the subgroup $\MT(H^+(X))$ is also injective, \textit{i.e.}~$\pi_2^+$ is injective and hence it is an isomorphism.

$(ii)$ Assume now that the odd cohomology of $X$ is non-trivial. Since $\MT(H^*(X))\subset \im(\sigma)$ by Lemma \ref{lem:contained}, we deduce as above that the kernel of the morphism $\pi_2\colon\MT(H^*(X))\to \MT(H^2(X))$ is contained in the kernel of $\pr_2\colon \im(\sigma)\to \im(\sigma^2)$. By Lemma \ref{lem:imageSigma}, this is an order $2$ central subgroup of $\im(\sigma)$, generated by $w(-1)= \iota$. Clearly $w(-1)$ is contained in $\MT(X)$, and it follows that $\pi_2$ is an isogeny of degree 2 whose kernel is generated by $\iota$.

Finally, let $A$ be any Kuga--Satake abelian variety for $H^2(X)$, in the sense of Definition \ref{def:Kuga-Satake}. Then $\langle H^1(A) \rangle_{\mathrm{HS}} $ is the unique tannakian subcategory such that $\langle H^2(X)\rangle_{\mathrm{HS}} = \langle H^1(A) \rangle_{\mathrm{HS}}^{\mathrm{ev}} \subsetneq \langle H^1(A) \rangle_{\mathrm{HS}}$, by Theorem \ref{thm:kuga-satake}. Therefore it is enough to show that $\langle H^*(X) \rangle_{\mathrm{HS}} $ also satisfies this property. Consider the commutative diagram
\[
\begin{tikzcd} 
\MT(H^*(X)) \arrow[two heads]{rd}{\pi^{\mathrm{ev}}} \arrow[two heads]{rr}{\pi_2} && \MT(H^2(X))\\
 & \MT(\langle H^*(X)\rangle^{\mathrm{ev}}_{\mathrm{HS}}) \arrow[two heads]{ru}{\pi^{\mathrm{ev}}_2}.
 \end{tikzcd}
\]
We have just proven that $\pi_2$ is an isogeny of degree 2, and we know that the morphism $\pi^{\mathrm{ev}}$ is also an isogeny of degree 2, see \S\ref{subsec:KScategory}; we conclude that $\pi_2^{\mathrm{ev}}$ is an isomorphism and hence $\langle H^*(X) \rangle^{\mathrm{ev}}_{\mathrm{HS}}=\langle H^2(X)\rangle_{\mathrm{HS}}$.
\end{proof}

The following observation will be used in the proof of Theorem \ref{thm:product2}.

\begin{lemma}\label{lem:commute}
	Let $G$ be a group acting on $H^*(X)$ via graded algebra automorphisms. If $G$ acts trivially on $H^2(X)$, then the $G$-action commutes with the action of the LLV Lie algebra (\S\ref{subsec:LLV}).
\end{lemma}
\begin{proof}
Let $g\in G$. By assumption, $g$ commutes with $\theta$ and $L_x$, for any $x\in H^2(X)$. Moreover, if~$x$ has the Lefschetz property, then $g$ commutes with~$\Lambda_x$ as well: indeed,~$L_x=gL_xg^{-1}$,~$\theta=g\theta g^{-1}$ and $g\Lambda_xg^{-1}$ form an $\mathfrak{sl}_2$-triple, and since $\Lambda_x$ is uniquely determined by the elements $L_x$ and $\theta$, we must have $g\Lambda_xg^{-1}=\Lambda_x$. One can conclude since the various operators $L_x$ and $\Lambda_x$, for $x\in H^2(X)$, generate the Lie algebra $\mathfrak{g}_{\LLV}(X)$.
\end{proof}

We now turn to the proof of the main result of this section.

\begin{theorem}[Splitting]\label{thm:product2}
Let $X$ be a projective hyper-K\"ahler variety with $b_2(X)\neq 3$.
%Notation is as before. 
%Assume that $b_2(X)\neq 3$.
Then, inside $\Gmot(\mathcal{H}(X))$, the subgroups $P(X)$ and $\MT(H^*(X))$ commute, intersect trivially with each other and generate the whole group. In short, we have an equality:
    $$\Gmot(\mathcal{H}(X))=\MT(H^*(X))\times P(X).$$
 Similarly, the even defect group is a direct complement of the even Mumford--Tate group in the motivic Galois group of the even Andr\'e motive of $X$,
 $$\Gmot(\mathcal{H}^+(X))=\MT(H^+(X))\times P^+(X).$$
\end{theorem}

\begin{proof}%[Proof of Theorem \ref{thm:product}]
We first treat the even motive.
We have a commutative diagram
\[
	\begin{tikzcd}
	\Gmot(\mathcal{H}^+(X)) \arrow[two heads]{rr}{\pi^{+}_{2, \operatorname{mot}}} \arrow[hookleftarrow]{d}{i_+} && \Gmot(\mathcal{H}^2(X)) \arrow[hookleftarrow]{d}{i_2}\\
	\MT(H^+(X)) \arrow[two heads]{rr}{\pi^+_2 } && \MT(H^2(X))
	\end{tikzcd}
\]
Here, $i_+$ and $i_2$ denote the natural inclusions; $\pi^+_2$ and $i_2$ are isomorphisms due to Proposition~\ref{cor:MT} and Theorem~\ref{abelianityH2} respectively. It follows that we have a section $s=i_+\circ( i_2 \circ \pi^+_2)^{-1}$ of $\pi^+_{2,\operatorname{mot}}$, whose image is $\MT(H^+(X))$. 

Recall that $P^+(X)$ is defined as the kernel of the map $\pi^+_{2,\operatorname{mot}}$. We deduce that $\Gmot(\mathcal{H}^+(X))$ is the semidirect product of its subgroups $P^+(X)$ and $\MT (H^+(X))$, which intersect trivially.
In order to show that $\Gmot(\mathcal{H}^+(X)) = \MT (H^+(X)) \times P^+(X)$, it thus suffices to show that $P^+(X)$ and $\MT (H^+(X))$ commute. By Lemma \ref{lem:contained}, it suffices to show that $P^+(X)$ commutes with the image of the representation $\sigma^+$. Since $P^+(X)$ preserves the grading on $H^+(X)$, its action clearly commutes with the weight cocharacter $w$.
Note that every element of $\Gmot(\mathcal{H}^+(X))$ acts via algebra automorphisms, since the cup-product is given by an algebraic correspondence (namely, the small diagonal in $X \times X\times X$). Moreover, if $p\in P^+(X)$, then by definition $p$ acts trivially on $H^2(X)$; hence, its action commutes with that of the LLV-Lie algebra thanks to Lemma \ref{lem:commute}.
It follows that $P^+(X)$ commutes with the image of the representation $\tilde{\rho}^+$, and therefore $P^+(X)$ commutes with $\sigma^+$ as desired.

%The action of $P^+(X)$ obviously commutes with the image of the weight cocharacter $w\colon \mathbb{G}_{m}\to \Gmot(H^+(X))$, and hence it suffices to show that $P^+(X)$ commutes with the special Mumford--Tate group $\SMT(H^+(X))$.
%Note that every element of $\Gmot(\mathcal{H}^+(X))$ acts via algebra automorphisms, since the cup-product is given by an algebraic correspondence (namely, the small diagonal in $X \times X\times X$). Moreover, if $p\in P^+(X)$, by definition it acts trivially on $H^2(X)$; hence, its action commutes with that of the LLV-Lie algebra thanks to Lemma \ref{lem:commute}.
%It follows that $P^+(X)$ commutes with $\SMT(H^+(X))\subset \tilde{\rho}^+(\SO(H))$.

Assume now that the odd cohomology of $X$ does not vanish, and choose a Kuga-Satake variety~$A$ for $H^2(X)$, see Appendix \ref{Appendix}. By Lemma \ref{lem:KSvsodd} below, the motive $\mathcal{H}^1(A)$ belongs to $\langle \mathcal{H}(X) \rangle_{\AM}$. We consider the commutative diagram
\[
	\begin{tikzcd}
	\Gmot(\mathcal{H}(X)) \arrow[two heads]{rr}{\pi_{A,\operatorname{mot}}} \arrow[hookleftarrow]{d}{i} && \Gmot(\mathcal{H}^1(A)) \arrow[hookleftarrow]{d}{i_A}\\
	\MT(H^*(X)) \arrow[two heads]{rr}{\pi_A } && \MT(H^1(A))
	\end{tikzcd}
\]
The morphisms $\pi_A$ and $i_A$ are isomorphisms by Proposition \ref{cor:MT}$(ii)$ and Theorem \ref{mhcabelian} respectively. Note that by Theorem \ref{thm:kuga-satake}, the kernel $P(X)$ of $\pi_{A,\mathrm{mot}}$ does not depend on the choice of the Kuga--Satake abelian variety $A$; this group is by definition the defect group of $X$. As above, we deduce the existence of a section of $\pi_{A,\mathrm{mot}}$ with image $\MT(H^*(X))$, and to conclude we need to show that $P(X)$ and $\MT(H^*(X))$ commute. To this end, we consider the commutative diagram with exact rows
\[
\begin{tikzcd}
1\arrow[r] & P(X) \arrow{r} \arrow{d}
&
\Gmot(\mathcal{H}(X)) \arrow{r}{\pi_{A,\operatorname{mot}}} \arrow{d}{=}
& \Gmot(\mathcal{H}^1(A)) \arrow[two heads]{d}\arrow[r] & 1\\
1\arrow[r]&Q(X) \arrow{r} & \Gmot(\mathcal{H}(X))\arrow{r}{\pi_{2,\operatorname{mot}}} 
& \Gmot(\mathcal{H}^2(X))\arrow[r] & 1
\end{tikzcd}
\]
The group $Q(X)$ commutes with the action of $\mathfrak{g}_{\mathrm{LLV}}$, by Lemma \ref{lem:commute}, and it therefore commutes with the Mumford--Tate group, thanks to Lemma \ref{lem:contained}. Since $P(X)$ is a subgroup of $Q(X)$, it also commutes with $\MT(H^*(X))$, and we have $\Gmot(\mathcal{H}(X))=P(X)\times \MT(H^*(X))$. Also note that we have $Q(X)\cap \MT(H^*(X))= \langle \iota \rangle$, and that $Q(X)=P(X)\times \langle\iota\rangle$.
\end{proof} 

In the previous proof, we have used the following result. See Appendix \ref{Appendix} for the notation.
\begin{lemma}\label{lem:KSvsodd}
Assume that the odd cohomology of $X$ does not vanish and $b_2(X)\neq 3$. Let $A$ be any Kuga--Satake variety (Definition \ref{def:Kuga-Satake}) for the Hodge structure $H^2(X)$. Then $\mathcal{H}^1(A)\in \langle \mathcal{H}(X) \rangle_{\AM}$.
\end{lemma}
\begin{proof}
Note that since $\mathcal{H}^2(X)$ is an abelian motive by Andr\'{e}'s Theorem \ref{abelianityH2}, any Kuga-Satake variety $A$ for $H^2(X)$ satisfies $\langle \mathcal{H}^1(A)\rangle^{\mathrm{ev}}=\langle \mathcal{H }^2(X)\rangle$, see Corollary \ref{cor:motivicKS}. Choose any such $A$, and consider the Andr\'{e} motive $\mathcal{H}(X)\oplus \mathcal{H}^1(A)$. The inclusions of the summands $\mathcal{H}(X)$ and $\mathcal{H}^1(A)$ determine surjective homomorphisms 
$q\colon \Gmot(\mathcal{H}(X)\oplus \mathcal{H}^1(A))\to \Gmot(\mathcal{H}(X))$, and $q_A\colon \Gmot(\mathcal{H}(X)\oplus \mathcal{H}^1(A))\to \Gmot(\mathcal{H}^1(A))$. The desired conclusion is equivalent to the inclusion $\ker(q)\subset \ker(q_A)$. In fact, this precisely means that the tannakian category generated by $\mathcal{H}^1(A)$ is contained in $\langle \mathcal{H}(X)\rangle$, which then implies that $q$ is an isomorphism. 
We consider the analogous morphisms for the even parts
\[
q^{\mathrm{ev}}\colon \Gmot\big(\langle\mathcal{H}(X)\oplus \mathcal{H}^1(A)\rangle^{\mathrm{ev}}\big)\to \Gmot\big(\langle\mathcal{H}(X)\rangle^{\mathrm{ev}}\big),
\]
\[
q_A^{\mathrm{ev}}\colon \Gmot\big(\langle \mathcal{H}(X)\oplus \mathcal{H}^1(A)\rangle^{\mathrm{ev}}\big)\to \Gmot\big(\langle\mathcal{H}^1(A)\rangle^{\mathrm{ev}}\big).
\]
The conclusion of Lemma \ref{rmk:oddHS} holds for Andr\'{e} motives as well. Therefore, the preimage of~$\ker(q^{\mathrm{ev}})$ (respectively, of $\ker(q_A^{\mathrm{ev}})$) under the morphism $\Gmot(\mathcal{H}(X)\oplus \mathcal{H}^1(A))\to \Gmot(\langle \mathcal{H}(X)\oplus \mathcal{H}^1(A)\rangle^{\mathrm{ev}})$ equals $\langle\iota\rangle\times \ker(q)$ (respectively, $\langle \iota\rangle\times \ker(q_A)$), and it suffices to show that $\ker(q^{\mathrm{ev}})\subset \ker(q_A^{\mathrm{ev}})$.
To this end, consider the commutative diagram with short exact rows
\[
	\begin{tikzcd}
	1\arrow{r} & \ker({q_{A}^{\mathrm{ev}}}) \arrow{r} \arrow{d}{j} &
	\Gmot(\langle\mathcal{H}(X)\oplus \mathcal{H}^1(A)\rangle^{\mathrm{ev}}) \arrow[two heads]{r}{q_A^{\mathrm{ev}}} \arrow[two heads]{d}{q^{\mathrm{ev}}} & \Gmot(\langle\mathcal{H}^1(A)\rangle^{\mathrm{ev}}) \arrow[]{d}{\cong} \arrow{r} & 1  \\
	1\arrow{r} & K \arrow{r} & \Gmot(\langle\mathcal{H}(X)\rangle^{\mathrm{ev}}) \arrow{r}{\pi_{2,\operatorname{mot}}^{\mathrm{ev}}} & \Gmot(\mathcal{H}^2(X)) \arrow{r} & 1
	\end{tikzcd}
\]
The rightmost vertical map is an isomorphism by assumption. The snake lemma now yields that $\ker(j)=\ker(q^{\mathrm{ev}})$, which shows that $ \ker(q^{\mathrm{ev}})\subset \ker(q_A^{\mathrm{ev}})$.
\end{proof}

\subsection{What does the defect group measure?}\label{subsec:Justification}
With the structure result of the motivic Galois group being proved in Theorem \ref{thm:product2}, we can deduce that the defect group indeed grasps the essential difficulty of meta-conjecture \ref{conj:meta} for Andr\'e motives.
%See Corollary \ref{cor:Pdefect} for the precise statement.

\begin{corollary}
\label{cor:Pdefect2}
    For any projective hyper-K\"ahler variety $X$ with $b_2(X)\neq 3$, the following conditions are equivalent:
    \begin{enumerate}[$(i^+)$]
        \item The even defect group $P^+(X)$ is trivial.
        \item The even Andr\'e motive $\mathcal{H}^+(X)$ is in the tannakian subcategory generated by $\mathcal{H}^2(X)$.
        \item $\mathcal{H}^+(X)$ is abelian.
        \item Conjecture \ref{mhc} holds for $\mathcal{H}^+(X)$: $\MT(H^+(X))=\Gmot(\mathcal{H}^+(X)).$
    \end{enumerate}
    Similarly, if some odd Betti number of $X$ is not zero, we have the following equivalent conditions:
    \begin{enumerate}[$(i)$]
        \item The defect group $P(X)$ is trivial.
        \item The Andr\'e motive $\mathcal{H}(X)$ is in the tannakian subcategory generated by $\mathcal{H}^1(\mathrm{KS}(X))$, where $\mathrm{KS}(X)$ is any Kuga--Satake abelian variety associated to $H^2(X)$.
        \item $\mathcal{H}(X)$ is abelian.
        \item Conjecture \ref{mhc} holds for $\mathcal{H}(X)$: $\MT(H^*(X))=\Gmot(\mathcal{H}(X))$.
    \end{enumerate}
\end{corollary}

\begin{proof}%[Proof of Corollary \ref{cor:Pdefect}]
We first treat the even motive.
It follows immediately from Theorem \ref{thm:product2} that $(i^+)$ and  $(iv^+)$ are equivalent.\\
$(i^+) $ implies $(ii^+)$: By the definition of $P^+(X)$, if it is trivial, then the natural surjection $\Gmot(\mathcal{H}^+(X))\to \Gmot(\mathcal{H}^2(X))$ is an isomorphism. Then $(ii^+)$ follows from the Tannaka duality.\\
The implication from $(ii^+)$ to $(iii^+)$ follows from the fact that $\mathcal{H}^2(X)$ is abelian, which is Andr\'e's Theorem \ref{abelianityH2}.\\
Finally, $(iii^+)$ implies $(iv^+)$ thanks to Andr\'e's Theorem \ref{mhcabelian}.

In the presence of non-vanishing odd Betti numbers, the proof is similar to the even case: the equivalence of $(i)$ and $(iv)$ is immediate from Theorem \ref{thm:product2}. $(ii)$ obviously implies $(iii)$; $(iii)$ implies $(iv)$ by Andr\'{e}'s Theorem \ref{mhcabelian}. Finally, let us show that $(i)$ implies $(ii)$: if $P(X)$ is trivial then $\Gmot(\mathcal{H}(X))\to \Gmot(\mathcal{H}^1(A))$ is an isomorphism, where $A$ is any Kuga--Satake variety for $H^2(X)$ in the sense of Definition \ref{def:Kuga-Satake}. Therefore, $\mathcal{H}(X)$ is in $\langle \mathcal{H}^1(A) \rangle_{\AM}$ by Tannaka duality.
%except that the implication $(i)\Longrightarrow (ii)$ requires some more argument: let $A$ be a Kuga--Satake abelian variety for $\mathcal{H}^2(X)$ (see Definition \ref{def:Kuga-Satake} in Appendix \ref{Appendix}); we have seen that $\mathcal{H}^1(A)\in\langle\mathcal{H}(X)\rangle_{\AM}$, and that the defect group $P(X)$ is identified with the kernel of the corresponding morphism $\Gmot(\mathcal{H}(X))\to \Gmot(\mathcal{H}^1(A))$. Hence, if $P(X)$ is trivial, this morphism is an isomorphism and $ \mathcal{H}(X)\in \langle \mathcal{H}^1(A)\rangle$.
%except that the implication $(i)\Longrightarrow (ii)$ requires some %more argument. 
%In \cite[Theorem 4.1]{KSV2017}, Kurnosov--Soldatenkov--Verbitsky %constructed  an abelian variety $\mathrm{KS}_{\mathrm{KSV}}(X)$, %with an embedding of Hodge structures $j\colon H(X)\subset %H(\mathrm{KS}_{\mathrm{KSV}}(X))^{\otimes k}$, for some positive %integer $k$. As $H(\mathrm{KS}_{\mathrm{KSV}}(X))=\bigwedge %H^1(\mathrm{KS}_{\mathrm{KSV}}(X))$, this implies that the total %Hodge structure $H(X)$ belongs to the tannakian subcategory %generated by $H^1(\mathrm{KS}_{\mathrm{KSV}}(X))$. If Conjecture %\ref{mhc} holds for $\mathcal{H}(X)$, we can lift the embedding $j$ %to an embedding of Andr\'{e} motives $j\colon \mathcal{H}(X)\subset %\mathcal{H}(\mathrm{KS}_{\mathrm{KSV}}(X))^{\otimes k}$.
\end{proof}

\subsection{Deformation invariance}
We have seen in the above proof that the action of the defect group commutes with the $\mathrm{LLV}$-Lie algebra. We prove now that defect groups are deformation invariant in algebraic families.
%see Theorem \ref{thm:DeformationInv} for the statement.
The relevant notation and results are recalled in \S\ref{subsubsec:familiesAM}. Let $f\colon \mathcal{X}\to S$ be a smooth and proper family over a non-singular quasi-projective variety~$S$ such that all fibres $X_s$ are projective hyper-K\"{a}hler varieties with $b_2\neq 3$. We have naturally the following generalized Andr\'{e} motives over $S$ (Definition \ref{def:RelativeAM}): $\mathcal{H}(\mathcal{X}/S)$, $\mathcal{H}^{i}(\mathcal{X}/S)$ and $\mathcal{H}^{+}(\mathcal{X}/S)$.
%\[ 
%\mathcal{H}(Y/S)=\bigoplus_i \mathcal{H}^{i}(Y/S),
%\]
%which is the sum of its K\"{u}nneth components. The realization of $\mathcal{H}^i(Y/S)$ is the weight $i$ variation of Hodge structures $H^{i}(Y/S)=R^{i}f_*\Q$. We also consider the subfamily of even motives $\mathcal{H}^+(Y/S)=\bigoplus_i \mathcal{H}^{2i}(Y/S)$.
Up to replacing $S$ by an \'etale cover, we can assume that the algebraic monodromy group $\mathrm{G}_{\mathrm{mono}}(\mathcal{H}(\mathcal{X}/S))$ is connected.

%\begin{proposition}\label{prop:deformation}
%For all $s, s'\in S$ we have natural identifications $P^+(Y_s)\cong P^+(Y_{s'})$. The same conclusion holds for the full defect group $P(X)$.
%\end{proposition}

\begin{theorem}[Deformation invariance of defect groups]\label{thm:DeformationInv2}
    Let $S$ be a smooth quasi-projective variety and $\mathcal{X}\to S$ be a smooth proper morphism with fibers being projective hyper-K\"ahler manifolds with $b_2\neq 3$. Then for any $s, s'\in S$, the defect groups $P(X_s)$ and $P(X_{s'})$ are canonically isomorphic, and similarly for the even defect groups.
\end{theorem}

\begin{proof}%[Proof of Theorem \ref{thm:DeformationInv}]
We prove first the invariance of the even defect group.
%the same argument proves the invariance of the full defect group if the odd cohomology does not vanish. 
For any point $s\in S$, we have $\Gmot(\mathcal{H}^+(X_{s}))=\MT(H^+(X_{s})) \times P^+(X_{s})$ by Theorem \ref{thm:product2}. Let $s_0\in S$ be a very general point. By Theorem \ref{thm:familiesMotives} $(i)$ and $(ii)$, we have $\mathrm{G}_{\mathrm{mono}}(\mathcal{H}^+(\mathcal{X}/S))_{s_0} \subset \MT(H^+(X_{s_0}))$. Hence, the monodromy acts trivially on $P^+(X_{s_{0}})$, which therefore extends to a constant local system $P^+(\mathcal{X}/S)$ such that we have a splitting 
\[
\Gmot(\mathcal{H}^+(\mathcal{X}/S))=\MT(H^+(\mathcal{X}/S))\times P^+(\mathcal{X}/S)
\]
of local systems of algebraic groups over $S$. The local system $P^+(\mathcal{X}/S)$ is identified with the kernel of the natural morphism of generic motivic Galois groups 
\[
\Gmot(\mathcal{H}^+(\mathcal{X}/S)) \twoheadrightarrow \Gmot(\mathcal{H}^2(\mathcal{X}/S)).
\]

For any $s\in S$ we have the inclusion of $\Gmot(\mathcal{H}^+(X_s))$ into $\Gmot(\mathcal{H}^+(\mathcal{X}/S))_s$, which restricts to the inclusions $\MT(H^+(X_s))\hookrightarrow \MT(H^+(\mathcal{X}/S))_s$ and $P^+(X_s)\hookrightarrow P^+(\mathcal{X}/S)_s$.

It is enough to show that, for all $s\in S$, the equality  $P^+(X_s)=P^+(\mathcal{X}/S)_s$ holds.\\
By Theorem~\ref{thm:familiesMotives}$(iv)$, we have 
%\[
%\mathrm{G}_{\mathrm{mono}}(H^+(\mathcal{X}/S))_s\cdot %\MT(H^+(\mathcal{X}_s))= \MT(H^+(\mathcal{X}/S))_s
%\]
%and
\[
		\mathrm{G}_{\mathrm{mono}}(\mathcal{H}^+(\mathcal{X}/S))_s \cdot \Gmot(\mathcal{H}^+(X_s)) = \Gmot(\mathcal{H}^+(\mathcal{X}/S))_s.
\]
But we know that $\mathrm{G}_{\mathrm{mono}}(\mathcal{H}^+(\mathcal{X}/S))_s$ is contained in
\[
\{ 1 \}\times \MT(H^+(\mathcal{X}/S))_s \subset P^+(\mathcal{X}/S)_s \times \MT(H^+(\mathcal{X}/S))_s=\Gmot(\mathcal{H}^+(\mathcal{X}/S))_s,
\]
and therefore we have
\begin{align*} 
	\mathrm{G}_{\mathrm{mono}}(\mathcal{H}^+(\mathcal{X}/S))_s  \cdot \Gmot & (\mathcal{H}^+(X_s)) = \\ 
	 & = \mathrm{G}_{\mathrm{mono}}(\mathcal{H}^+(\mathcal{X}/S))_s \cdot \bigl(P^+(X_s)\times \MT^+(X_s)\bigr) \\
	 & = P^+(X_s) \times \bigl (\mathrm{G}_{\mathrm{mono}}(\mathcal{H}^+(\mathcal{X}/S))_s\cdot \MT^+(X_s) \bigr),
\end{align*}
which forces $P^+(X_s)=P^+(\mathcal{X}/S)_s$.

In presence of non-vanishing Betti numbers in odd degree, the proof is similar. Again, choosing a very general point $s_0\in S$, we obtain a local system $P(\mathcal{X}/S)$ with fiber $P(X_{s_0})$ such that 
\[
\Gmot(\mathcal{H}(\mathcal{X}/S))=\MT(H^*(\mathcal{X}/S))\times P(\mathcal{X}/S).
\]
The Kuga-Satake construction can be performed in families, see \cite{deligne1971conjecture}, to obtain a smooth proper family $\mathcal{A}\to S$ such that $\mathcal{A}_s$ is a Kuga-Satake variety for $H^2(X_s)$ in the sense of Definition \ref{def:Kuga-Satake}, for all $s$. Thanks to Lemma \ref{lem:KSvsodd} we have a natural morphism of generic motivic Galois groups
\[
\Gmot(\mathcal{H}(\mathcal{X}/S))\twoheadrightarrow \Gmot(\mathcal{H}^1(\mathcal{A}/S))
\]
which does not depend on any choice involved in the construction of $\mathcal{A}$; the local system $P(\mathcal{X}/S)$ is identified with the kernel of the morphism above.

It follows that for any $s\in S$ the inclusion $ \Gmot(\mathcal{H}(X_s))\hookrightarrow  \Gmot(\mathcal{H}(\mathcal{X}/S))_s$ restricts to inclusions $\MT({H}^*(X_s))\hookrightarrow  \MT({H}^*(\mathcal{X}/S))_s$
and $P(X_s)\hookrightarrow P(\mathcal{X}/S)_s $. Now we conclude via the same argument given for the even case. 
\end{proof}

\section{Applications}

\subsection{Andr\'e motives of hyper-K\"ahler varieties}\label{subsec:KnownHK}
As we have seen in Theorem \ref{thm:DeformationInv2}, the defect group does not change along smooth proper algebraic families. In fact, the defect group is invariant in the whole deformation class.  

\begin{corollary}\label{cor:deformationDefect}
	Let $X$ and $X'$ be two deformation equivalent projective hyper-K\"{a}hler varieties with $b_2\neq 3$. Then their defect groups are isomorphic: $P^+(X)\cong P^+(X')$ and $P(X)\cong P(X')$. 
\end{corollary}
\begin{proof}
Pick two deformation equivalent projective hyper-K\"{a}hler varieties $X$ and $X'$ with $b_2\neq 3$. It has been shown by Soldatenkov ({\cite[\S6.2]{soldatenkov19}}) that there exists finitely many smooth proper algebraic families $f_i\colon  Y^i\to S_i$, $i=0,1,\dots,k$ over smooth quasi-projective curves $S_i$ and points $a_i,b_i\in S_i$ together with isomorphisms
	\[
	X\cong Y^{0}_{a_0}, \ 	\	 Y^{i}_{b_i} \cong Y^{i+1}_{{a_{i+1}}}, \ \mathrm{for} \ i=1,2,\dots, k-1, \ \mathrm{and}\  	Y^{k}_{b_k}\cong X'.
	\]
 We therefore find a chain of smooth proper families with projective fibers connecting $X$ and $X'$. The conclusion now follows via an iterated application of Theorem \ref{thm:DeformationInv2}.
\end{proof}

\begin{corollary}\label{cor:application}
Fix a deformation class of compact hyper-K\"{a}hler manifolds with $b_2\neq 3$. 
	\begin{enumerate}[(i)]
	\item (Soldatenkov \cite{soldatenkov19}) If one projective hyper-K\"ahler variety in the deformation class has abelian Andr\'e motive, then so does any other projective member in this class.
	\item There exists an Andr\'{e} motive $\mathcal{D}^+$ depending only on the deformation class, with Hodge realization being of Tate type, and  such that for any projective hyper-K\"ahler variety $X$ in this deformation class we have
	\[
	\langle \mathcal{H}^+(X) \rangle_{\AM} = \langle\mathcal{H}^2(X), \mathcal{D}^+ \rangle_{\AM}. 
	\]
	\item Similarly, if some odd Betti number is non-zero in the chosen deformation class, there exists an Andr\'{e} motive $\mathcal{D}$ depending only on the deformation class, with Hodge realization being of Tate type, and such that for any projective $X$ in the chosen deformation class we have
	\[
	\langle \mathcal{H}(X) \rangle_{\AM} =\langle\mathcal{H}^1(\mathrm{KS}(X)), \mathcal{D} \rangle_{\AM}\, ,
	\]
	where $\KS(X)$ is any Kuga--Satake variety for $H^2(X)$ (Definition \ref{def:Kuga-Satake}).
\end{enumerate} 
\end{corollary}
\begin{proof}
$(i)$ follows from the combination of Corollary \ref{cor:Pdefect2} and Corollary \ref{cor:deformationDefect}.\\
$(ii)$. This follows via a reinterpretation of Theorem \ref{thm:product2} in terms of a defect motive. Recall that we have $\Gmot(\mathcal{H}^+(X))=\MT(H^+(X))\times P^+(X) $. 
	%As all motives in $\langle \mathcal{H}^+(X)\rangle\subset \AM_k$ have even weight, the category of representations of $\SGmot(\mathcal{H}^+(X))$ is equivalent to the tannakian subcategory of $\AM_{k}$ generated by the weight $0$ motives $\mathcal{H}^{2i}(X)(i)$ (because in this case $\SGmot(\mathcal{H}^+(X))=\Gmot(\bigoplus_i \mathcal{H}^{2i}(X)(i))$). 
	The category $\mathrm{Rep}(P^+(X))$ can be seen as the subcategory of $\langle \mathcal{H}^+(X)\rangle_{\AM}$ on which $\MT(H^+(X))$ acts trivially, \ie, it consists of the motives in $\langle \mathcal{H}^+(X)\rangle$ whose realization is of Tate type. 
	By \cite[Proposition 3.1]{deligne1982hodge}, the category $\mathrm{Rep}(P^+(X))$ is generated as a tannakian category by any faithful representation of $P^+(X)$; choosing one such representation determines a motive $\mathcal{D}^+(X)\in \langle\mathcal{H}^+(X)\rangle$, such that inside $\AM$,
	\[
	\langle\mathcal{H}^+(X)\rangle = \langle \mathcal{D}^+(X), \mathcal{H}^2(X) \rangle. 
	\]
	Let now  $\mathcal{X}\to S$ be a smooth proper family with fibres projective hyper-K\"{a}hler varieties with $b_2\neq 3$ over a smooth quasi-projective base $S$. We assume that the monodromy group $G_{\mathrm{mono}}(\mathcal{X}/S)$ is connected. We consider the generalized Andr\'{e} motive $\mathcal{H}^+(\mathcal{X}/S)$ over $S$, with realization $H^+(\mathcal{X}/S)$; by Theorem \ref{thm:DeformationInv2}, we have a splitting of local systems of algebraic groups $\Gmot(\mathcal{H}^+(\mathcal{X}/S))=\MT({H}^+(\mathcal{X}/S))\times P^+(\mathcal{X}/S)$ such that $P^+(X_s)=P^+(\mathcal{X}/S)_s$ for all $s\in S$. 
We choose a tensor construction ${T}^+(\mathcal{X}/S) = {H}^+(\mathcal{X}/S)^{\otimes a} \otimes {H}^+(\mathcal{X}/S)^{\vee, \otimes b}$ such that the subspace $W^+(X_s)\subset T^+({X}_s)$ of $\MT(H^+(\mathcal{X}/S))_s$-invariants is a faithful $P^+(X_s)$-representation. Since $W^+(X_s)$ is stable for the action of $\Gmot(\mathcal{H}^+(\mathcal{X}/S))_s$, we obtain generalized Andr\'{e} motives $\mathcal{W}^+(\mathcal{X}/S) \subset \mathcal{T}^+(\mathcal{X}/S)$ over $S$, with realizations $W^+(\mathcal{X}/S) \subset T^+(\mathcal{X}/S)$. 
For all $s\in S$ we have $\langle \mathcal{H}^+({X}_s)\rangle =\langle \mathcal{W}^+(X_s), \mathcal{H}^2(X_s) \rangle$.

Since the monodromy group is connected, by Theorem \ref{thm:familiesMotives}(i) the local system $W^+(\mathcal{X}/S)$ is constant. Now Theorem \ref{thm:DefPrinciple} implies that for any two points $s_0,  s_1 \in S$ we have an isomorphism of motives $\mathcal{W}^+(X_{s_0})\cong \mathcal{W}^+(X_{s_1})$. 
In fact, let $0$ be any point of $S$ and let $\mathcal{D}^+=\mathcal{W}^+(X_0)$. Let $\mathcal{D}^+/S$ be the constant generalized Andr\'{e} motive over $S$ with fibre $\mathcal{D}^+$, supported onto the constant local system $D^+/S$. Then the identity $\id \colon {W}^+(X_{0})\to ({D}^+/S)_0$ is monodromy invariant and motivated, and hence it extends to a global section $\xi$ of the local system $\Hom(W^+(\mathcal{X}/S), D^+/S)$ such that $\xi_s $ is the realization of an isomorphism of motives $\mathcal{W}^+(X_s)\cong \mathcal{D}^+_{s}$, for any $s\in S$; hence, we have $\langle \mathcal{H}^+({X}_s)\rangle = \langle \mathcal{D}^+, \mathcal{H}^2({X}_s)\rangle$. 
Thanks to {\cite[\S6.2]{soldatenkov19}}, we can connect any two deformation equivalent projective hyper-K\"{a}hler varieties with $b_2\neq 3$ via finitely many families as above and iterate the argument given. \\
(iii). Same argument as above.
\end{proof}

We can now prove that the defect group of any known projective hyper-K\"{a}hler variety is trivial. 

\begin{proof}[Proof of Corollary \ref{cor:KnownHKs}]
The second Betti numbers of known hyper-K\"{a}hler varieties are as follows: $22$ for $\mathrm{K3}$ surfaces \cite{huyK3}; $23$ and $7$ for varieties of $\mathrm{K3}^{[n]}$-type and of generalized Kummer type respectively, see \cite{beauville1983varietes}; $24$ for varieties of $\mathrm{OG}10$-type and $8$ for those of $\mathrm{OG}6$-type, as computed by Rapagnetta in \cite{Rap08} and \cite{Rap07}. Hence, the triviality of the defect group is a deformation invariant property by Corollary \ref{cor:deformationDefect}. It is therefore enough to find in each of the known deformation classes a representative whose defect group is trivial, or equivalently, whose Andr\'{e} motive is abelian.\\
$(i)$ For $\mathrm{K3}$ surfaces, this is Andr\'{e} \cite[Th\'{e}or\`{e}me 0.6.3]{andre1996Motives}.\\
$(ii)$ For the $\mathrm{K3}^{[n]}$-type, the motivic decomposition of de Cataldo--Migliorini \cite{deCataldoMigliorini2002}, together with the case of K3 surfaces $(i)$, implies that the Andr\'e motive of a Hilbert scheme of a $\mathrm{K3}$ surface is abelian.\\
$(iii)$ For the generalized Kummer type, using the work of Cataldo--Migliorini \cite{dCM04} on semi-small resolutions, a motivic decomposition for a generalized Kummer variety associated to an abelian surface in terms of abelian motives was obtained in \cite{Xu18} and \cite[Corollary 6.3]{FTV19}.\\
$(iv)$ For the O'Grady-6 deformation type, it follows from \cite{MRS18}, as observed by Soldatenkov \cite{soldatenkov19}: in \cite{MRS18}, some hyper-K\"ahler variety of this deformation type was constructed as the quotient of some hyper-K\"ahler variety of $\mathrm{K3}^{[3]}$-type by a birational involution (with well-understood indeterminacy loci). One can then conclude by $(ii)$.\\
$(v)$ For O'Grady-10 deformation type, we use Corollary \ref{cor:InfChowAb}.
\end{proof}

\subsection{Motivated Mumford--Tate conjecture}\label{subsec:MTC}
We first recall a strengthening of the Mumford--Tate conjecture involving motivic Galois groups, see Moonen's survey \cite[\S 3.2]{moonen17} for details.
In the sequel, $k$ is a finitely generated subfield of $\C$, and $\ell$ is a prime number. Attached to a smooth projective variety $X$ defined over $k$, we have on the one hand the rational singular (Betti) cohomology $H^*_{\mathrm{B}}(X)=\bigoplus_{i} H^{i}_{\mathrm{B}}(X):=\bigoplus_i H^i(X_\C^{\operatorname{an}}, \Q)$, naturally equipped with a Hodge structure, and on the other hand the $\ell$-adic \'{e}tale cohomology $H_{\ell}^*(X)=\bigoplus_{i}H_{\ell}^{i}(X):=\bigoplus_{i}H_{\text{\'{e}t}}^{i}(X_{\bar{k}},\Q_{\ell})$, which is a continuous $\Q_{\ell}$-representation of $\Gal(\bar{k}/k)$. These two cohomology theories provide realization functors from $\AM(k)$, the category of Andr\'e motives over $\Spec(k)$: $$r_{\mathrm{B}}\colon \AM(k) \to \mathrm{HS}^{\mathrm{pol}}_{\Q};$$ $$r_{\ell}\colon \AM(k) \to \mathrm{Rep}_{\Q_{\ell}}\left(\Gal(\bar{k}/k)\right).$$
%such that $r_{\mathrm{B}}(\mathcal{H}(X))$ is the polarizable Hodge structure $H_{\mathrm{B}}(X)$ and $r_{\ell}(\mathcal{H}(X))$ is the $\Q_{\ell}$-Galois representation on the \'{e}tale cohomology $H_{\ell}(X)$ for any smooth and projective $k$-variety $X$.

Given a Galois representation $\sigma\colon \Gal(\bar{k}/k)\to \GL(V)$ on a $\Q_{\ell}$-vector space $V$, we let $\mathcal{G}(V)$ denote the $\Q_{\ell}$-algebraic subgroup of $\GL(V)$ which is the Zariski closure of the image of $\sigma$. This algebraic group is not necessarily connected, but becomes so after a finite field extension of $k$. It is not known to be reductive in general. The category of $\Q_{\ell}$-Galois representations is a neutral tannakian abelian category, and the tannakian subcategory $\langle V \rangle$ is equivalent to the category of finite dimensional $\Q_\ell$-representations of $\mathcal{G}(V)$. 
%Similarly, the composition of $r_\ell$ with the forgetful functor provides another fiber functor for the neutral tannakian category $\AM(\bar{k})$.

These different realizations are related via Artin's comparison theorem: for any $M\in \AM(k)$ there is a canonical isomorphism of $\Q_{\ell}$-vector spaces 
$\gamma \colon r_{\mathrm{B}}(M) \otimes \Q_{\ell} \cong r_{\ell}(M)$. This gives rise to an isomorphism of $\Q_\ell$-algebraic groups $\gamma\colon \GL(r_{\operatorname{B}}(M))\otimes \Q_\ell\cong \GL(r_\ell(M))$, under which $\Gmot(M_{\C})\otimes \Q_\ell$ is identified with $\Gmot_{, \ell}(M_{\bar k})$, where the latter is the motivic Galois group of the tannakian category $\langle M_{\bar k}\rangle_{\AM({\bar k})}$ with fiber functor $r_{\ell}$ composed with the forgetful functor. The following conjecture is a motivic extension of the Mumford--Tate conjecture \cite{Mum66}.

\begin{conjecture}[Motivated Mumford--Tate conjecture]\label{conj:mtcMot}
The canonical isomorphism $\gamma$ induces identifications of $\Q_{\ell}$-algebraic groups
\[
\MT(r_{\mathrm{B}}(M)) \otimes \Q_{\ell} 
= \Gmot(M_{\C})\otimes \Q_{\ell} \cong \Gmot_{, \ell}(M_{\bar k}) =\mathcal{G}(r_{\ell}(M))^0.
\]
\end{conjecture}
\begin{remark}
The first equality is the content of Conjecture \ref{mhc} and the last equality is the analogous statement saying that all Tate classes are motivated. The original statement of the Mumford--Tate conjecture only predicts that under $\gamma$ we have
\[
\MT(H^*_{\mathrm{B}}(X)) \otimes \Q_{\ell} = \mathcal{G}(H^*_{\ell}(X))^0,
\]
for any smooth and projective variety $X$ over $k$.
\end{remark}

Let us define a hyper-K\"{a}hler variety over $k$ to be a smooth projective variety $X$ over $k$ such that $X_\C$ is a hyper-K\"{a}hler variety. The following result confirms Conjecture \ref{conj:mtcMot} for the degree-2 part of the motive of $X$, see Moonen \cite{moonen2017} for some generalizations.
\begin{theorem}[Andr\'e \cite{Andre1996}]\label{thm:mtcH2}
    Let $X$ be a  hyper-K\"{a}hler variety defined over $k$ with $b_2\neq 3$. Then the motivated Mumford--Tate conjecture holds for the Andr\'{e} motive $\mathcal{H}^2(X)$.
\end{theorem}

Let $X$ be as above. Then $\Gmot(\mathcal{H}(X_{\bar{k}}))\cong \Gmot(\mathcal{H}(X_\C))=\MT(H^*_{\mathrm{B}}(X)) \times P(X)$ by Theorem~\ref{thm:product2}.

\begin{proposition}\label{mtcDefect}
    If $P^+(X)$ is finite (\resp trivial), then the Mumford--Tate conjecture (\resp the motivated Mumford--Tate conjecture) holds for the motive $\mathcal{H}^+(X)$.
    If $P(X)$ is finite (\resp trivial), then the Mumford--Tate conjecture (\resp the motivated Mumford--Tate conjecture) holds for the motive $\mathcal{H}(X)$.
\end{proposition}
\begin{proof}
Let us identify $\Gmot(M_{\bar k})\otimes \Q_\ell$ and $\Gmot_{, \ell}(M_{\bar k})$ using Artin's comparison isomorphism. Consider the following commutative diagram 
\[
\begin{tikzcd}
\MT(H_{\mathrm{B}}^+(X))\otimes \Q_{\ell} \arrow[hookrightarrow]{r}{\cong} \arrow["\cong"{sloped, above}]{d} & \Gmot(\mathcal{H}^+(X_{\bar{k}}))^0 \otimes\Q_\ell \arrow[two heads]{d} \arrow[hookleftarrow]{r} & \mathcal{G} (H^+_{\ell}(X))^0\arrow[two heads]{d}\\
\MT(H_{\mathrm{B}}^2(X))\otimes\Q_{\ell} \arrow{r}{\cong} & \Gmot(\mathcal{H}^2(X_{\bar{k}}))\otimes\Q_\ell & \mathcal{G}(H^2_{\ell}(X))^0 \arrow[swap]{l}{\cong}
\end{tikzcd}
\]
The two horizontal morphisms on the bottom are isomorphisms due to Theorem~\ref{thm:mtcH2}, the vertical map on the left is an isomorphism thanks to Proposition~\ref{cor:MT} and the top left arrow is an isomorphism by Theorem \ref{thm:product2} since $P^+(X)$ is finite by assumption. It follows that all arrows in the diagram are isomorphisms, and so
\[
\mathcal{G}(H_{\ell}^+(X))^0 \cong \Gmot(\mathcal{H}^+(X_{\bar{k}}))^0 \otimes \Q_{\ell} \cong \MT(H_{\mathrm{B}}^+(X)) \otimes\Q_{\ell}.
\]
If $P^+(X)$ is actually trivial, then $\Gmot(\mathcal{H}^+(X_{\bar{k}}))$ is connected, and we conclude that the motivated Mumford--Tate conjecture holds for $\mathcal{H}^+(X)$ in this case.

If the odd cohomology of $X$ is trivial, we are done. Otherwise, assume that $P(X)$ is finite, which implies that also $P^+(X)$ is finite. We consider another commutative diagram
\[
\begin{tikzcd}
\MT (H^*_{\mathrm{B}}(X))\otimes \Q_{\ell}  \arrow[hookrightarrow]{r}{\cong}  \arrow["\sim"{sloped, above}]{d} & \Gmot(\mathcal{H}(X_{\bar{k}}))^0 \otimes \Q_\ell \arrow[two heads]{d} \arrow[hookleftarrow]{r} & \mathcal{G}(H^*_{\ell}(X))^0\arrow[two heads]{d} \\
\MT(H_{\mathrm{B}}^+(X)) \otimes \Q_{\ell}\arrow{r}{\cong} & \Gmot(\mathcal{H}^+(X_{\bar{k}}))^0 \otimes \Q_\ell & \mathcal{G}(H^+_{\ell}(X))^0 \arrow[swap]{l}{\cong}
\end{tikzcd}
\]
The horizontal arrows on the bottom are isomorphisms due to the above; the top left horizontal map is an isomorphism by Theorem \ref{thm:product2}, since $P(X)$ is finite by assumption, while the leftmost vertical arrow is an isogeny due to Proposition~\ref{cor:MT}. It follows that also the middle vertical arrow is an isogeny. We deduce that $\Gmot(\mathcal{H}(X_{\bar{k}}))^0 \otimes \Q_\ell$ and $\mathcal{G}(H^*_{\ell}(X))^0$ are connected algebraic groups of the same dimension over $\Q_{\ell}$. Hence, the inclusion
\[
\mathcal{G}(H^*_{\ell}(X))^0 \hookrightarrow \Gmot(\mathcal{H}(X_{\bar{k}}))^0 \otimes\Q_{\ell}
\]
is an isomorphism. If $P(X)$ is actually trivial, then $\Gmot(\mathcal{H}(X_{\bar{k}}))$ is connected, and we conclude that the motivated Mumford--Tate conjecture holds for the full Andr\'{e} motive $\mathcal{H}(X)$. 
\end{proof}
 %(Observed in \cite{floccari2019}?)

\begin{definition}
Let $k\subset \C$ be a finitely generated field. Define $\mathcal{C}_k$ to be the tannakian subcategory of $\AM(k)$ generated by the motives of all hyper-K\"ahler varieties whose associated complex manifold is of one of the four known deformation types.
\end{definition}

\begin{remark}
Note that this category contains already the motive of cubic fourfolds, as they are motivated by their Fano varieties of lines (see for example \cite{Lat17}). Very likely, $\mathcal{C}_k$ also contains the motive of some interesting Fano varieties whose cohomology is of K3-type, for instance, Gushel--Mukai varieties \cite{Gus83} \cite{Muk89}, Debarre--Voisin Fano varieties \cite{DV10} and many more \cite{MF19}.
\end{remark} 

\begin{theorem}\label{thm:abelianHKmotives1}
    The motivated Mumford--Tate conjecture holds for any motive $M\in \mathcal{C}_k$. In particular, for any smooth projective variety motivated by a product of projective hyper-K\"{a}hler varieties of known deformation type, the Hodge conjecture and the Tate conjecture are equivalent.
\end{theorem}
\begin{proof}
    By Commelin \cite[Theorem 10.3]{commelin2019}, the subcategory of abelian Andr\'e motives satisfying the Mumford--Tate conjecture\footnote{For abelian Andr\'e motives, the Mumford--Tate conjecture is equivalent to its motivated version \ref{conj:mtcMot}, thanks to Andr\'e's result Theorem \ref{mhcabelian}.} is a tannakian subcategory. Therefore, it suffices to check the abelianity and the Mumford--Tate conjecture for the generators of $\mathcal{C}_k$.\\
    %in Definition \ref{def:KnownHK}.\\ 
By Corollary \ref{cor:KnownHKs} the defect group of any hyper-K\"{a}hler variety $X$ of known deformation type is trivial. Hence, the motive $\mathcal{H}(X)\in \mathcal{C}_k$ is abelian by Corollary \ref{cor:Pdefect2}, and the motivated Mumford--Tate conjecture holds for its Andr\'e motive by Proposition \ref{mtcDefect}. 
\end{proof}

\begin{remark}
Thanks to \cite{commelin2019}, we can put even more generators in the category $\mathcal{C}_k$ to obtain new evidence for the Mumford--Tate conjecture. Since the conjecture is known to hold for 
\begin{enumerate}[$(i)$]
    \item geometrically simple abelian varieties of prime dimension, by Tankeev \cite{Tan83},
   \item abelian varieties of dimension $g$ with trivial   endomorphism ring over $\bar k$ such that $2g$ is
    neither a $k$-th power for some odd $k>1$ nor of the form     ${2k}\choose{k}$ for some odd $k>1$, thanks to Pink \cite{Pin98},
   % \item surfaces with $p_g=1$ and such that the Hodge structure $H^2$ varies when deforming the surface, by Moonen \cite{moonen2017},
\end{enumerate}
we deduce that the Mumford--Tate conjecture holds for any variety motivated by a product of varieties in $(i)$ and $(ii)$ above and hyper-K\"{a}hler varieties of the known deformation types. See Moonen \cite{moonen2017} for more potential examples. 
\end{remark}

\appendix
\section{The Kuga--Satake category} \label{Appendix}

%Let $V$ be a polarizable rational Hodge structure of $\mathrm{K3}$-type, \ie~ $V$ is pure of weight~2 with $h^{2,0}=h^{0,2}=1$ and $h^{p,q}=0$ whenever $p$ or $q$ are negative. The classical Kuga--Satake construction (\cite{deligne1971conjecture}) produces an abelian variety $\mathrm{KS}(V)$, closely related to $V$. Our main object of focus will be the following \textit{Kuga--Satake category} of $V$: $$\mathsf{KS}(V):=\langle H^1(\mathrm{KS}(V))\rangle\subset \HS^{\operatorname{pol}}_\Q,$$ rather than the abelian variety $\KS(V)$ itself.

Let $V$ be a polarizable rational Hodge structure of $\mathrm{K3}$-type, \ie~ $V$ is pure of weight~2 with $h^{2,0}=h^{0,2}=1$ and $h^{p,q}=0$ whenever $p$ or $q$ is negative. The Kuga--Satake construction \cite{deligne1971conjecture} produces an abelian variety $\mathrm{KS}(V)$ closely related to $V$, which is defined up to isogeny.
This isogeny class is not unique, but the main point of this Appendix is to characterize the tannakian subcategory of Hodge structures generated by this abelian variety, which we  call the \emph{Kuga--Satake category} attached to $V$, $$\mathsf{KS}(V):=\langle H^1(\mathrm{KS}(V))\rangle\subset \HS^{\operatorname{pol}}_\Q. $$
In the appendix, all the cohomology groups are with rational coefficients and the notation $\langle - \rangle$ means the generated tannakian subcategory inside $\HS^{\operatorname{pol}}_\Q$, if not otherwise specified. We first briefly review the classical construction.
%The main result is Theorem \ref{thm:kuga-satake}.

\subsection{The Kuga--Satake construction}\label{subsec:KSconstruction}

Choose a polarization $q$ of $V$, and consider the Clifford algebra $\mathrm{Cl}(V,q)$. Deligne showed in \cite{deligne1971conjecture} that there is a unique way to induce a weight-1 effective Hodge structure on $\mathrm{Cl}(V,q)$, which is polarizable and therefore equals $H^1(\KS(V))$ for some abelian variety $\KS(V)$, well-defined up to isogeny.
The key relation between $V$ and $\KS(V)$ is the fact that the natural action of $V$ on $\mathrm{Cl}(V,q)$ via left multiplication yields an embedding of Hodge structures
\[
V(1)\hookrightarrow H^1(\KS(V))\otimes H^1(\KS(V))^{\vee}.
\]

Consider the weight cocharacters $w_V\colon \mathbb{G}_{m}\to \GL(V)$ and $w_{\KS(V)}\colon \mathbb{G}_{m}\to \GL(H^1(\KS(V)))$, defined by $w_V (\lambda)=\lambda^2\cdot \id$ and $w_{\KS}(\lambda)=\lambda\cdot \id$ respectively, for all $\lambda$; we have $\MT(V)\subset w_V(\mathbb{G}_{m}) \cdot \SO(V,q)$ and $\MT(H^1(\KS(V)))\subset w_{\KS}(\mathbb{G}_{m}) \cdot \Spin(V,q)$, The inclusion $\langle V \rangle \subset \langle H^1(\KS(V))$ induces a surjective morphism $\phi\colon \MT(V)\to \MT(H^1(\KS(V)))$, which we claim is a double cover. Indeed, there is a commutative diagram with exact rows 
\[
\begin{tikzcd}
1 \arrow{r}
& \Spin(V,q)\cap\MT(H^1(\KS(V))) \arrow{d}{\phi'} \arrow{r}
& \MT(H^1(\KS(V))) \arrow{d}{\phi} \arrow{r}
& \mathbb{G}_{m} \arrow{d}{\cong} \arrow{r} 
& 1
\\
1 \arrow{r}
& \SO(V,q)\cap\MT(V) \arrow{r}
& \MT(V) \arrow{r}
& \mathbb{G}_{m} \arrow{r} 
& 1
\end{tikzcd}
\]
in which $\phi'$ is the restriction of the double cover $\Spin(V,q)\to \SO(V,q)$, and the vertical map on the right is an isomorphism due to the fact that $w_{\KS}(-1)\in \Spin(V,q)$.

\begin{remark}
The above construction can be performed given any non-degenerate quadratic form $q$ on $V$ such that the restriction of $q\otimes{\R}$ to $\bigl(H^{2,0}(V) \oplus H^{0,2}(V)\bigr) \cap (V\otimes \R) $ is positive definite and $q(\sigma)=0$ for any $\sigma \in H^{2,0}(V)$, see \cite[\S4, Remark 2.3]{huyK3}. 
\end{remark}

\subsection{The Kuga--Satake category} \label{subsec:KScategory}

Given a tannakian subcategory $\mathsf{C}\subset \mathrm{HS}_{\Q}^{\mathrm{pol}}$ we denote by $\mathsf{C}^{\mathrm{ev}}$ the full subcategory of~$\mathsf{C}$ consisting of objects of \emph{even} weight. The grading via weights on $\mathsf{C}$ is given by a central cocharacter $w\colon \mathbb{G}_{m,\Q}\to \MT(\mathsf{C})$. We let $\iota:= w(-1)$; it acts as $-1$ on any Hodge structure of odd weight in $\mathsf{C}$ and as the identity on $\mathsf{C}^{\mathrm{ev}}$. This means that, whenever $\mathsf{C}$ contains a Hodge structure of odd weight, the natural morphism of algebraic groups $ \MT(\mathsf{C}) \to \MT(\mathsf{C}^{\mathrm{ev}})$ is 
%an isomorphism if every Hodge structure in $\mathsf{C}$ has even weight, and it 
an isogeny with kernel the order $2$ cyclic group generated by $\iota$.

\begin{definition}\label{def:Kuga-Satake}
Let $V$ be a polarizable Hodge structure of K3-type. A \textit{Kuga--Satake variety} for $V$ is an abelian variety $A$ such that 
$\langle H^1(A) \rangle^{\mathrm{ev}}
= \langle V \rangle$.
\end{definition}

\begin{lemma}[Equivalent definition]\label{lemma:eqdefKS}
An abelian variety $A$ is a Kuga--Satake variety for $V$ if and only if $V\in \langle H^1(A) \rangle$ and the induced surjective morphism
$\MT(H^1(A))\to \MT(V)$ is an isogeny of degree $2$.
\end{lemma}
\begin{proof}
The only-if part is explained before. Conversely, assume that $V\in \langle H^1(A)\rangle$ and that the induced surjection $\MT(H^1(A))\to \MT(V)$ is an isogeny of degree $2$. Since $V$ has even weight, this morphism factors over $\MT(\langle H^1(A)\rangle^{\mathrm{ev}})\to \MT(V)$, and it follows that the the latter is an isomorphism. Hence, $ \langle H^1(A)\rangle^{\mathrm{ev}} = \langle V\rangle$.
\end{proof}

%Note that the classical Kuga--Satake construction from \cite{deligne1971conjecture} yields a Kuga--Satake abelian variety $\mathrm{KS}(V)$ for $V$ in the sense of Definition \ref{def:Kuga-Satake}, as by construction there is an embedding of Hodge structures 
%\[
%V \subset H^1(\mathrm{KS}(V))\otimes H^1(\mathrm{KS}(V))^{\vee}(-1)
%\]
%such that the induced morphism of Mumford--Tate groups $\MT(H^1(\mathrm{KS}(V))) \to \MT(V)$ is a degree-$2$ isogeny.

By Lemma \ref{lemma:eqdefKS} and the discussion in \S\ref{subsec:KSconstruction}, the abelian variety $\KS(V)$ is a Kuga--Satake variety for $V$ in the sense of our Definition \ref{def:Kuga-Satake}. It is clear that Kuga--Satake varieties are not unique, but the main observation of the appendix is that the corresponding Kuga--Satake category is~so.
\begin{theorem}\label{thm:kuga-satake}
    Let $V$ be a polarizable Hodge structure of $\mathrm{K3}$-type. Then there exists a unique tannakian subcategory $\mathsf{KS}(V)$ of $\mathrm{HS}^{\mathrm{pol}}_{\Q}$ such that \[
    \langle V\rangle = \mathsf{KS}(V)^{\mathrm{ev}} \subsetneq \mathsf{KS}(V).
    \] 
    If $A$ is any Kuga--Satake variety for $V$, we have $\langle H^1(A) \rangle =\mathsf{KS}(V)$.
\end{theorem}

Let us first prove the following straightforward lemma. 
Consider tannakian subcategories $\mathsf{C}\subset \mathsf{D}$ of $\mathrm{HS}_{\Q}^{\mathrm{pol}}$. Assume that both contain some Hodge structure of odd weight. The inclusion of $\mathsf{C}$ in $\mathsf{D}$ induces surjective homomorphisms of pro-algebraic groups $q\colon \MT(\mathsf{D})\to \MT(\mathsf{C})$ and $q^{\mathrm{ev}}\colon \MT(\mathsf{D}^{\mathrm{ev}})\to \MT(\mathsf{C}^{\mathrm{ev}})$. Let $\pi$ denote the double cover $\MT(\mathsf{D})\to \MT(\mathsf{D}^{\mathrm{ev}})$.
\begin{lemma}\label{rmk:oddHS}
In the above situation, the morphism $\pi\colon \MT(\mathsf{D})\to \MT(\mathsf{D}^{\mathrm{ev}})$ induces an isomorphism $\ker(q)\cong \ker(q^{\mathrm{ev}})$, and $\pi^{-1}(\ker(q^{\mathrm{ev}}))= \langle \iota\rangle\times \ker(q)$.
\end{lemma}
\begin{proof}
%Consider tannakian subcategories $\mathsf{C}\subset \mathsf{D}$ of $\mathrm{HS}_{\Q}^{\mathrm{pol}}$. Assume that both contains some Hodge structure of odd weight. The inclusion of $\mathsf{C}$ in $\mathsf{D}$ induces surjective homomorphisms of pro-algebraic groups $q\colon \MT(\mathsf{D})\to \MT(\mathsf{C})$ and $q^{\mathrm{ev}}\colon \MT(\mathsf{D}^{\mathrm{ev}})\to \MT(\mathsf{C}^{\mathrm{ev}})$, which fit in 
Consider the commutative diagram with exact rows
\[
\begin{tikzcd}
1\arrow{r} &\langle \iota \rangle \arrow{d}{\cong} \arrow{rr} && \MT(\mathsf{D})\arrow{d}[two heads]{q} \arrow{rr}{\pi} && \MT(\mathsf{D}^{\mathrm{ev}}) \arrow{d}[two heads]{q^{\mathrm{ev}}} \arrow{r}&1\\
1\arrow{r} &\langle\iota\rangle \arrow{rr} && \MT(\mathsf{C})\arrow{rr} && \MT(\mathsf{C}^{\mathrm{ev}})\arrow{r}&1.
\end{tikzcd}
\]
The snake lemma implies the isomorphism $\ker(q) \cong \ker(q^{\mathrm{ev}})$. Moreover, since $\iota\notin\ker(q)$ by assumption and it is central in $\MT(\mathsf{D})$, we have $\pi^{-1}(\ker{q}^{\mathrm{ev}})=\iota\times \ker(q)$.
%The same considerations holds with $\AM$ in place of $\mathrm{HS}^{\mathrm{pol}}_{\Q}$.
\end{proof}

\begin{proof}[Proof of Theorem \ref{thm:kuga-satake}]
%The existence is given by the classical Kuga--Satake construction. Let us show the uniqueness.
%which is a Kuga--Satake abelian variety for $V$ in the sense of Definition \ref{def:Kuga-Satake}. Let $\mathsf{KS}(V)$ be the tannakian category $\langle H^1(\mathrm{KS}(V)) \rangle$.
Assume given two tannakian subcategories $\mathsf{D}_1, \mathsf{D}_2 \subset \mathrm{HS}^{\mathrm{pol}}_{\Q}$, both containing some Hodge structure of odd weight and such that $\langle V\rangle=\mathsf{D}_i^{\mathrm{ev}}\subsetneq \mathsf{D}_i$ for $i=1, 2$. Let $\mathsf{E}$ be the tannakian subcategory generated by $\mathsf{D}_1$ and $\mathsf{D}_2$.
We have surjective morphisms of pro-algebraic groups
$q_i\colon \MT(\mathsf{E}) \to  \MT(\mathsf{D}_i)$, $i=1,2$. We claim that these are both isomorphisms. From the commutative diagram
\[
\begin{tikzcd}
\MT(\mathsf{E}^{\mathrm{ev}}) \arrow[two heads]{rr}{q_1^{\mathrm{ev}}} \arrow[two heads]{d}{q_2^{\mathrm{ev}}} && \MT(\mathsf{D}_1^{\mathrm{ev}}) \arrow[two heads]{d}{\cong} \\
\MT(\mathsf{D}_2^{\mathrm{ev}}) \arrow[two heads]{rr}{\cong} && \MT(V)
\end{tikzcd}
\]
it is apparent that $\ker(q_1^{\mathrm{ev}})=\ker(q_2^{\mathrm{ev}})$.
Lemma \ref{rmk:oddHS} now implies that $\ker(q_1)=\ker(q_2)$ in $\MT(\mathsf{E})$. But this precisely means that the subcategories $\mathsf{D}_1$ and $\mathsf{D}_2$ of $\mathsf{E}$ coincide, and we conclude that we have $\mathsf{D}_1=\mathsf{E}=\mathsf{D}_2$.
\end{proof}

Thanks to Andr\'e's Theorem \ref{mhcabelian}, we can lift Theorem \ref{thm:kuga-satake} to the category of abelian Andr\'e motives.
\begin{corollary}[Motivic Kuga--Satake category]\label{cor:motivicKS}
    If $M\in \AM$ is an abelian Andr\'e motive whose Hodge realization is of $\mathrm{K3}$-type, then
    there exists a unique tannakian subcategory $\mathsf{KS}(M)$ of $\AM$ such that $$\langle M \rangle_{\AM}=\mathsf{KS}(M)^{\mathrm{ev}}\subsetneq\mathsf{KS}(M).$$ Moreover, $\mathsf{KS}(M)=\langle \mathcal{H}^1(A) \rangle_{\AM} $ for any Kuga--Satake variety $A$ (Definition \ref{def:Kuga-Satake}) for the Hodge structure $r(M)$.
\end{corollary}
The above discussion leads us naturally to the following question about relations among different Kuga--Satake abelian varieties. 

\emph{Question:} Let $A$ and $B$ be abelian varieties such that $\langle H^1(A) \rangle = \langle H^1(B) \rangle$ in $\mathrm{HS}^{\mathrm{pol}}_{\Q}$. Does this imply the existence of integers $k, l$, such that $A$ is dominated by $B^k$ and $B$ is dominated by $A^l$?

\bibliographystyle{amsplain}
\bibliography{bibliography}{}

\providecommand{\bysame}{\leavevmode\hbox to3em{\hrulefill}\thinspace}
\providecommand{\MR}{\relax\ifhmode\unskip\space\fi MR }
\providecommand{\MRhref}[2]{%
  \href{http://www.ams.org/mathscinet-getitem?mr=#1}{#2}
}
\providecommand{\href}[2]{#2}
\begin{thebibliography}{10}

\bibitem{AT14}
Nicolas Addington and Richard Thomas, \emph{Hodge theory and derived categories
  of cubic fourfolds}, Duke. Math. J. \textbf{163} (2014), 1885--1927,
  \url{https://doi.org/10.1215/00127094-2738639}.

\bibitem{AHLH}
Jarod Alper, Daniel Halpern-Leistner, and Heinloth Jochen, \emph{Existence of
  moduli spaces for algebraic stacks}, arXiv preprint (2018),
  \url{https://arxiv.org/abs/1812.01128}.

\bibitem{Andre1996}
Yves Andr{\'e}, \emph{On the {S}hafarevich and {T}ate conjectures for
  hyper-k{\"a}hler varieties}, Mathematische Annalen \textbf{305} (1996),
  no.~1, 205--248, \url{https://doi.org/10.1007/BF01444219}.

\bibitem{andre1996Motives}
\bysame, \emph{Pour une th{\'e}orie inconditionnelle des motifs}, Publications
  Math{\'e}matiques de l'IH{\'E}S \textbf{83} (1996), no.~1, 5--49,
  \url{http://www.numdam.org/item?id=PMIHES_1996__83__5_0}.

\bibitem{andre}
\bysame, \emph{Une introduction aux motifs (motifs purs, motifs mixtes,
  p\'{e}riodes)}, Panoramas et Synth\`eses, no.~17, Soci\'{e}t\'{e}
  {M}ath\'{e}matique de {F}rance, {P}aris, 2004.

\bibitem{MR2271985}
Donu Arapura, \emph{Motivation for {H}odge cycles}, Adv. Math. \textbf{207}
  (2006), no.~2, 762--781.

\bibitem{Ara13}
\bysame, \emph{An abelian category of motivic sheaves}, Adv. Math. \textbf{233}
  (2013), 135--195, \url{https://doi.org/10.1016/j.aim.2012.10.004}.

\bibitem{BLMNPS}
Arend Bayer, Mart\'{i} Lahoz, Emanuele Macr\`{i}, Howard Nuer, Alexander Perry,
  and Paolo Stellari, \emph{Stability conditions in families}, arXiv preprint
  (2019), \url{https://arxiv.org/abs/1902.08184}.

\bibitem{BLMS}
Arend Bayer, Mart\'{i} Lahoz, Emanuele Macr\`{i}, Paolo Stellari, and Xiaolei
  Zhao, \emph{Stability conditions on {K}uznetsov components}, arXiv preprint
  (2017), \url{https://arxiv.org/abs/1703.10839}.

\bibitem{BM14a}
Arend Bayer and Emanuele Macr\`{i}, \emph{{M}{M}{P} for moduli of sheaves on
  {K}3s via wall-crossing: nef and movable cones, {L}agrangian fibrations},
  Invent. Math. \textbf{198} (2014), no.~3, 505--590,
  \url{https://doi.org/10.1007/s00222-014-0501-8}.

\bibitem{BM14b}
\bysame, \emph{Projectivity and birational geometry of {B}ridgeland moduli
  spaces}, J. Amer. Math. Soc. \textbf{27} (2014), no.~3, 707--752,
  \url{https://doi.org/10.1090/S0894-0347-2014-00790-6}.

\bibitem{beauville1983varietes}
Arnaud Beauville, \emph{Vari{\'e}t{\'e}s {K}{\"{a}}hl\'eriennes dont la
  premi{\`e}re classe de {C}hern est nulle}, J. Differential Geom. \textbf{18}
  (1983), no.~4, 755--782,
  \url{http://projecteuclid.org/euclid.jdg/1214438181}.

\bibitem{Bea95}
\bysame, \emph{Sur la cohomologie de certains espaces de modules de fibr\'{e}s
  vectoriels}, Geometry and Analysis, Bombay, 1992, Tata Institute of
  Fundamental Research, 1995, pp.~37--40.

\bibitem{BD85}
Arnaud Beauville and Ron Donagi, \emph{La vari\'et\'e des droites d'une
  hypersurface cubique de dimension 4}, C. R. Acad. Sci. Paris S\'{e}r. I Math.
  \textbf{301} (1985), no.~14, 703--706.

\bibitem{Bon10}
Mikhail~V. Bondarko, \emph{Weight structures vs. t-structures; weight
  filtrations, spectral sequences, and complexes (for motives and in general)},
  J. K-theory \textbf{6} (2010), 387--504,
  \url{https://doi.org/10.1017/is010012005jkt083}.

\bibitem{Bri07}
Tom Bridgeland, \emph{Stability conditions on triangulated categories}, Ann. of
  Math. \textbf{166} (2007), no.~2, 317--345,
  \url{https://doi.org/10.4007/annals.2007.166.317}.

\bibitem{Bri08}
\bysame, \emph{Stability conditions on {K}3 surfaces}, Duke Math. J.
  \textbf{141} (2008), no.~2, 241--291,
  \url{https://doi.org/10.1215/S0012-7094-08-14122-5}.

\bibitem{Bue18}
Tim-Henrik B{\"{u}}lles, \emph{Motives of moduli spaces on {K}3 surfaces and of
  special cubic fourfolds}, Manuscripta Math. (2018),
  \url{https://doi.org/10.1007/s00229-018-1086-0}.

\bibitem{CCL18}
Chiara Camere, Alberto Cattaneo, and Robert Laterveer, \emph{On the {C}how ring
  of certain {L}ehn-{L}ehn-{S}orger-van {S}traten eightfolds}, arXiv preprint
  (2018), \url{https://arxiv.org/abs/1812.11554v1}.

\bibitem{CM13}
Fran\c{c}ois Charles and Eyal Markman, \emph{The standard conjectures for
  holomorphic symplectic varieties deformation equivalent to {H}ilbert schemes
  of {K}3 surfaces}, Compos. Math. \textbf{149} (2013), no.~3, 481--494,
  \url{https://doi.org/10.1112/S0010437X12000607}.

\bibitem{commelin2019}
Johan Commelin, \emph{The {M}umford--{T}ate conjecture for products of abelian
  varieties}, Algebr. Geom. \textbf{6} (2019), no.~6, 650--677,
  \url{http://content.algebraicgeometry.nl/2019-6/2019-6-028.pdf}.

\bibitem{deCataldoMigliorini2002}
Mark de~Cataldo and Luca Migliorini, \emph{The {C}how groups and the motive of
  the {H}ilbert scheme of points on a surface}, J. Algebra \textbf{251} (2002),
  no.~2, 824--848, \url{https://doi.org/10.1006/jabr.2001.9105}.

\bibitem{dCM04}
\bysame, \emph{The {C}how motive of semismall resolutions}, Math. Res. Lett.
  \textbf{11} (2004), no.~2-3, 151--170,
  \url{https://dx.doi.org/10.4310/MRL.2004.v11.n2.a2}.

\bibitem{dCRS}
Mark de~Cataldo, Antonio Rapagnetta, and Giulia Sacc\`a, \emph{The {H}odge
  numbers of {O}'{G}rady 10 via {N}g{\^o} strings}, arXiv preprint (2019),
  \url{https://arxiv.org/abs/1905.03217}.

\bibitem{DK19}
Olivier Debarre and Alexander Kuznetsov, \emph{{G}ushel-{M}ukai varieties:
  linear spaces and periods}, Kyoto J. Math. \textbf{40} (2019), 5--57,
  \url{https://doi.org/10.1215/21562261-2019-0030}.

\bibitem{DV10}
Olivier Debarre and Claire Voisin, \emph{Hyper-{K}\"{a}hler fourfolds and
  {G}rassmann geometry}, J. Reine Angew. Math \textbf{649} (2010), 63--87,
  \url{https://doi.org/10.1515/CRELLE.2010.089}.

\bibitem{Ban01}
Sebastian del Ba\~{n}o, \emph{On the {C}how motive of some moduli spaces}, J.
  reine angew. Math. \textbf{532} (2001), 105--132,
  \url{https://doi.org/10.1515/crll.2001.019}.

\bibitem{deligne1971conjecture}
Pierre Deligne, \emph{La conjecture de {W}eil pour les surfaces {K}3},
  Inventiones mathematicae \textbf{15} (1971), no.~3, 206--226.

\bibitem{deligne1971theorie}
\bysame, \emph{Th{\'e}orie de {H}odge: {I}{I}}, Publications Math{\'e}matiques
  de l'IH{\'E}S \textbf{40} (1971), 5--57.

\bibitem{deligne1982hodge}
\bysame, \emph{Hodge cycles on abelian varieties}, Hodge cycles, motives, and
  Shimura varieties, Lecture Notes in Mathematics, vol. 900, Springer, 1982,
  pp.~9--100.

\bibitem{MF19}
Enrico Fatighenti and Giovanni Mongardi, \emph{Fano varieties of {K}3 type and
  {IHS} manifolds}, arXiv preprint (2019),
  \url{https://arxiv.org/abs/1904.05679v1}.

\bibitem{floccari2019}
Salvatore Floccari, \emph{On the {M}umford-{T}ate conjecture for
  hyper-{K}\"{a}hler varieties}, arXiv preprint (2019),
  \url{https://arxiv.org/abs/1904.06238}.

\bibitem{Fon15}
Anton~V. Fonarev, \emph{On the {K}uznetsov-{P}olishchuk conjecture}, Proc.
  Steklov Inst. Math. \textbf{290} (2015), 11--25,
  \url{https://doi.org/10.1134/S0081543815060024}.

\bibitem{FTV19}
Lie Fu, Zhiyu Tian, and Charles Vial, \emph{Motivic hyper-{K}\"ahler resolution
  conjecture {I}: {G}eneralized {K}ummer varieties}, Geometry \& Topology
  \textbf{23} (2019), 427--492, \url{https://doi.org/10.2140/gt.2019.23.427}.

\bibitem{GKLR}
Mark Green, Yoon-Joo Kim, Radu Laza, and Colleen Robles, \emph{The {LLV}
  decomposition of hyper-{K}\"ahler cohomology}, arXiv preprint (2019),
  \url{https://arxiv.org/abs/1906.03432}.

\bibitem{MR0268189}
Alexander Grothendieck, \emph{Standard conjectures on algebraic cycles},
  Algebraic {G}eometry ({I}nternat. {C}olloq., {T}ata {I}nst. {F}und. {R}es.,
  {B}ombay, 1968), Oxford Univ. Press, London, 1969, pp.~193--199.

\bibitem{Gus83}
N.~P. Gushel, \emph{On {F}ano varieties of genus 6}, Izv. Math. \textbf{21}
  (1983), no.~3, 445--459.

\bibitem{Har16}
Daniel Harrer, \emph{Comparison of the categories of motives defined by
  {V}oevodsky and {N}ori}, Ph.D. thesis, Albert-Ludwigs-Universit\"{a}t
  Freiburg, 2016, \url{https://arxiv.org/abs/1609.05516}.

\bibitem{HPL18}
Victoria Hoskins and Simon Pepin-Lehalleur, \emph{A formula for the {V}oevodsky
  motive of the moduli stack of vector bundles on a curve}, arXiv preprint
  (2018), \url{https://arxiv.org/abs/1809.02150}.

\bibitem{HPL19}
\bysame, \emph{On the {V}oevodsky motive of the moduli space of {H}iggs bundles
  on a curve}, arXiv preprint (2019), \url{https://arxiv.org/abs/1910.04440}.

\bibitem{HKMS}
Annette Huber-Klawitter and Stefan M\"{u}ller-Stach, \emph{Periods and {N}ori
  motives}, Ergebnisse der Mathematik und ihrer Grenzgebiete. 3. Folge. A
  Series of Modern Surveys in Mathematics, vol.~65, Springer, 2017.

\bibitem{Huy99}
Daniel Huybrechts, \emph{Compact hyper-{K}\"ahler manifolds: basic results},
  Invent. Math. \textbf{135} (1999), no.~1, 63--113,
  \url{https://doi.org/10.1007/s002220050280}.

\bibitem{huyK3}
\bysame, \emph{Lectures on {K}3 surfaces}, vol. 158, Cambridge University
  Press, 2016.

\bibitem{Jan92}
Uwe Jannsen, \emph{Motives, numerical equivalence, and semi-simplicity},
  Invent. Math. \textbf{107} (1992), no.~3, 447--452,
  \url{https://doi.org/10.1007/BF01231898}.

\bibitem{KLS06}
Dmitry Kaledin, Manfred Lehn, and Christoph Sorger, \emph{Singular symplectic
  moduli spaces}, Invent. Math. \textbf{164} (2006), no.~3, 591--614,
  \url{https://doi.org/10.1007/s00222-005-0484-6}.

\bibitem{MR1265519}
Steven~L. Kleiman, \emph{The standard conjectures}, Motives ({S}eattle, {WA},
  1991), Proc. Sympos. Pure Math., vol.~55, Amer. Math. Soc., Providence, RI,
  1994, pp.~3--20.

\bibitem{KSV2017}
Nikon Kurnosov, Andrey Soldatenkov, and Misha Verbitsky, \emph{Kuga-{S}atake
  construction and cohomology of hyperk\"{a}hler manifolds}, Adv. Math.
  \textbf{351} (2019), 275--295,
  \url{https://doi.org/10.1016/j.aim.2019.04.060}.

\bibitem{Kuz10}
Alexander Kuznetsov, \emph{Derived categories of cubic fourfolds}, Progress in
  Mathematics, vol. 282, pp.~219--243, Birkh\"{a}user Boston, 2010,
  \url{https://doi.org/10.1007/978-0-8176-4934-0_9}.

\bibitem{KP16}
Alexander Kuznetsov and Alexander Perry, \emph{Derived categories of
  {G}ushel-{M}ukai varieties}, Compositio Mathematica \textbf{154} (2018),
  1362--1406, \url{https://doi.org/10.1112/S0010437X18007091}.

\bibitem{Lat17}
Robert Laterveer, \emph{A remark on the motive of the {F}ano variety of lines
  of a cubic}, Ann. Math. Qu\'{e}. \textbf{41} (2017), no.~1, 141--154,
  \url{https://doi.org/10.1007/s40316-016-0070-x}.

\bibitem{LLSvS}
Christian Lehn, Manfred Lehn, Christoph Sorger, and Duco van Straten,
  \emph{Twisted cubics on cubic fourfolds}, J. Reine Angew. Math. \textbf{731}
  (2017), 87--128, \url{https://doi.org/10.1515/crelle-2014-0144}.

\bibitem{LS06}
Manfred Lehn and Christoph Sorger, \emph{La singularit\'{e} de {O}'{G}rady}, J.
  Algebraic Geom. \textbf{15} (2006), 753--770,
  \url{https://doi.org/10.1090/S1056-3911-06-00437-1}.

\bibitem{LPZ18}
Chunyi Li, Laura Pertusi, and Xiaolei Zhao, \emph{Twisted cubics on cubic
  fourfolds and stability conditions}, arXiv preprint (2018),
  \url{https://arxiv.org/abs/1802.01134}.

\bibitem{LPZ19}
\bysame, \emph{Elliptic quintics on cubic fourfolds, {O}'{G}rady 10, and
  {L}agrangian fibrations}, In preparation (2019).

\bibitem{looijenga1997lie}
Eduard Looijenga and Valery~A. Lunts, \emph{A {L}ie algebra attached to a
  projective variety}, Invent. Math. \textbf{129} (1997), no.~2, 361--412,
  \url{https://doi.org/10.1007/s002220050166}.

\bibitem{MS18}
Emanuele Macr\`{i} and Paolo Stellari, \emph{Lectures on {N}on-commutative {K}3
  {S}urfaces, {B}ridgeland {S}tability, and {M}oduli {S}paces}, {B}irational
  {G}eometry of {H}ypersurfaces, Lecture Notes of the Unione Matematica
  Italiana, vol.~26, Springer, 2019,
  \url{https://doi.org/10.1007/978-3-030-18638-8_6}.

\bibitem{Man68}
Yuri~I. Manin, \emph{Correspondences, motifs and monoidal transformations},
  Mat. Sb. (N.S.) \textbf{77} (1968), no.~119, 475--507,
  \url{http://dx.doi.org/10.1070/SM1968v006n04ABEH001070}.

\bibitem{MZ17}
Alina Marian and Xiaolei Zhao, \emph{On the group of zero-cycles of holomorphic
  symplectic varieties}, \'Epijournal de G\'eom\'etrie Alg\'ebrique, Volume 4
  (2020), Article no.3 (2020), \url{https://epiga.episciences.org/6197}.

\bibitem{Mar07}
Eyal Markman, \emph{Integral generators for the cohomology ring of moduli
  spaces of sheaves over {P}oisson surfaces}, Advances in Mathematics
  \textbf{208} (2007), 622--646,
  \url{https://doi.org/10.1016/j.aim.2006.03.006}.

\bibitem{Mar02}
\bysame, \emph{Generators of the cohomology ring of moduli spaces of sheaves on
  symplectic surfaces}, J. Reine Angew. Math. \textbf{544} (2012), 61--82,
  \url{https://doi.org/10.1515/crll.2002.028}.

\bibitem{MRS18}
Giovanni Mongardi, Antonio Rapagnetta, and Giulia Sacc\`{a}, \emph{The {H}odge
  diamond of {O}'{G}rady’s six-dimensional example}, Compositio Mathematica
  \textbf{154} (2018), no.~5, 984--1013,
  \url{https://doi.org/10.1112/S0010437X1700803X}.

\bibitem{moonen17}
Ben Moonen, \emph{Families of {M}otives and the {M}umford-{T}ate {C}onjecture},
  Milan J. of Math. \textbf{85} (2017), no.~2, 257--307,
  \url{https://doi.org/10.1007/s00032-017-0273-x}.

\bibitem{moonen2017}
\bysame, \emph{On the {T}ate and {M}umford--{T}ate conjectures in codimension
  {$1$} for varieties with {$h^{2,0}=1$}}, Duke Math J. \textbf{166} (2017),
  no.~4, 739--799, \url{https://dx.doi.org/10.1215/00127094-3774386}.

\bibitem{Muk84}
Shigeru Mukai, \emph{Symplectic structure of the moduli space of sheaves on an
  abelian or {$K3$} surface}, Invent. Math. \textbf{77} (1984), no.~1,
  101--116, \url{https://doi.org/10.1007/BF01389137}.

\bibitem{Muk89}
\bysame, \emph{Biregular classification of {F}ano {$3$}-folds and {F}ano
  manifolds of coindex {$3$}}, Proc. Nat. Acad. Sci. U.S.A. \textbf{86} (1989),
  no.~9, 3000--3002, \url{https://doi.org/10.1073/pnas.86.9.3000}.

\bibitem{Mum66}
David Mumford, \emph{Families of abelian varieties}, Algebraic {G}roups and
  {D}iscontinuous {S}ubgroups ({P}roc. {S}ympos. {P}ure {M}ath., {B}oulder,
  {C}olo., 1965), Amer. Math. Soc., Providence, R.I., 1966, pp.~347--351.

\bibitem{Mur93}
Jacob Murre, \emph{On a conjectural filtration on the {C}how groups of an
  algebraic variety. {I}. the general conjectures and some examples}, Indag.
  Math. (N.S.) \textbf{4} (1993), 177--188,
  \url{https://doi.org/10.1016/0019-3577(93)90038-Z}.

\bibitem{O'G99}
Kieran O'Grady, \emph{Desingularized moduli spaces of sheaves on a {K}3}, J.
  Reine Angew. Math. \textbf{512} (1999), 49--117,
  \url{https://doi.org/10.1515/crll.1999.056}.

\bibitem{O'G03}
\bysame, \emph{A new six-dimensional irreducible symplectic variety}, J.
  Algebraic Geom. \textbf{12} (2003), no.~3, 435--505,
  \url{https://doi.org/10.1090/S1056-3911-03-00323-0}.

\bibitem{O'G97}
Kieran~G. O'Grady, \emph{The weight-two {H}odge structure of moduli spaces of
  sheaves on a {$K3$} surface}, J. Algebraic Geom. \textbf{6} (1997), no.~4,
  599--644.

\bibitem{Ped12}
Claudio Pedrini, \emph{On the finite dimensionality of a {K}3 surface},
  Manuscripta Math. \textbf{138} (2012), no.~1-2, 59--72,
  \url{https://doi/org/10.1007/s00229-011-0483-4}.

\bibitem{PR13}
Arvid Perego and Antonio Rapagnetta, \emph{Deformation of the {O}'{G}rady
  moduli spaces}, J. Reine Angew. Math. \textbf{678} (2013), 1--34,
  \url{https://doi.org/10.1515/CRELLE.2011.191}.

\bibitem{PPZ19}
Alexander Perry, Laura Pertusi, and Xiaolei Zhao, \emph{Stability conditions
  and moduli spaces for {K}uznetsov components of {G}ushel-{M}ukai varieties},
  arXiv preprint (2019), \url{https://arxiv.org/abs/1912.06935}.

\bibitem{Pin98}
Richard Pink, \emph{{$l$}-adic algebraic monodromy groups, cocharacters, and
  the {M}umford-{T}ate conjecture}, J. Reine Angew. Math. \textbf{495} (1998),
  187--237, \url{https://doi.org/10.1515/crll.1998.018}.

\bibitem{Rap07}
Antonio Rapagnetta, \emph{Topological invariants of {O}'{G}rady's six
  dimensional irreducible symplectic variety}, Mathematische Zeitschrift
  \textbf{256} (2007), 1--34, \url{https://doi.org/10.1007/s00209-006-0022-2}.

\bibitem{Rap08}
\bysame, \emph{On the {B}eauville form of the known irreducible symplectic
  varieties}, Mathematische Annalen \textbf{340} (2008), 77--95,
  \url{https://doi.org/10.1007/s00208-007-0139-6}.

\bibitem{Rie14}
Ulrike Rie\ss, \emph{On the {C}how ring of birational irreducible symplectic
  varieties}, Manuscripta Math. \textbf{145} (2014), no.~3-4, 473--501,
  \url{https://doi.org/10.1007/s00229-014-0698-2}.

\bibitem{Sch12}
Ulrich Schlickewei, \emph{On the {A}ndr\'{e} motive of certain irreducible
  symplectic varieties}, Geom. Dedicata \textbf{156} (2012), no.~1, 141--149,
  \url{https://doi.org/10.1007/s10711-011-9594-z}.

\bibitem{SV16}
Mingmin Shen and Charles Vial, \emph{The {F}ourier transform for certain
  hyper-{K}\"ahler fourfolds}, vol. 240, Mem. Amer. Math. Soc., no. 1139,
  American Mathematical Society, 2016,
  \url{https://bookstore.ams.org/memo-240-1139}.

\bibitem{soldatenkov19}
Andrey Soldatenkov, \emph{Deformation principle and {A}ndr\'e motives of
  projective hyper-{K}\"{a}hler manifolds}, arXiv preprint (2019),
  \url{https://arxiv.org/abs/1904.11320}.

\bibitem{solda19}
\bysame, \emph{On the {H}odge structure of compact hyper-{K}\"{a}hler
  manifolds}, to appear in Math. Res. Lett. (2019),
  \url{https://arxiv.org/abs/1905.07793}.

\bibitem{Tab15}
Gon\c{c}alo Tabuada, \emph{Noncommutative motives}, University Lecture Series,
  vol.~63, American Mathematical Society, Providence, RI, 2015,
  \url{http://www.ams.org/books/ulect/063}.

\bibitem{Tan83}
S.~G. Tankeev, \emph{Cycles of abelian varieties of prime dimension over finite
  and number fields}, Izv. Akad. Nauk SSSR Ser. Mat. \textbf{47} (1983), no.~2,
  356--365.

\bibitem{verbitsky1995mirror}
Misha Verbitsky, \emph{Mirror symmetry for hyper-{K}{\"a}hler manifolds},
  AMS/IP Stud. Adv. Math. \textbf{10} (1995), 115--156.

\bibitem{verbitsky1996cohomology}
\bysame, \emph{Cohomology of compact hyper-{K}{\"a}hler manifolds and its
  applications}, Geometric \& Functional Analysis GAFA \textbf{6} (1996),
  no.~4, 601--611, \url{https://doi.org/10.1007/BF02247112}.

\bibitem{Via13}
Charles Vial, \emph{Algebraic cycles and fibrations}, Doc. Math. \textbf{18}
  (2013), 1521--1553,
  \url{https://www.math.uni-bielefeld.de/documenta/vol-18/47.html}.

\bibitem{VialCrelle}
\bysame, \emph{On the motive of some hyper{K}\"{a}hler varieties}, J. Reine
  Angew. Math. \textbf{725} (2017), 235--247.

\bibitem{Via17}
\bysame, \emph{Remarks on motives of abelian type}, Tohoku Math. J. (2)
  \textbf{69} (2017), no.~2, 195--220,
  \url{https://doi.org/10.2748/tmj/1498269623}.

\bibitem{Voe00}
Vladimir Voevodsky, \emph{Triangulated categories of motives over a field},
  Cycles, transfers, and motivic homology theories, vol. 143, Ann. of Math.
  Stud., no.~2, Princeton Univ. Press, 2000, pp.~188--238.

\bibitem{Wil15}
J\"{o}rg Wildeshaus, \emph{On the interior motive of certain {S}himura
  varieties: the case of {P}icard surfaces}, Manuscripta Math. \textbf{148}
  (2015), no.~3-4, 351--377, \url{https://doi.org/10.1007/s00229-015-0747-5}.

\bibitem{Xu18}
Ze~Xu, \emph{Algebraic cycles on a generalized {K}ummer variety}, Int. Math.
  Res. Not. \textbf{3} (2018), 932--948,
  \url{https://doi.org/10.1093/imrn/rnw266}.

\bibitem{YY14}
Shintarou Yanagida and K\={o}ta Yoshioka, \emph{Bridgeland's stabilities on
  abelian surfaces}, Math. Z. \textbf{276} (2014), no.~1-2, 571--610,
  \url{https://doi.org/10.1007/s00209-013-1214-1}.

\bibitem{Yos01}
K\={o}ta Yoshioka, \emph{Moduli spaces of stable sheaves on abelian surfaces},
  Math. Ann. \textbf{321} (2001), no.~4, 817--884,
  \url{https://doi.org/10.1007/s002080100255}.

\end{thebibliography}

\end{document}